\documentclass[11pt]{ut-thesis}
\usepackage{amsmath}
\usepackage{amsfonts, amssymb, amsthm, amscd, graphicx, tikz}
\pdfoutput=1
\usetikzlibrary{positioning}
\def\A{\mathbb{A}}
\def\C{\mathbb{C}}

\def\L{\mathcal{L}}

\def\P{\mathbb{P}}

\def\R{\mathbb{R}}
\def\T{\mathbb{T}}

\def\a{\mathsf{a}}
\def\c{\mathfrak{c}}
\def\g{\mathfrak{g}}
\def\d{\mathfrak{d}}
\def\dr{{\mbox{d}}}
\def\h{\mathfrak{h}}
\def\k{\mathfrak{k}}
\def\l{\mathfrak{l}}
\def\p{\mathfrak{p}}
\def\q{\mathfrak{q}}
\def\r{\mathfrak{r}}
\def\s{\mathsf{s}}
\def\t{\mathsf{t}}
\def\u{\mathsf{u}}

\def\ph{p_{\h}}
\def\Ad{\mathrm{Ad}}
\def\ad{\mathrm{ad}}
\def\da{\dasharrow}
\def\act{\circlearrowright}
\def\tp{\tilde{\mathfrak{p}}}
\def\lt{\tilde{\ell}}
\def\tl{\tilde{\l}}

\def\Ad{\mathrm{Ad}}
\def\<{\langle}
\def\>{\rangle}

\def\rd{\mathrm{red}}
\def\rk{\mathrm{rk}}
\def\dim{\mathrm{dim}}
\def\gp{\rightrightarrows}

\def\md{\,\,\mathrm{mod}}

\def\ann{\mathrm{ann}}
\def\md{\,\,\mathrm{mod}\,}
\def\ph{\mathrm{pr}_{\h}}
\def\pg{\mathrm{pr}_{\g}}
\def\pr{\mathrm{pr}_{\r}}

\newcommand{\Cour}[1]      {[\![#1]\!]}
\theoremstyle{definition}
\newtheorem{thm}{Theorem}[section]
\newtheorem{lem}[thm]{Lemma}
\newtheorem{prop}[thm]{Proposition}
\newtheorem{cor}[thm]{Corollary}
\newtheorem{dfn}{Definition}[section]

\newtheorem*{thmn}{Theorem}
\theoremstyle{remark}
\newtheorem{example}{Example}[section]

\newtheorem{rmrk}{Remark}[section]
\tikzset{node distance=2cm, auto}

\degree{Doctor of Philosophy}
\department{Mathematics}
\gradyear{2014}
\author{Patrick J. Robinson}
\title{The Classification of Dirac Homogeneous Spaces}

\begin{document}


\maketitle


\begin{abstract}
A well known result of Drinfeld classifies Poisson Lie groups $(H,\Pi)$ in terms of Lie algebraic data in the form of Manin triples $(\d,\g,\h)$; he also classified compatible Poisson structures on $H$-homogeneous spaces
$H/K$ in terms of Lagrangian subalgebras $\l\subset\d$ with $\l\cap\h=\k=$Lie$(K)$.  Using the language of Courant algebroids and groupoids, Li-Bland and Meinrenken formalized the notion of \emph{Dirac Lie groups} and classified them in terms of so-called ``$H$-equivariant Dirac Manin triples" $(\d,\g,\h)_\beta$; this generalizes the first result of Drinfeld, as each Poisson Lie group gives a unique Dirac Lie group structure.\\
In this thesis, we consider a notion of homogeneous space for Dirac Lie groups, and classify them in terms of $K$-invariant coisotropic subalgebras $\c\subset\d$, with $\c\cap\h=\k$.  The relation between Poisson and Dirac morphisms makes Drinfeld's second result a special case of this classification.

\end{abstract}




\begin{acknowledgements}
I certainly would not have come this far without the love and support of my parents, who have always taught me to be inquisitive, to think for myself, and to love learning.  They have encouraged me in all that I do.\\
I am eternally grateful to my supervisor, Eckhard Meinrenken, for his patience, his willingness to spend hours discussing ideas (and answering my inane questions), his attempts to keep me on track by \emph{always} telling me I was behind, and for everything he has taught me.  I truly could not have asked for a better doctoral supervisor.\\
I would like to thank all of my colleagues with whom I have had interesting discussions, and have continually helped me through various mathematical issues:  Tyler Holden, Peter Crooks, David Li-Bland, Dan Rowe, Ivan Khatchatourian, Parker Glynn-Adey, and Madeleine Jotz.  I am indebted to all of the teachers and professors who have deeply inspired and encouraged my love of mathematics: Mr. Marzewski, Mrs. Kopach, Mr. Evst, Dr. Yeung, Kiumars Kaveh, Peter Rosenthal, Alfonso Gracia-Saz, and Joel Kamnitzer.  I'm also very thankful for my NSERC USRA supervisor, Dr. Neil Barnaby, who rescued me from the doldrums of undergrad and introduced me to what research was really like; I will always carry the tachyons with me.\\
I would certainly like to thank the staff of the Mathematics Department, without whom we would surely all perish.  I have received immense help from Ida Bulat, Jemima Merisca, Donna Birch, Pamela Brittain, Patrina Seepersaud, and Marie Bachtis.
Lastly, I would like to thank Stephanie Jones, for her support, love, and for having helped keep me sane during this mathematical adventure.
\end{acknowledgements}

\tableofcontents







\chapter{Introduction}
\section{Poisson Geometry}
Poisson developed the ``Poisson bracket" in 1809 \cite{poisson} as a means of solving problems in Lagrangian mechanics; the significance of the bracket was later recognized by Jacobi, who investigated its algebraic properties \cite{jacobi}. In his work on transformation groups, Lie was able to make a connection between the Poisson bracket and Lie theory when he discovered a canonical Poisson structure on the dual of any Lie algebra \cite{lie}. Poisson geometry experienced a resurgence beginning in the 1970s due to such mathematicians as Lichnerowicz \cite{Lich}, Weinstein \cite{wein1}\cite{wein2}, Koszul \cite{Koszul}, Kirillov \cite{kirillov}, and others as connections to varying areas of mathematics were discovered. Poisson geometry plays an important role in noncommutative geometry, representation theory, singularity theory, integrable systems, and cluster algebras to name a few (see \cite{cluster}, \cite{singularity}, \cite{noncom}, etc).\\
In modern mathematical language, a \emph{Poisson Manifold} is a manifold $M$, together with a Lie algebra bracket $\{,\}$ on the smooth functions of $M$ such that for any $f\in C^\infty(M)$, the operator
    \begin{equation}\label{hvector} X_f = \{f,\cdot\} \end{equation}
is a derivation.  This means $X_f$ is a vector field, called the \emph{Hamiltonian vector field} associated to $f$.



With Poisson manifolds comes the notion of a \emph{Poisson morphism}: a smooth map $\varphi:\,M_1\to M_2$ between Poisson manifolds whose pullback is compatible with the Poisson bracket, i.e.:
    \begin{equation}\{\varphi^*f,\varphi^*g\}_1 = \varphi^*\{f,g\}_2\qquad \forall\, f,g\in C^\infty(M_2).\end{equation}
\emph{Poisson Lie groups} were defined by Drinfeld in \cite{Drinfeldh}, with significant subsequent developments by Weinstein, Lu \cite{luth} \cite{luwein1}, Semenov-Tian-Shansky \cite{STS}, and others.  The original motivation came from the Drinfeld-Faddeev formulation of quantum groups \cite{Drinfeldq}, which arose as solutions to the Quantum Yang-Baxter equation, an important equation in quantum field theory; Drinfeld and Jimbo recognized that these quantum groups were simply Hopf algebras. Taking the ``classical limit" of a quantum group yields a Poisson Lie group (see \cite{Drinfeldq}, \cite{quantum}).  From a purely mathematical perspective, a Poisson Lie group $(H,\{,\})$ is a Lie group with Poisson structure such that the multiplication map $H\times H\to H$ is a Poisson morphism, with $H\times H$ given the product Poisson structure.  A \emph{Poisson action} is defined as an action of a Poisson Lie group $(H,\{,\}_H)$ on a Poisson manifold $(M,\{,\}_M)$ such that the action map
    \begin{equation} (H,\{,\}_H)\times(M,\{,\}_M)\longrightarrow (M,\{,\}_M) \end{equation}
is a Poisson morphism \cite{STS}. In general, the Poisson bracket of $M$ is not preserved (see Proposition 5 in \cite{STS}).
If a given Poisson action is transitive, then $M$ is referred to as a \emph{Poisson homogeneous space}; a choice of basepoint $m\in M$ identifies $M\cong H/K$, with $K$ being the stabiliser of $m$.
\section{Classifications of Drinfeld}
A \emph{Manin triple} is a triple of Lie algebras $(\d,\g,\h)$, where $\d$ is a quadratic Lie algebra, and $\g,\h$ are transverse Lagrangian subalgebras.  Drinfeld showed in \cite{Drinfeldh} that Poisson Lie groups $(H,\Pi_H)$ are in 1-1 correspondence with Manin triples $(\d,\g,\h)$, with $\h$ the Lie algebra of $H$, and an action of $H$ on $\d$ by Lie algebra automorphisms integrating the adjoint action of $\h$ on $\d$, and restricting to the adjoint action of $H$ on $\d$.  Since $\h$ and $\g$ are each Lagrangian, we have the association $\g\cong\h^*$ in the following manner:
    \begin{equation} \h^* \cong (\d/\g)^* \cong \g^\perp = \g. \end{equation}
The dual space $\h^*$ inherits a Lie bracket from the Poisson structure:  since $\{f,g\}(e)=0$ for any $f,g\in C^\infty(H)$, we can obtain a linear Poisson structure on $\h$, which is equivalent to a Lie algebra structure on $\h^*$ (\cite{luwein1}, \cite{Drinfeldh}; see Example \ref{PoisLiestru}). 
The metric on $\d$ is given by the natural pairing, and the bracket is the unique one  such that $\g$ and $\h$ are Lagrangian and the metric is
invariant.  Hence, given a Lie group $H$, the Lie algebra structure of $\h$ as well as the vector space structure of $\d=\g\oplus\h$ is automatically determined; it is the Poisson structure on $H$ which determines the Lie algebra structure of $\g$, and by extension, $\d$.\\
Drinfeld later considered Poisson homogeneous spaces in \cite{Drinfeldp}.  He concluded that given a Poisson Lie group $(H,\{,\})$ and an $H$-homogeneous space $H/K$, the Poisson structures on $H/K$ which make the action into a Poisson morphism are in 1-1 correspondence with Lagrangian, $K$-stable subalgebras $\l\subseteq\d$ such that $\l\cap\h=\k$ (where $\k=\mathrm{Lie}(K)$).\footnote{Drinfeld did not actually include a proof of this theorem in \cite{Drinfeldp}, a 1.5 page paper, but referred to the proof being ``more or less straightforward".  A non-Dirac geometric proof can be found in \cite{evenslu}}
\section{Dirac Geometry, Dirac Lie Groups}
An important property of Poisson manifolds is the ability to reformulate the structure of the Poisson bracket in terms of bivector fields:  a Poisson bracket $\{,\}_M$ on $M$ is equivalent to a bivector field $\Pi\in \wedge^2(\mathrm{T}M)$ such that the induced bracket on $C^\infty(M)$ given by
    \begin{equation} \{f,g\}\,:= \,\<\Pi, \mathrm{d}f\wedge\mathrm{d}g\> \end{equation}
satisfies the Jacobi identity.  This, in turn, is equivalent to the Poisson bivector $\Pi$ satisfying the equation $[\Pi,\Pi]=0$, for the Schouten bracket, the natural generalization of the Lie bracket of vector fields to higher rank tensors (\cite{Schouten},\cite{nij}).  This structure became integral to the study of Dirac geometry.\\
In the 1980s, Courant's work in \cite{courdir}, and (with Weinstein) in \cite{courwein}, put Dirac's theory of constrained Hamiltonian mechanics \cite{dirac} into a geometric perspective.  Dirac developed a generalization of the Poisson bracket, called the \emph{Dirac Bracket}, on a certain subclass of functions (usually referred to as \emph{admissible functions}) on the phase space of a mechanical system with constraints; his analysis allowed such systems to undergo canonical quantization \cite{dirac}.  Courant started by considering the \emph{generalized tangent bundle} $\T M$ of a manifold $M$ (also called the Pontryagin bundle), given by $\T M:=\mathrm{T}M\oplus\mathrm{T}^*M$, and carrying a non-degenerate symmetric bilinear form
    \begin{equation} \< (v_1,\alpha_1),(v_2,\alpha_2)\> = \<v_1,\alpha_2\> + \<v_2,\alpha_1\>. \end{equation}
He also described a bilinear bracket (called the Courant bracket, or sometimes the Dorfman bracket in this context) on the sections of $\T M$ given by
    \begin{equation}\Cour{(v_1,\alpha_1),(v_2,\alpha_2)} = ([v_1,v_2],\,\, \L_{v_1}\alpha_2 - \iota_{v_2}\dr\alpha_1) \end{equation}
where $[v_1,v_2]$ is the Lie bracket on vector fields.  The Courant bracket is not a Lie bracket since it is not anti-symmetric; it does, however, restrict to a Lie bracket on involutive, maximally isotropic (i.e.: Lagrangian) sub-bundles.  Courant referred to such sub-bundles as ``Dirac structures'', and they were at the core of his analysis, with the key examples being the graphs of Poisson bivectors on $M$:
    \begin{equation} \mathrm{Gr}(\Pi) = \{(\iota_\alpha\Pi,\alpha)\,|\,\alpha\in \mathrm{T}^*M\}\subseteq\T M \end{equation}
(in fact, as shown in \cite{courdir} for a given manifold $M$, the Dirac structures in $\T M$ transverse to T$M$ are precisely the graphs of Poisson bivectors on $M$).  One of the important aspects of Courant's work was in the ability to \emph{restrict} Dirac structures to sub-manifolds $S\subseteq M$, effectively imposing constraints.  Courant showed that if one took the restriction of the graph of a Poisson bivector on $M$, the inherited bracket on a certain important subclass of functions coincided with Dirac's bracket from \cite{dirac} (for more details, see \cite{david}, \cite{courwein}, \cite{courdir}).\\
Courant's work on $\T M$ together with its metric and bracket structure became the foundation for the mathematical object now known as a \emph{Courant algebroid}.  His ideas were generalized by Liu, Weinstein, and Xu in \cite{lwx}, defining abstract Courant algebroids to be vector bundles $\A\to M$ with a bundle metric $\<\,,\,\>$, a map $\a:\,\A\to \mathrm{T}M$ called the \emph{anchor}, and a bracket on the sections
    \begin{equation}\Cour{\,,\,}:\Gamma(\A)\times\Gamma(\A) \to \Gamma(\A) \end{equation}
which satisfies certain axioms.  Dirac structures are defined similar to Courant's case to be involutive, Lagrangian sub-bundles $E\subset\A$. Since the Courant bracket again restricts to a Lie bracket on $E$, these Dirac structures turn out to be \emph{Lie algebroids} (see Section \ref{plalg}).  Further work on Courant algebroids was done by Dorfman \cite{dorfman}, Roytenberg \cite{royt}, Severa \cite{severa} (letter no.7), Uchino \cite{uch}, and others.  With Courant algebroids come \emph{Courant morphisms}: given two Courant algebroids $\A,\A'$ over $M,M'$, and a smooth map $\Phi:\,M\to M'$, a Courant morphism $\A\da\A'$ is a Dirac structure in $\A'\times\bar{\A}$ restricted to the graph of $\Phi$ (\cite{anton}, \cite{severa2}).  This definition allows that if any two sets of sections are related by the morphism, then so are their brackets.  The dashed arrow is a reminder that Courant morphisms are relations, and not always actual maps.

To construct the notion of a Lie group in Dirac geometry, extending the definition of Poisson Lie groups, the Courant algebroids $\A$ must themselves carry a multiplicative structure. This is achieved by requiring that $\A$ has a $VB$-groupoid structure $\A\rightrightarrows\g$ over $H\rightrightarrows\mathrm{pt}$ (\cite{raj}, \cite{davidpavol}) such that the groupoid multiplication Mult: $\A\times\A\da\A$ is a Courant morphism along the group multiplication $H\times H \to H$; such Courant algebroids are referred to as $CA$-groupoids.  This gives rise to the notion of a \emph{multiplicative Dirac structure} $E\subset \A$, being a Dirac structure which is also a sub-groupoid. For the standard Courant algebroid $\T H$, the multiplicative Dirac structures were independently classified by Jotz \cite{JotzDir} and Ortiz \cite{ortizd}\footnote{While both refer to these as ``Dirac Lie groups", we reserve the term and only refer to them as multiplicative Dirac structures.}.  Given a Lie group $H$, a \emph{Dirac Lie group} structure on $H$ is a pair $(\A,E)\gp\g$ of a $CA$-groupoid and Dirac structure, with the caveat that any element of $E_{h_1h_2}$ is Mult-related to a unique element of $E_{h_1}\times E_{h_2}$; this type of morphism is called a \emph{Dirac morphism}.  This definition, and further analysis, is due to Li-Bland and Meinrenken in \cite{lmdir}; they showed that the Dirac Lie group structures on $H$ are in 1-1 correspondence with so-called \emph{H-equivariant Dirac Manin triples} $(\d,\g,\h)_\beta$, with the ad-invariant $\beta\in S^2\d$ determining a (possibly degenerate) bilinear form, and $H$ acting on $\d$ by Lie automorphisms, integrating the adjoint action of $\h\subset\d$ and restricting to the adjoint action of $H$ on $\h$.  Unlike the Poisson case, the transverse subalgebras $\g,\h$ are not necessarily Lagrangian; the only condition is that $\g$ is $\beta$-coisotropic, meaning $\beta$ restricted to $\ann(\g)$ is 0.\\
Every Poisson Lie group can be regarded as a Dirac Lie group:  $(H,\Pi)$ gives rise to $(\A,E) = (\T H,\mathrm{Gr}(\Pi))$, with the groupoid structure of $\T H$ being the direct sum of the tangent group $\mathrm{T}H\gp\mathrm{pt}$, and the standard symplectic groupoid $\mathrm{T}^*H\gp\h^*$.  In fact, all of the Dirac Lie group structures for $\T H$ arise from Poisson structures on $H$, though considering more general Courant algebroids $\A$ give many more interesting examples.  One of the convenient properties of Dirac Lie group structures is that they carry a natural trivialization $(\A,E)\cong(H\times\q,H\times\g)$, where $\q$ is constructed from the classifying triple $(\d,\g,\h)_\beta$; the trivialization arises from the groupoid multiplication of $E$, together with the fact that $E$ is \emph{vacant} in the sense that $E^{(0)} = E_e$ \cite{lmdir}.\\
A powerful aspect of Dirac geometry, and Dirac Lie groups in particular, is it allows for short and straightforward proofs of various theorems in Poisson geometry.  As an example in outline, we look at Drinfeld's classification of Poisson homogeneous spaces:  for a given Poisson Lie group $(H,\Pi)$, begin with a Poisson homogeneous space $(H/K,\Pi_{H/K})$.  The Poisson action $H\times H/K\to H/K$ lifts to a Dirac morphism
    \begin{equation} (\T H, \mathrm{Gr}(\Pi_H))\times (\T(H/K),\mathrm{Gr}(\Pi_{H/K}))\da (\T(H/K),\mathrm{Gr}(\Pi_{H/K})) \end{equation}
making $\T(H/K)$ and the Dirac structure into a groupoid space.  Taking the backwards image along the Courant morphism $R_\pi$ induced by the natural projection $\pi:\,H\to H/K$, we obtain the Dirac morphism
    \begin{equation} (H\times\d, H\times\g)\times(H\times\d, L) \da (H\times\d, L) \end{equation}
This implies that $L$ is an $E$-module, giving us a trivialization $L\cong H\times\l$.  Since the backwards image of a Dirac structure is again a Dirac structure \cite{capg}, we get that $\l$ must be a Lagrangian subalgebra of $\d$.  By the definition of the backwards image, we have
    \begin{equation} \l = \{ (v,\pi^*\alpha)\in\h\oplus\h^*\,|\, v\in \h,\,\alpha\in (\h/\k)^*,\,\mathrm{T}\pi (v) = \iota_\alpha\Pi_{H/K}(\bar{e}) \}\end{equation}
which, in particular, means $\l\cap\h = \ker\mathrm{d}\pi = \k$.  For the second half of the statement, taking some Lagrangian subalgebra $\l\subset\d$ with $\l\cap\h=\k$, first we form $L=H\times\l$ and take its forward image via $R_\pi$.  (As a note, forwards images do not always yield smooth sub-bundles; a ``niceness" condition must be met, that $\ker R_\pi\cap L=0$, which is true in this case.)  Denoting $L' = R_\pi\circ L$, one can show that $L'\cap \mathrm{T}(H/K) = 0$.  Since $L'$ is a Dirac structure, this means it must be the graph of some Poisson bivector on $H/K$.  Pushing the map $(H\times\d,H\times\g)\times(H\times\d,L)\da(H\times\d,L)$ forward under $R_\pi$, we obtain the Dirac morphism
    \begin{equation} (H\times\d,H\times\g)\times(\T(H/K),L')\da (\T(H/K),L') \end{equation}
which yields the Poisson action of $H$ on $H/K$.

\section{Dirac Homogeneous Spaces, Results}
It is the goal of this thesis to classify Dirac homogeneous spaces, with initial data similar to Drinfeld's theorem of Poisson homogeneous spaces.  One of the main initial challenges in doing so is deciding on the proper definition of homogeneous space for the Dirac case:  since $(\A,E)$ are $VB$-groupoids (over $H$), it is natural to choose a pair $(\P,L)$ of Courant algebroid and Dirac structure over a group homogeneous space $H/K$ which are $(A,E)$-modules. There is no reasonable notion of transitive groupoid action, however, so this definition falls short.  Looking to the Poisson case, we obtain the Dirac morphism
    \begin{equation} (\T H, \mathrm{Gr}(\Pi_H))\times (\T(H/K),\mathrm{Gr}(\Pi_{H/K}))\da (\T(H/K),\mathrm{Gr}(\Pi_{H/K})) \end{equation}
as in the previous section. We also note that if we take the quotient $\A/E$, we obtain the group T$H$, as well as an induced action on the quotient
\[\T(H/K)/\mathrm{Gr}(\Pi_{H/K}) \cong \mathrm{T}(H/K).\]
Indeed, this action is transitive since it is the tangent lift of the transitive action of $H$ on $H/K$.\\
Even in the more general Dirac Lie group setting, the quotient $\A/E$ always yields a group.  This allows us to define a Dirac homogeneous space as a $CA$-groupoid module and Dirac structure $(\P,L)$ over $H/K$, such that the induced group action $\A/E\act \P/L$ is transitive.  Just as the groupoid multiplication of $E$ trivialises $(\A,E)$, so too does the groupoid action of $E$ affect the structure of $(\P,L)$, giving the modules the structure of associated bundles.  This allows for significant analysis to be done at the level of the fibres, and is where we begin in chapter 3 with \emph{quadratic linear groupoids}.  In chapter 4, we look at the groupoid modules for \emph{vacant} vector bundle- and Lie algebroid-groupoids, as the Dirac structures in Dirac Lie groups are vacant.  The vacant condition means that for the $VB$-groupoid $E\to H$ over the group $H$,  $E^{(0)} = E_e$. We extend this analysis in chapter 5 to metrized $VB$-groupoids $(\A,E)$: we take a quadratic VB-groupoid $\A\gp\g$ over $H$, together with a vacant, Lagrangian subgroupoid $E$, to approximate the Dirac structure in this more general setting.  We once again get a trivialization $(\A,E)\cong (H\times\q,H\times\g)$.  Considering modules $(\P,L)$ over $H/K$, with a transitive group action $\A/E\act \P/L$, we are able to write $(\P,L)\cong(H\times_K\p,H\times_K\l)$.  It turns out the global action and structure are determined by that at the identity fibres, giving the classification of such homogeneous spaces in terms of $\lambda$-coisotropic subspaces $\l\subset\g$ which are invariant under an action of the subgroup $K$.\\
Lastly, in chapter 6, we combine all of the previous results to classify the full Dirac case.  While the metrized linear- and $VB$-groupoid cases give us the global bundle structure of $(\P,L)$, as well as the formulae for the groupoid actions, we need to use reduction to help recover the Courant structure:  in \cite{lmdir}, the CA-groupoid $\A$ is obtained as a reduction $C/C^\perp$, with $C$ being a particular coisotropic sub-groupoid of $\T H\times\bar{\q}\times\q$; since there is a natural Dirac action of $\T H\times\bar{\q}\times\q$ on $\T(H/K)\times\bar{\q}$, we take a certain coisotropic $D_1\subset\T(H/K)\times{\bar{\q}}$ on which $C$ acts, and obtain $\P$ as a reduction $D_1/{D_1}^\perp$.\\
Ultimately, the analysis yields the following classification, effectively extending Drinfeld's Poisson theorem:
\begin{thmn} (\textbf{Classification of Dirac Homogeneous Spaces}) Given a Dirac Lie group $(H,\A,E)$ classified by $(\d,\g,\h)_\beta$, the Dirac homogeneous spaces on $(H/K,\P,L)$ are classified by $K$-invariant $\beta$-coisotropic subalgebras $\c\subset\d$ such that $\c\cap\h=\k$.
\end{thmn}
We give a number of examples for various Dirac Lie groups, in addition to giving an overview of how to recover the Lagrangian $\l$ in the Poisson case using Dirac geometry. Lastly, we extend a previous result of Lu \cite{luh} for producing Poisson homogeneous spaces from certain Poisson Lie subgroups to the Dirac case, and discuss some ideas regarding general Dirac actions.

\chapter{Preliminaries}
\section{Poisson Geometry}
\subsection{Poisson Manifolds, Poisson Bivectors, Morphisms}
We recall the definition of the \emph{graph} of a smooth map $\Phi:\,M\to M'$:
    \begin{equation} \mathrm{Gr}(\Phi) = \{ (\Phi(m),m)\,|\,m\in M\} \subset M'\times M. \end{equation}

\begin{dfn} Given a smooth manifold $M$, a \emph{Poisson structure} on $M$ is a Lie bracket $\{\cdot,\cdot\}$ on $C^\infty(M)$ such that $\{f,\cdot\}$ is a derivation.  This bracket is referred to as a \emph{Poisson bracket}, and a manifold $M$ together with such a bracket is called a \emph{Poisson manifold}.
\end{dfn}
\begin{example} Every manifold $M$ carries the trivial Poisson structure, $\{f,g\}=0$ for all $f,g$.
\end{example}
A \emph{linear Poisson structure} on a finite dimensional $\R$-vector space $V$ is a Poisson structure such that the Poisson bracket of any two linear functions yields another linear function.  In this case, the map $(f,g)\mapsto\{f,g\}$, restricted to linear functions gives a Lie bracket structure on $V^*$.  Conversely, Lie algebras give rise to linear Poisson structures on their duals.
\begin{example}\label{PoisLiestru} Given a Lie algebra $\g$, a linear Poisson structure is obtained on $\g^*$ in the following way:  given two linear functions $f_1,f_2$, the Poisson bracket is defined to be
\begin{equation} \{f_1,f_2\}(\alpha) = \<\alpha,[\mathrm{d}f_1(\alpha),\mathrm{d}f_2(\alpha)]_\g\>\end{equation}
for each $\alpha\in\g^*$.  This is referred to as the \emph{Lie-Poisson structure}.
\end{example}
With Poisson manifolds come the notion of morphisms.
\begin{dfn} If $(M_1, \{,\}_1),\,(M_2,\{,\}_2)$ are two Poisson manifolds, then a smooth map $\varphi:\,M_1\to M_2$ is a \emph{Poisson morphism} if the pullback $\varphi^*:\,C^\infty(M_2)\to C^\infty(M_1)$ is a Lie algebra morphism for the Poisson structures.
\end{dfn}
\begin{example} Give a Lie algebra morphism $\varphi:\,\g\to\h$, the dual map $\varphi^*:\,\h^*\to\g^*$ is a Poisson morphism for the natural linear Poisson structures.
\end{example}
The derivation property of the bracket means $\{f,g\}$ depends only on the differential, and hence defines a bivector field $\Pi\in \Gamma(\wedge^2\mathrm{T}M)$ such that:
\begin{equation} \{f,g\} = \<\Pi,\dr f\wedge \dr g\>.\end{equation}
\begin{prop} A given bivector field $\Pi$ on $M$ defines a Poisson structure if and only if $[\Pi,\Pi]=0$, where $[\cdot,\cdot]$ is the \emph{Schouten bracket}, the natural generalisation of the Lie bracket of vector fields to higher rank tensors (\cite{Schouten},\cite{nij}).
\begin{proof}
The proof follows from writing formulae in local coordinates. See section 1.8.1 of \cite{dufng}.
\end{proof}
\end{prop}


\subsection{Poisson Lie Groups, Homogeneous Spaces}
The notion of Poisson manifolds can be extended to Lie groups.
\begin{dfn} A Lie group $H$ with a Poisson structure $\Pi$ is called a \emph{Poisson Lie group} if the multiplication map
\begin{equation} (H,\Pi)\times(H,\Pi)\to (H,\Pi);\quad (h_1,h_2)\mapsto h_1h_2 \end{equation}
is a Poisson morphism.  Equivalently, $\Pi$ is a \emph{Poisson Lie structure} on $H$ if and only if it is a multiplicative Poisson tensor. i.e.: for all $h_1,h_2\in H$
\begin{equation} \Pi(h_1h_2) = (L_{h_1})_*\Pi(h_2) + (R_{h_2})_*\Pi(h_1). \end{equation}
\end{dfn}
\begin{dfn} (\cite{STS}) Let $(H,\Pi)$ be a Poisson Lie group, and $(M,\Pi_M)$ a Poisson manifold. A \emph{Poisson action} of $H$ on $M$ is a group action such that the action map
\begin{equation} (H,\Pi_H)\times(M,\Pi_M)\longrightarrow (M,\Pi_M) \end{equation}
is a Poisson morphism.  If the action is transitive, then $M$ is called a \emph{Poisson homogeneous space}.
\end{dfn}
\begin{example} The left and right multiplication of any Poisson Lie group are Poisson actions.
\end{example}
As a result of the definition, if $(H,\Pi)$ is a Poisson Lie group, then $\Pi(e) = 0$.  Hence, we obtain a linear Poisson structure on $\h$, and so a Lie bracket $[\cdot,\cdot]^*$ on $\h^*$.
\begin{thm} (\cite{Drinfeldq}) If $(H,\Pi)$ is a Poisson Lie group, then there is a natural Lie algebra structure on $\d=\h\times\h^*$ given by
\begin{equation} [(x,\alpha),(y,\beta)] = ([x,y]+\ad^*_\alpha y-\ad^*_\beta x, [\alpha,\beta]^* + \ad^*_x\beta - \ad^*_y\alpha). \end{equation}
This Lie algebra is often denoted $\d = \h\bowtie\h^*$, and has a non-degenerate, ad-invariant scalar product
\begin{equation} \<(x,\alpha),(y,\beta)\> = \<x,\beta\> + \<y,\alpha\> \end{equation}
for which both $\h$ and $\h^*$ are Lagrangian subalgebras of $\d$.
\end{thm}
\begin{dfn}\label{lietrip} Let $H$ be a Lie group, with Lie$(H)=\h$.  An \emph{H-equivariant Lie algebra triple} $(\d,\g,\h)$ is a Lie algebra $\d$, with two transverse subalgebras $\g,\h$, and an action of $H$ on $\d$ by Lie automorphisms, integrating the adjoint action of $\h\subseteq\d$, restricting to the adjoint action of $H$ on $\h$.  If $\d$ is a quadratic Lie algebra, such that $\g,\h$ are Lagrangian, then $(\d,\g,\h)$ with the $H$ action is called an \emph{H-equivariant Manin triple}.
\end{dfn}

Drinfeld classified Poisson Lie group structures on a given Lie group $H$ via this Lie algebraic data.
\begin{thm} (\cite{Drinfeldh}) If $H$ is a Lie group, then Poisson structures on $H$ making it a Poisson Lie Group are classified by $H$-equivariant Manin triples $(\d,\g,\h)$.
\end{thm}
Since both $\h$ and $\g$ are Lagrangian, we can associate $\g\cong\h^*$ via
    \begin{equation} \h^* = (\d/\g)^* \cong \g^\perp = \g. \end{equation}
So for a given $\h$, both the vector space structure of $(\d,\g,\h)$ and the Lie algebra structure of $\h$ remain the same regardless of the Poisson structure $\Pi$; it is the Lie algebra structure of $\g$ which depends on $\Pi$.  Drinfeld classified Poisson homogeneous spaces for $(H,\Pi)$ using data related to the Manin triples.
\begin{thm} (\cite{Drinfeldp}) Let $(H, \Pi_H)$ be a Poisson Lie group with associated $H$-equivariant Manin triple $(\d,\g,\h)$, and let $K$ be a closed Lie subgroup.  There is a 1-1 correspondence between Poisson structures on $H/K$ which make the action map a Poisson morphism, and Lagrangian $K$-stable subalgebras $\l\subset\d$ with $\l\cap\h=\k:=\mathrm{Lie}(K)$.
\end{thm}
\section{Groupoids and Groupoid Spaces}\label{gmodules}
\begin{dfn} A \emph{Lie groupoid} $H\gp H^{(0)}$ is a manifold $H$, and a closed submanifold $H^{(0)}\subset H$, together with:
  \begin{enumerate}
    \item Two submersions $\s,\t:\,H\to H^{(0)}$ called the \emph{source} and \emph{target} maps.
    \item A smooth product map $H^{(2)} \to H$, $(x,y)\mapsto x\circ y$ where
     \begin{equation} H^{(2)}=H {}_\s\times_\t H = \{ (x,y)\in H\times H\,|\, \s(x)=\t(y)\}. \end{equation}
     The product map satisfies $\s(x\circ y) = \s(y)$, $\t(x\circ y) = \t(y)$, and is associative for all compatible triples.
    \item An embedding $H^{(0)}\to H,\,\,c\mapsto 1_c$, with the image called the \emph{units} or \emph{identities}, which satisfy
     \begin{equation} 1_{\t(x)} \circ x = x, \quad x\circ 1_{\s(x)} = x. \end{equation}
     For brevity, we will usually identify $1_c$ with $c$.
    \item A smooth map, called the \emph{inversion map} $\mathsf{i}:\, H \to H$, such that for all $x\in H$, $\s(\mathsf{i}(x)) = \t(x)$, $\t(\mathsf{i}(x)) = \s(x)$, and
    \begin{equation} \mathsf{i}(x) \circ x = 1_{\s(x)},\quad x\circ \mathsf{i}(x) = 1_{\t(x)}. \end{equation}
     We will denote $\mathsf{i}(x)$ by $x^{-1}$.
    \end{enumerate}
\end{dfn}

\begin{example}
For any manifold $M$, we can form the \emph{pair groupoid} $M\times M\gp M$, with $\s(y,z) = z$, $\t(y,z) = y$, and the multiplication given by
\begin{equation} (y,z) \circ (z,x) = (y,x). \end{equation}
The units are embedded as $M_\triangle$, and the inversion map is given by $(y,z)^{-1} = (z,y)$.
\end{example}
\begin{example}
If $G$ is a Lie group, and $M$ is a manifold with a $G$-action, then we can define the \emph{action groupoid} $(G\times M) \gp M$ by the following:
\begin{equation} \s(g,m) = m, \quad \t(g,m) = g.m, \quad (g_1, (g_2).m) \circ (g_2, m) = (g_1g_2, m). \end{equation}
The identity at $m$ is $(e,m)$, and the inversion map is given by $(g,m)^{-1} = (g^{-1}, g.m)$.
\end{example}
\begin{example}
A group $G$ is a groupoid with a single object, $G\gp \mathrm{pt}$.
\end{example}
There is a notion of a groupoid action on a set which while analogous to a group action, requires slightly more information.

\begin{dfn}Let $H\gp H^{(0)}$ be a Lie groupoid. A \emph{groupoid space} for $H$ (also called an $H$-space, or an $H$-module) is a manifold $M$ together with a smooth structure map $\u:M\to H^{(0)}$, and a smooth \emph{action map}
\begin{equation}
H{}_\s\times_\u M\to M,\,\qquad (h,m)\mapsto h\circ m
\end{equation}
where $H{}_\s\times_\u M = \{(h,m)\in H\times X\,|\, \s(h)=\u(m)\}$, such that the following compatibility conditions are satisfied for all composable $h,k\in H$, $(h,m)\in H{}_\s\times_\u M$:
\begin{enumerate}
\item $\u(h\circ m) = \t(h),$
\item $h\circ (k\circ m) = (h\circ k)\circ m,$
\item $(\u(m))\circ m = m.$
\end{enumerate}
The structure map $\u$ is sometimes referred to as a ``moment map".
\end{dfn}
\begin{example} $H$ is trivially an $H$-space via the groupoid multiplication, with $\u=\t$.
\end{example}
\begin{example} Any Lie groupoid $H\gp H^{(0)}$ acts on its object space $H^{(0)}$ with $\u = \mathsf{i}$, the unit embedding.
\end{example}
\section{$VB$-Groupoids}
\subsection{$VB$-Groupoids, Dual $VB$-Groupoids}
\begin{dfn}\label{VBgpdf} Let $H\gp H^{(0)}$ be a Lie groupoid. A \emph{$VB$-groupoid} over $H$ is a vector bundle $V\to H$ with a Lie groupoid structure $V\gp V^{(0)}$ such that $\mathrm{Gr}(Mult_V)\subseteq V\times V\times V$ is a vector sub-bundle along the graph of the groupoid multiplication of $H$.  A $VB$-groupoid is often denoted by the diagram:
\begin{center}
\begin{tikzpicture}

  \node (G)   {$V$};
  \node (H) [right of=G, node distance=1.5cm] {$V^{(0)}$};
  \node (I) [below of=G, node distance=1.3cm] {$H$};
  \node (J) [right of=I, node distance=1.5cm] {$H^{(0)}$};

  \path[->]
([yshift= 2pt]G.east) edge node[above] {} ([yshift= 2pt]H.west)
([yshift= -2pt]G.east) edge node[below] {} ([yshift= -2pt]H.west);
\path[->]
([yshift= 2pt]I.east) edge node[above] {} ([yshift= 2pt]J.west)
([yshift= -2pt]I.east) edge node[below] {} ([yshift= -2pt]J.west);
  \draw[->] (G) to node {} (I);
  \draw[->] (H) to node [swap] {} (J);
  \end{tikzpicture}
\end{center}
\end{dfn}
\begin{example} Let $H\to H^{(0)}$ be any Lie groupoid.  Then

\begin{center}
\begin{tikzpicture}

  \node (G)   {T$H$};
  \node (H) [right of=G, node distance=1.5cm] {T$H^{(0)}$};
  \node (I) [below of=G, node distance=1.3cm] {$H$};
  \node (J) [right of=I, node distance=1.5cm] {$H^{(0)}$};

  \path[->]
([yshift= 2pt]G.east) edge node[above] {} ([yshift= 2pt]H.west)
([yshift= -2pt]G.east) edge node[below] {} ([yshift= -2pt]H.west);
\path[->]
([yshift= 2pt]I.east) edge node[above] {} ([yshift= 2pt]J.west)
([yshift= -2pt]I.east) edge node[below] {} ([yshift= -2pt]J.west);
  \draw[->] (G) to node {} (I);
  \draw[->] (H) to node [swap] {} (J);
  \end{tikzpicture}
\end{center}

with $\s_{\mathrm{T}H} = \mathrm{T}\s_H$, etc., is a $VB$-groupoid, sometimes known as the \emph{tangent prolongation groupoid}.
\end{example}
\begin{dfn}Let $V,V'$ be $VB$-groupoids over $H,H'$.  A \emph{$VB$-groupoid morphism} is a Lie groupoid morphism $\Phi:V\to V'$ such that the induced map $\mathrm{Gr}(\mathrm{Mult}_V)\to \mathrm{Gr}(\mathrm{Mult}_{V'})$ is a vector bundle morphism.
\end{dfn}
\begin{prop}\label{vbmor} If $V, V'$ are $VB$-groupoids, then a vector bundle map $\Phi: V\to V'$ is a $VB$-groupoid morphism if and only if the induced map
    \begin{equation} V\times V\times V \longrightarrow V'\times V'\times V' \end{equation}
takes $\mathrm{Gr}(\mathrm{Mult}_V)$ to $\mathrm{Gr}(\mathrm{Mult}_{V'})$
\begin{proof}
Suppose the induced map, which we will call $\tilde{\Phi}$, takes $\mathrm{Gr}(\mathrm{Mult}_V)$ to $\mathrm{Gr}(\mathrm{Mult}_{V'})$.  In other words,
    \begin{equation}\label{gpm} \tilde{\Phi}(v\circ w, v, w) = (\Phi(v\circ w),\Phi(v),\Phi(w)) = (\Phi(v)\circ\Phi(w),\Phi(v),\Phi(w))\end{equation}
with appropriate basepoints.  Taking $v=\s_V(w)$ shows $\Phi$ is compatible with the source maps; with $w=\t(v)$, we have the compatibility with the targets.  Equation \ref{gpm} shows $\Phi$ respects the groupoid multiplications, and so we can conclude that $\Phi$ is a Lie groupoid morphism.\\
The other direction follows from the definition.
\end{proof}
\end{prop}

\begin{dfn} Let $V\gp V^{(0)}$ be a $VB$-groupoid over $H\gp H^{(0)}$.  A $VB$-module for $V\gp V^{(0)}$ is a vector bundle $P\to M$ with a groupoid action of $V$ such that the graph of the groupoid action inside $P\times V\times P$ is a vector sub-bundle.
\end{dfn}

\begin{example} The \emph{standard symplectic groupoid} $\mathrm{T}^*H\gp \h^*$ is a $VB$-groupoid over the Lie group $H$.  Given an $H$-homogeneous space $H/K$, the cotangent bundle $\mathrm{T}^*(H/K)$ becomes a $VB$-module for $\mathrm{T}^*H$ in the following way:  in left-trivialisation, we have $\mathrm{T}^*H\cong H\times\h^*$, and $\mathrm{T}^*(H/K) \cong H\times_K (\h/\k)^*$.  Since $(\h/\k)^*\cong \ann(\k)\subset\h^*$, the moment map
    \[ \u:\, H\times_K\ann(\k)\to \h^*,\qquad [h,\alpha] \mapsto \Ad_h\alpha \]
makes $\mathrm{T}^*(H/K)$ into a $VB$-module for $\mathrm{T}^*H\gp \h^*$ with action defined by
    \[ (h_1, \Ad_{h_2}\alpha)\circ [h_2,\alpha] = [h_1h_2,\alpha].\]
\end{example}
\begin{dfn} Given a $VB$-groupoid $V\gp V^{(0)}$ over $H\gp H^{(0)}$, the vector bundle
    \begin{equation}\mathrm{core}(V) := V\vert_{H^{(0)}}/V^{(0)} \end{equation}
is called the \emph{core} of $V$; the core can be identified with both $\ker(\s_V)\vert_{H^{(0)}}$ and $\ker(\t_V)\vert_{H^{(0)}}$ as both are complements to $V^{(0)}$ inside $V\vert_{H^{(0)}}$.  These isomorphisms are given by
\begin{align*}
\text{core}(V) \to \ker(\s_V)|_{H^{(0)}},&\quad [v]\mapsto v-\s_V(v)\\
\text{core}(V) \to \ker(\t_V)|_{H^{(0)}},&\quad [v]\mapsto v-\t_V(v).
\end{align*}
\end{dfn}

\begin{dfn}\label{vacantdef} If $V\gp V^{(0)}$ is a $VB$-groupoid over $H\gp H^{(0)}$ such that $\mathrm{core}(V)=0$, then $V$ is called a \emph{vacant} $VB$-groupoid.  In other words, $V$ satisfies the property $V^{(0)}= V\vert_{H^{(0)}}$.
\end{dfn}
\begin{lem} There are canonical isomorphisms
    \begin{equation} \ker(\s_V)\cong \t_H^*\mathrm{core}(V),\quad \ker(\t_V) \cong \s_H^*\mathrm{core}(V). \end{equation}
\begin{proof}
View $H^{(0)}\subseteq H$.  The groupoid $H$ acts on $\ker(\s_V)$ by translation from the right:
    \[ v\mapsto v\circ a^{-1},\qquad v\in \ker(\s_V)_h,\, a\in H;\, \s_H(a)=\s_H(h).\]
In particular, for $a=h$, $h\circ a^{-1}=\t_H(h)$, and so the translation induces an isomorphism $\ker(\s_V)_h\cong \ker(\s_V)_{\t_H(h)}\cong\text{core}(V)_{\t_H(h)}$.\\
The second isomorphism follows similarly, using left translation.
\end{proof}
\end{lem}
This gives us two short exact sequences:
 \begin{equation}\label{vbses1}
    \begin{CD}
    0 @>>> \t_H^*(\text{core}(V)) @>>> V @>{\s_{V}}>> \s_H^*(V^{(0)}) @>>> 0,\\
    0 @>>> \s_H^*(\text{core}(V)) @>>> V @>{\t_{V}}>> \t_H^*(V^{(0)}) @>>> 0.
    \end{CD}
    \end{equation}
Along $H^{(0)}$, these sequences have canonical splittings via the projection $V\vert_{H^{(0)}}\to V^{(0)}$.  The inclusions in the sequences \ref{vbses1} are given by:
\begin{align*}
\t_H^*(\text{core}(V))\to V,&\quad [v]\mapsto (v-\s_V(v))\circ h;\quad v\in V_{\t_H(h)},\\
\s_H^*(\text{core}(V))\to V,&\quad [v]\mapsto h\circ (w-\t_V(w));\quad w\in V_{\s_H(h)}.
\end{align*}
\begin{thm}\label{vbdualgrpd} (\cite{Pradinesv}) The dual vector bundle to $V\gp V^{(0)}$ becomes a $VB$-groupoid
\begin{equation} V^*\gp (V^*)^{(0)} \end{equation}
with $(V^*)^{(0)} = \text{core}(V)^*\cong \text{ann}_{V^*}(V^{(0)})$, and $\text{core}(V^*) = (V^{(0)})^*$.
\begin{proof}
(Outline) Dualising the short exact sequences \ref{vbses1}, we get:
    \begin{equation}\label{vbses2}
    \begin{CD}
    0 @>>> \s_H^*(V^{(0)})^* @>>> V^* @>{\t_{V^*}}>> \t_H^*(\ann(V^{(0)})) @>>> 0, \\
    0 @>>> \t_H^*(V^{(0)})^* @>>> V^* @>{\s_{V^*}}>> \s_H^*(\ann(V^{(0)})) @>>> 0.
    \end{CD}
    \end{equation}

The surjections in \ref{vbses2} define the target and source for $V^*$. Explicitly, for any $h\in H$, $\alpha\in {V^*}_h$, and $[w]\in \text{core}(V)$:
    \begin{align*}
    \< \t_{V^*}(\alpha), [w]\> &= \< \alpha, (w-\s_V(w))\circ h\>,\\
    \< \s_{V^*}(\alpha), [w]\> &= \< \alpha, h\circ(w-\t_V(w))\>.
    \end{align*}
For the multiplication, if $\alpha \in {V^*}_h$ and $\tau \in {V^*}_g$ are composable elements, then for any composable $\xi \in V_h$, and $\zeta \in V_g$
    \begin{equation} \< \alpha\circ\tau, \xi\circ\zeta\> = \<\alpha, \xi\> + \<\tau,\zeta\>. \end{equation}
For full details, see \cite{alfraj}, \cite{Pradinesv}.
\end{proof}
\end{thm}
\begin{rmrk} Note that even if $V$ is a group (i.e.: $V^{(0)} = 0$ over $H^{(0)} = $pt), the dual bundle is only a groupoid in general.  As an example, any tangent group $\mathrm{T}H\gp\text{pt}$ for $H$ a Lie group yields the standard symplectic groupoid $\mathrm{T}^*H\gp\h^*$ upon dualising.
\end{rmrk}

\begin{rmrk}\label{annr} From the definition, the graph of the groupoid multiplication in $V^*$ is the annihilator of the graph of the groupoid multiplication on $V$.
\end{rmrk}

\begin{prop} If $\Phi:\,V\to V'$ is a morphism of $VB$-groupoids over a common base $H$, then $\Phi^*:\,(V')^*\to V^*$ is a morphism of $VB$-groupoids.
\begin{proof}
We recall from Proposition \ref{vbmor} that if $V, V'$ are $VB$-groupoids, then a vector bundle map $V\to V'$ is a $VB$-groupoid morphism if and only if the induced map
    \begin{equation} V\times V\times V \longrightarrow V'\times V'\times V' \end{equation}
takes $\mathrm{Gr}(\mathrm{Mult}_V)$ to $\mathrm{Gr}(\mathrm{Mult}_{V'})$.  Dualising $\Phi$, we get a map from the annihilator of $\mathrm{Gr}(\mathrm{Mult}_{V'})$ to the annihilator of $\mathrm{Gr}(\mathrm{Mult}_{V})$. Via Remark \ref{annr}, this means $\Phi^*$ is a $VB$-groupoid morphism from $(V')^*\to V^*$.

\end{proof}
\end{prop}

\subsection{Dual $VB$-Modules}
Let $V\gp V^{(0)}$ be a $VB$-groupoid over $H\gp H^{(0)}$, and let $P\to M$ be a $VB$-module for $V$, with moment maps $\u_P:\,P\to V^{(0)}$, $\u_M:\,M\to H^{(0)}$.  We will show that $P^*$ has the structure of a module for $V^*\gp(V^*)^{(0)}$.

\begin{lem}\label{dualmu} There is a natural bundle map $\u_{P^*}:\,P^*\longrightarrow (V^*)^{(0)}$.
\begin{proof}
The action of $V$ on $P$ defines linear maps
    \[ V_h\times_{V^{(0)}_{\s(h)}} P_m \,\longrightarrow P_{h\circ m} \]
for all $h\in H,\,m\in M$ with $\s_H(h)=\u_M(m)$.  Taking $h=\u_M(m)\in H^{(0)}$, and restricting to $0$ elements in $P_m$, this becomes a map whose domain is
\begin{align*}
    V_{\u(m)}\times_{V^{(0)}_{s(h)}}\{0\}_m &= \{(v,0)\,|\,s_V(v) = \u_P(0) = 0 \} \subseteq V_{\u(m)}\times\{0\}_m\\
                                            &= \ker(s_V)_{\u(m)}\\
                                            &\cong \text{core}(V)_{\u(m)}
\end{align*}
Thus, the map is
    \[ \text{core}(V)_{\u(m)} = V_{\u(m)}/V^{(0)}_{\u(m)} \longrightarrow P_m \]
with $(v,0_m) \mapsto (v\circ 0_m)\in P_m$.  Dualising gives us a map
    \[ \u_{P^*}:\,P^* \to (\text{core}(V))^* = (V^*)^{(0)}.\qedhere \]

\end{proof}
\end{lem}

This map is, in fact, a moment map for a groupoid action of $V^*$ on $P^*$.

\begin{prop}
If $P\to M$ is a $VB$-groupoid module for $V\gp V^{(0)}$ with moment maps $\u_P:\,P\to V^{(0)}$ and $\u_M:\,M\to H^{(0)}$, then $P^*\to M$ is a $VB$-groupoid module for $V^*\gp (V^*)^{(0)}$ with moment map $\u_{P^*}$.
\begin{proof}
Using $\u_{P^*}$ from Lemma \ref{dualmu}, let $\alpha\in {V^*}_h$, $\eta\in {P^*}_m$ with $s_{V^*}(\alpha)=\u_{P^*}(\eta)$ and $s_H(h)=\u_M(m)$.  Then $\alpha\circ \eta\in {P^*}_{h\circ m}$ is described as follows:  for any $x\in P_{h\circ m}$, there exists $v\in V_h$, $y\in P_m$ with $x=v\circ y$.  Then
    \begin{equation}\label{Vdmdeq} \<\alpha\circ\eta, v\circ y\>_P = \<\alpha,v\>_V + \<\eta,y\>_P. \end{equation}
This is independent of decomposition:  if $x = v'\circ y'$ is another decomposition, define $a:= v-v'$ and $b:=y-y'$.  It's clear that $a\circ b$ is defined, and
    \[ a\circ b = (v-v')\circ(y-y') = v\circ y - v'\circ y' = x-x = 0_{h\circ m} \]
It follows that $a\in \ker(t_V)$, and so there exists a $c\in V_{\s(h)}$ such that $a = 0_h\circ c$.  We get
    \begin{align*}
        \<\alpha,a\> + \<\eta,b\> &= \< \alpha, 0_h\circ c\> + \<\eta,b\> \\
                                  &= -\<\s_{V^*}(\alpha), c^{-1}\> + \<\eta,b\> \\
                                  &= -\<\u_{P^*}(\eta), c^{-1}\> + \<\eta,b\> \\
                                  &= \< \eta, b - (\u_{P^*})^*(c^{-1})\>.
    \end{align*}
Since $c^{-1}\in V_{\s(h)} = V_{\u_M(m)}$, then by definition, $(\u_{P^*})^*(c^{-1}) = (c^{-1})\circ 0_m$.  As $a=0_h\circ c$, and $a\circ b= 0_{h\circ m}$, this implies
    \[ 0_{h\circ m} = 0_h\circ c \circ b. \]
Cancelling and rearranging, this gives us that $b= c^{-1}\circ 0_m$ and hence $\<\eta, b-(\u_{P^*})^*(c^{-1})\>=0$.\\
Since $\eta$ was arbitrary, this means that
    \[ \<\alpha,v\> + \<\eta,y\> = \<\alpha,v'\> + \<\eta,y'\> \]
for any choice of decomposition.\\
Lastly, we show that $\u_{P^*}(\alpha\circ \eta) = t_{V^*}(\alpha)$.  Let $\eta\in \text{core}(V)_{t(h)}$.  Then
    \begin{align*}
        \<\t_{V^*}(\alpha) - \u_{P^*}(\alpha\circ\eta),\zeta\> =& \<\t_{V^*}(\alpha),\eta\> - \<\alpha\circ\eta,(\u_{P^*})^*\eta\>\\
                        &= \<\alpha, \zeta\circ 0_h\> - \<\alpha\circ \eta, \zeta\circ 0_{h\circ m} \>\\
                        &= \<\alpha, \zeta\circ 0_h\> - \<\alpha\circ \eta,(\zeta\circ 0_h)\circ 0_m\>\\
                        &= \<\alpha,\zeta\circ 0_h\> - \<\alpha,\zeta\circ 0_h\> - \<\eta, 0_m\>\\
                        &= 0.
    \end{align*}
As $\zeta$ was arbitrary, the result follows.
\end{proof}
\end{prop}

\begin{prop}If $\Phi: V\to V'$ is a morphism of $VB$-groupoids over a common base $H$, together with a corresponding morphism of groupoid modules $\psi:\,P\to P'$ over a common base $M$, then the dual map $\psi^*:\,(P')^*\to P^*$ is a module morphism for $\Phi^*:\,(V')^*\to V^*$.
\begin{proof}
First, let $\eta'\in {(P')^*}_m$.  Then for every $\zeta\in \text{core}(V)_{\u(m)}$ we have
    \begin{align*}  \<\u_{P^*}(\psi^*(\eta')),\zeta\> &= \<\psi^*(\eta'),\zeta\circ 0_m\>\\
                                &= \< \eta', \Phi(\zeta)\circ\psi(0_m)\>
    \end{align*}
as well as
    \begin{align*} \<\Phi^*(\u_{(P')^*}(\eta')),\zeta\> &= \< \u_{(P')^*}(\eta'), \Phi(\zeta)\>\\
                                &= \< \eta', \Phi(\zeta)\circ\psi(0_m)\>
    \end{align*}
since $\psi(0_m) = 0'_m$.  As $\zeta$ was arbitrary, we conclude
\[ \Phi^*\circ \u_{(P')^*} = \u_{P^*}\circ \psi^* \]
Secondly, with $\eta'$ as above, let $\alpha'\in {(V')^*}_h$ with $\s(h)=\u_{(P')^*}(m)$ and $\alpha'\circ \eta'$ defined.  Then for all $x \in P_{h\circ m}$, decomposed as $x = v\circ y$ for $v\in V_h$, $y\in P_y$, we have
    \begin{align*}
    \<\psi^*(\alpha'\circ\eta'),x\> = \<\alpha'\circ\eta', \psi(v\circ y)\> &= \<\alpha'\circ\eta', \Phi(v)\circ\psi(y)\>\\
                    &= \<\alpha',\Phi(v)\> + \<\eta',\psi(y)\>\\
                    &= \<\Phi^*(\alpha'),v\> + \<\psi^*(\eta'),y\>\\
                    &= \<\Phi^*(\alpha')\circ\psi^*(\eta'), v\circ y\>.
    \end{align*}
So, $\psi^*(\alpha'\circ\eta') = \Phi^*(\alpha')\circ\psi^*(\eta')$ for all composable $\alpha'\in (V')^*,\,\eta'\in (P')^*$.
\end{proof}
\end{prop}

\begin{lem}\label{VBquotactt} Let $V\gp V^{(0)}$ be a $VB$-groupoid over $H\gp H^{(0)}$, and let $W\subset V$ be a $VB$-subgroupoid.  The quotient of $V$ by $W$ is a groupoid $V/W\gp V^{(0)}/W^{(0)}$, with core
\[ \text{core}(V/W) = V\vert_{H^{(0)}}/(W|_{H^{(0)}} + V^{(0)}). \]
The quotient map is a groupoid morphism.  The annihilator, $\text{ann}(W)\subset V^*$ is a subgroupoid of $V^*$, with units
\[ (\text{ann}(W))^{(0)} = \text{ann}(W)\cap (V^*)^{(0)}.\]
Furthermore, if $P$ is a $V$-module, and $L\subset P$ is a submodule with respect to $W\subset V$, then $P/L$ is a module over $V/W$, and $\text{ann}_{P^*}(L)$ is a module over $\text{ann}_{V^*}(W)$.
\begin{proof}
The first part is a special case of Lemma 5.9 in \cite{moermc}.  The core is
    \begin{align*} \text{core}(V/W) &= (V/W)|_{H^{(0)}}/(V/W)^{(0)} \\
                                    &= (V|_{H^{(0)}}/W|_{H^{(0)}})\,/\,(V^{(0)}/W^{(0)})\\
                                    &= V\vert_{H^{(0)}}/(W|_{H^{(0)}} + V^{(0)})
    \end{align*}
since $W^{(0)} = V^{(0)}\cap W|_{H^{(0)}}$. If two elements in ann$(W)$ are composable, then their product must also be in $\ann(W)$, by the definition of the multiplication in $V^*$.\\
Since range$(\u_L)\subseteq W^{(0)}$, $\u_P$ descends to a map $P/L\longrightarrow V^{(0)}/W^{(0)}$.  If $\bar{v}\in V/W$ and $\bar{x}\in P/L$ with $\s(\bar{v})=\u(\bar{x})$, then the composition is given by
\[ \bar{v}\circ\bar{x} = \overline{v\circ x} \]
for any composable lifts $v,x$.  If $x'$ is another lift of $\bar{x}$, and $v'$ another lift of $v$ with $\s_V(v') = \u_P(x')$, then $(v-v')=w\in W$ and $(x-x')=y\in L$ and are composable, so we get
\[ v'\circ x' = (v-w)\circ(x-y) = v\circ x + w\circ y. \]
This means $\overline{v\circ x} = \overline{v'\circ x'}$ since $(w\circ y)\in L$.\\
Lastly, $\text{ann}_{P^*}(L)$ being a module over $\text{ann}_{V^*}(W)$ follows from the definition of the groupoid action, and that $L$ is a $W$-module.  Its set of units is the dual of the core of $V/W$, and hence $\text{ann}(W)\cap (V^*)^{(0)}$.
\end{proof}
\end{lem}

\section{Lie Algebroids}\label{plalg}
\subsection{Definition and Examples}
The notion of a Lie algebroid was introduced by Pradines in \cite{Pradinesg} as the infinitesimal version of Lie groupoids.
\begin{dfn}
A \emph{Lie algebroid} $E\to M$ is a vector bundle over a manifold $M$, together with a Lie bracket $[\,\cdot\,,\,\cdot\,]$ on the sections $\Gamma(E)$, and a linear bundle map $\a:\,E\to \mathrm{T}M$ called the \emph{anchor} such that
\begin{equation}   [\sigma,f\tau] = f[\sigma,\tau] + (a(\sigma)f)\tau
\end{equation}
for all $\sigma,\tau\in \Gamma(E)$, and functions $f$ on $M$.
\end{dfn}
While this definition is sensible in many categories, we will mainly concern ourselves with smooth objects.
\begin{prop} Given a Lie algebroid $E\to M$, the anchor map $\a$ is an homomorphism of Lie algebras:
\begin{equation}  \a([\sigma,\tau]) = [\a(\sigma),\a(\tau)] \end{equation}
for all $\sigma,\tau \in \Gamma(E)$.
\begin{proof}
This follows from the Jacobi identity and the Leibniz rule.  See Lemma 8.1.4 in \cite{dufng}.
\end{proof}
\end{prop}

\begin{example} Let $M$ be a manifold. The tangent bundle T$M$ is a Lie algebroid over $M$ with the anchor map being the identity, and bracket being the normal Lie bracket of vector fields.
\end{example}
\begin{example} A Lie algebra is a Lie algebroid over a point.
\end{example}
\begin{example}\label{acLA} Let $M$ be a manifold, and $\g$ a Lie algebra together with an action $\psi: \,\g\to \mathfrak{X}(M)$.  We can associate to this data an \emph{action Lie algebroid} as follows:  the vector bundle $M\times\g \to M$, with anchor map defined as $\a(m,x) = \psi(x)\vert_m$, and a bracket on $\Gamma(M\times\g) = C^\infty(M,\g)$ given by
\begin{equation}
[\alpha,\beta]_m = [\alpha_m,\beta_m] + \psi(\alpha_m)_m\beta - \psi(\beta_m)_m\alpha.
\end{equation}
\end{example}

\begin{dfn} Let $E\to M$ be a Lie algebroid, and $S\subseteq M$ a submanifold.  A \emph{sub-Lie algebroid} of $E$ is a vector sub-bundle $F\to S$ such that $\a(F)\subseteq \mathrm{T}(S)$, and for all $\sigma_i\in \Gamma(E)$, if $\sigma_i|_S\in \Gamma(F)$, then $[\sigma_1,\sigma_2]|_S\in\Gamma(F)$.
\end{dfn}

\begin{dfn} Given Lie algebroids $E,E'$ over $M,M'$, a \emph{morphism of Lie algebroids} is a vector bundle morphism
\begin{center}
\begin{tikzpicture}
  \node (P) {$E$};
  \node (B) [right of=P] {$E'$};
  \node (A) [below of=P] {$M$};
  \node (C) [right of=A] {$M'$};

  \draw[->] (P) to node {} (B);
  \draw[->] (P) to node [swap] {} (A);
  \draw[->] (A) to node [swap] {$\Phi$} (C);
  \draw[->] (B) to node {} (C);
\end{tikzpicture}
\end{center}
whose graph in $E'\times E$ is a sub-Lie algebroid.  A \emph{comorphism of Lie algebroids}
\begin{center}
\begin{tikzpicture}
  \node (P) {$E$};
  \node (B) [right of=P] {$E'$};
  \node (A) [below of=P] {$M$};
  \node (C) [right of=A] {$M'$};

  \draw[dashed,->] (P) to node {} (B);
  \draw[->] (P) to node [swap] {} (A);
  \draw[->] (A) to node [swap] {$\Phi$} (C);
  \draw[->] (B) to node {} (C);
\end{tikzpicture}
\end{center}
is a vector bundle comorphism $\Phi^*E'\to E$ whose graph is a sub-Lie algebroid.
\end{dfn}

\subsection{Duals and Poisson Structures}
\begin{thm}\label{EdualPoisson} (See \cite{dufng}, Section 8.2) Let $E\to M$ be a Lie algebroid.  The dual bundle $E^*\to M$ is a Poisson manifold, with fibrewise linear bracket defined by
    \begin{align*}   \{f,g\}&=0 \\ \{\sigma,f\} &= \a(\sigma)f \\ \{\sigma,\tau\} &= [\sigma,\tau] \end{align*}
    with $f,g\in C^\infty(M)$, and $\sigma,\tau\in \Gamma(E)$ viewed as fibre-wise linear functions on $E^*$.  This bracket can be extended to other functions on $E^*$ via the Leibniz rule.
\end{thm}

\begin{thm}\label{Edualmor} There is a 1-1 correspondence between comorphisms of Lie algebroids $E\da E'$ and Poisson morphisms $E^*\to (E')^*$.
\end{thm}
See \cite{higginsmack}, \cite{higmack2}, \cite{comor}.

\subsection{Reduction of Lie Algebroids Over Principal Bundles}\label{laredk}

Let $K$ be a Lie group with Lie algebra $\k$, and $M$ a manifold on which $K$ acts.  Let $E \to M$ be a Lie Algebroid, such that $K$ acts on $E$ by Lie algebroid automorphisms.  This action gives us a Lie algebra homomorphism $\k\to \Gamma(\mathfrak{aut}(E))$, $\xi\mapsto \xi_E$ defined by
\begin{equation}
\xi_E\sigma = \frac{\partial}{\partial t}\bigg|_{t=0} \exp(-t\xi)^*\sigma,\qquad\xi\in\k,\,\,\sigma\in\Gamma(E).
\end{equation}
This is an action of $\k$ on $E$ by Lie algebroid derivations.\\
The vertical bundle of $M$ is $VM\cong M\times\k$, which inherits a natural Lie algebroid structure:  the anchor is the inclusion map $\a:M\times\k\to \mathrm{T}M$, and the bracket is obtained by extending the Lie bracket on the constant sections of $M\times\k$.
\begin{dfn} A $K$-equivariant bundle map $\rho\in \mathrm{Hom}_K(M\times\k,\,E)$ is said to ``define generators for the $K$-action on E" if
\begin{equation}
\xi_E = [\rho(\xi),\cdot\,]
\end{equation}
as operators on $\Gamma(E)$, using $\rho$ to denote both $M\times\k\to E$, and $\k\to\Gamma(E)$.
\end{dfn}

If one differentiates the $K$-equivariance property $(k^{-1})^*\rho(\xi)=\rho(\Ad_k\xi)$, it follows that
\begin{equation}
[\rho(\xi_1),\rho(\xi_2)]_E = \rho([\xi_1,\xi_2]_\k).
\end{equation}
The image of the generators via the anchor are the generating vector fields on $M$: let $\xi\in \k$, $\sigma\in \Gamma(E)$, $f\in C^\infty(M)$, then we have
\begin{equation}
\xi_E(f\sigma) = [\rho(\xi),f\sigma] = f[\rho(\xi),\sigma] + a(\rho(\xi))f\cdot\sigma = f(\xi_E\sigma) + a(\rho(\xi))f\cdot\sigma.
\end{equation}
But $\xi_E(f\sigma) = f(\xi_E\sigma) + (\xi_Mf)\sigma$, so equating both sides, we get
\begin{equation}
(\xi_Mf)\sigma = (\a_E(\rho(\xi))f)\sigma
\end{equation}
for all $f\in C^\infty(M),\,\sigma\in\Gamma(E),\,\xi\in\k$.  Hence,
\begin{equation}\label{rhinj}
\a_E(\rho(\xi))=\xi_M.
\end{equation}
If we suppose that the $K$ action on $M$ is principal, equation \ref{rhinj} implies $\a_E\circ\rho$ is injective, and so $\rho$ must be injective.  Hence, we get the commuting diagram
\begin{center}
\begin{tikzpicture}
  \node (P) {$M\times\k$};
  \node (B) [right of=P] {$E$};
  \node (A) [below of=P] {$\mathrm{V}M$};
  \node (C) [right of=A] {$\mathrm{T}M$};

  \draw[->] (P) to node {$\rho$} (B);
  \draw[->] (P) to node [swap] {$\cong$} (A);
  \draw[->] (A) to node [swap] {} (C);
  \draw[->] (B) to node {$\a_E$} (C);
\end{tikzpicture}
\end{center}

\begin{lem}\label{Kideal} Identify $M\times\k$ as a Lie sub-algebroid of $E$ via the injection $\rho$.  The space of $K$-invariant sections of $M\times\k$ is a Lie ideal of the $K$-invariant sections of $E$.
\begin{proof}
Let $\sigma\in \Gamma(E)^K$, $\kappa\in \Gamma(M\times\k)^K$.  Since $[\sigma,\kappa]$ is $K$-invariant, we need only show that it lies in $\Gamma(M\times\k)$.  Let $\{e_i\}$ be a basis for $\k$.  Then $\kappa\in\Gamma(M\times\k)^K$ can be written as $\kappa=\sum_i f_ie_i$, and we have:
\begin{align*}
[\sigma,\kappa] &= [\sigma,\sum_if_ie_i]\\
                &=\sum_i f_i[\sigma,e_i] + \sum_i \big(\a_E(\sigma)f_i\big)e_i.
\end{align*}
The first term vanishes, owing to the $K$-invariance of $\sigma$.  The term $\sum_i(\a_E(\sigma)f_i)e_i$ lies in $\Gamma(M\times\k)$ as required. It is $K$-invariant since $\Gamma(E)^K$ is closed under the bracket.  Hence, for any $\sigma\in\Gamma(E)^K,\,\kappa\in\Gamma(M\times\k)^K$, we have $[\sigma,\kappa]\in\Gamma(M\times\k)^K$.
\end{proof}
\end{lem}

\begin{thm}\label{KactonEthm} Let $E\to M$ be a Lie algebroid with an action $K\act E$ by Lie algebroid automorphisms such that the $K$ action on $M$ is a principal action, and let $\rho\in \mathrm{Hom}_K(M\times\k,E)$ define generators for the $K$ action. Then the space
\begin{equation}
E//K := (E/\rho(M\times\k)) / K
\end{equation}
is a Lie algebroid over $M/K$ with the induced bracket and anchor map.
\begin{proof}
Let $\pi: M\to M/K$ be the natural projection.  The sections of $E//K$ are characterised as $\Gamma(E//K) = \Gamma(E)^K/\Gamma(M\times\k)^K$.  The kernel of the vector bundle map $\phi_K:\,E\to E//K$ is $\rho(M\times\k)\cong VM \cong \ker\mathrm{T}\pi$, with $\mathrm{T}\pi: \mathrm{T}M\to \mathrm{T}(M/K)$.  Hence, the anchor map $\a_E:\, E\to \mathrm{T}M$ descends to a map $\a_{E//K}:\,E//K\to \mathrm{T}(M/K)$ via
\begin{equation}
\a_{E//K}( x\md M\times\k) = \a_E(x)\md VM
\end{equation}
or more concisely, $\a_{E//K}(\overline{x}) = \overline{\a_E(x)}$, where the overline denotes the image in the respective quotients.  This gives us the following commuting diagram:
\begin{center}
\begin{tikzpicture}
  \node (P) {$E$};
  \node (B) [node distance=2.5cm,right of=P] {T$M$};
  \node (A) [below of=P] {$E//K$};
  \node (C) [node distance=2.5cm,right of=A] {$\mathrm{T}(M/K)$};

  \draw[->] (P) to node {$\a_E$} (B);
  \draw[->] (P) to node [swap] {$\phi_K$} (A);
  \draw[->] (A) to node [swap] {$\a_{E//K}$} (C);
  \draw[->] (B) to node {T$\pi$} (C);
\end{tikzpicture}
\end{center}

By the Lemma \ref{Kideal}, $\Gamma(M\times\k)^K$ is an ideal in $\Gamma(E)^K$, and hence we get a Lie bracket on $\Gamma(E//K)=\Gamma(E)^K/\Gamma(M\times\k)^K$.  Let $\overline{\sigma},\overline{\tau}\in \Gamma(E//K)$, and let $\sigma,\tau\in \Gamma(E)^K$ be any lifts of them.  Then we see, for any $f\in C^\infty(M/K)=C^\infty(M)^K$:
\begin{align*}
 [\overline{\sigma},f\overline{\tau}] &:= \overline{[\sigma,f\tau]}\\
              &= \overline{f[\sigma,\tau]} + \overline{a_E(\sigma)f\tau}\\
              &= f[\overline{\sigma},\overline{\tau}]+a_{E//K}(\overline{\sigma})f\overline{\tau}.
 \end{align*}
Hence, $\a_{E//K}$ is an anchor map for $E//K$ and its descending bracket, meaning $E//K\to M/K$ is a Lie algebroid.
\end{proof}
\end{thm}

\subsection{Pullback Lie Algebroids and Their Relation to Reductions}
\begin{dfn}\label{pbLA} (\cite{higginsmack}) Let $L\to N$ be a Lie algebroid over a manifold $N$, and let $\Phi: M \to N$ be a smooth map such that $\mathrm{T}\Phi$ is transverse to the anchor of $L$.  The \emph{pullback Lie algebroid} of $L$ is defined as
\begin{equation}
\Phi^!L = \{(v,x)\,|\,\mathrm{T}\Phi(v) = a_{L}(x)\} \subset \mathrm{T}M\times L\vert_{\mathrm{Gr}(\Phi)},
\end{equation}
where we identify $\mathrm{Gr}(\Phi) = M$.  $\Phi^!L$ is a sub-Lie algebroid of T$M\times L$, and by the transversality condition, $\rk(\Phi^!L) = \rk(L) + \dim(M) - \dim(N)$.
\end{dfn}
\begin{example}
Let $\Phi:\,M\to S$ be a submersion.  Then the pullback $\Phi^!0$ is the vertical bundle $VM$, since
    \[ \Phi^!0 = \{ (v,0)\,|\,\mathrm{T}\Phi(v) = 0\} \subseteq \mathrm{T}M\times 0. \]
\end{example}
Given a Lie algebroid $E\to M$ over a principal $K$-bundle $M\xrightarrow{\,\,\pi\,\,} M/K$, and a Lie algebroid $E'$ over $M/K$, two natural questions to ask include  ``Is $\pi^!(E//K) = E$?", and ``Is $(\pi^!E')//K = E'$?".  For the latter, see Corollary \ref{otherLiered}.  For the former, we begin by noting that rk$\pi^!(E//K) = \mathrm{rk}(E//L) - \dim(M/K) + \dim(M)$, and by construction $\rk(E//K) = \rk(E) - \dim(\k)$.  Thus
\begin{equation}
\rk\pi^!(E//K) = \rk(E).
\end{equation}

\begin{prop}\label{LAlgpbk} Let $E$ be a Lie algebroid over $M$ carrying a $K$ action as in Theorem \ref{KactonEthm}, and let $\phi_K:\,E\to E//K$ be the Lie algebroid morphism taking $E$ to its reduction.  The natural map $f: E \to \mathrm{T}M\times (E//K)$ given by $f(x) = (\a_E(x),\phi_K(x))$ is an injective Lie algebroid automorphism.  Hence, the image of $E$ under $f$ is $\pi^!(E//K)$, and thus they are isomorphic as Lie algebroids.
\begin{proof}
Both the anchor map $\a_E$ and the reduction map $\phi_K$ are Lie algebroid morphisms, and hence $f$ must be.  If $x\in\ker f$ for $x\in E_m$, then $\phi_K(x) = 0$, which occurs if and only if $x\in M\times\k$.  Then $\a_E(x) = 0$ at $h$, but $\a_E(\rho(m,\xi)) = \xi_M\vert_m$, and since the $K$ action is principal, the generating vector fields cannot vanish.  Thus, $\ker f = \{0\}$.\\
Since $a_{E//K}(\phi_K(x)) = \mathrm{T}\pi(\a_E(x))$, this means $f(E)\subseteq \pi^!(E_{\rd})$.  Thus, since the map $f$ is an injective Lie algebroid morphism between two algebroids of the same rank, we conclude $E\cong \pi^!(E//K)$:
\begin{center}
\begin{tikzpicture}
  \node (P) {$\pi^!(E//K)$};
  \node (B) [node distance=2.5cm, right of=P] {$E//K$};
  \node (A) [node distance=2.5cm, below of=P] {$\mathrm{T}M$};
  \node (C) [node distance=2.5cm, below of=B] {$\mathrm{T}(M/K)$};
  \node (P1) [node distance=2cm, left of=P, above of=P] {$E$};
  \draw[->] (P) to node {} (B);
  \draw[->] (P) to node {$\a_{\pi!(E//K)}$} (A);
  \draw[->] (A) to node [swap] {$\mathrm{T}\pi$} (C);
  \draw[->] (B) to node {$\a_{E//K}$} (C);
  \draw[->, bend right] (P1) to node [swap] {$\a_E$} (A);
  \draw[->, bend left] (P1) to node {$\phi_K$} (B);
  \draw[->, dashed] (P1) to node {$f$} (P);
\end{tikzpicture}
\end{center}
\end{proof}
\end{prop}
\section{Courant Algebroids}\label{prelimCAs}
\begin{dfn}\label{courdef} (\cite{lwx}) A \emph{Courant algebroid}  over a manifold $M$ is a vector bundle $\A\to M$, with a bundle map $\a:\,\A\to TM$ called the \emph{anchor}, a non-degenerate symmetric bilinear form (referred to as an \emph{inner product} or \emph{metric}) $\<\cdot,\cdot\>$ on the fibres of $\A$, and a bilinear \emph{Courant bracket} $\Cour{\cdot,\cdot}$ on the sections $\Gamma(\A)$, such that the following are satisfied for every $\sigma_i\in\Gamma(\A)$:
\begin{enumerate}
\item $\Cour{\sigma_1,\Cour{\sigma_2,\sigma_3}}=\Cour{\Cour{\sigma_1,\sigma_2},\sigma_3}
+\Cour{\sigma_2,\Cour{\sigma_1,\sigma_3}}$,
\item $\a(\sigma_1)\< \sigma_2,\sigma_3\> =\< \Cour{\sigma_1,\sigma_2},\,\sigma_3\>+\< \sigma_2,\,\Cour{\sigma_1,\sigma_3}\>$,
\item $\Cour{\sigma_1,\sigma_2}+\Cour{\sigma_2,\sigma_1}=\a^*(\dr \< \sigma_1,\sigma_2\>)$,
\end{enumerate}
where $\a^*:\,T^*M\to \A^*\cong\A$ is the map dual to the anchor.
\end{dfn}
From these, a number of other properties may be discerned, including:
\begin{enumerate}
\item[4.] $\Cour{\sigma_1,f\sigma_2}=f\Cour{\sigma_1,\sigma_2}+(\a(\sigma_1)f)\sigma_2$,
\item[5.] $\a(\Cour{\sigma_2,\sigma_2})=[\a(\sigma_1),\a(\sigma_2)]_{\mathfrak{X}(M)}$,
\item[6.] $\a\circ\a^* = 0$
\end{enumerate}
for $\sigma_i\in\Gamma(\A),f\in C^\infty(M)$. Given a Courant algebroid $\A\to M$, a Lagrangian sub-bundle $E\subset\A$ whose space of sections $\Gamma(E)$ is closed under the Courant bracket is called a \emph{Dirac structure}.  It follows that $E$, with the restrictions of the anchor map and the Courant bracket, is a Lie algebroid.

\begin{example} \cite{courdir}  For any manifold $M$, the bundle $\T M = \mathrm{T}M\oplus\mathrm{T}^*M$ is a Courant algebroid under the following metric and bracket:
\begin{align*}
\< (v_1,\alpha_1),(v_2,\alpha_2)\> &= \<v_1,\alpha_2\> + \<v_2,\alpha_1\>\\
\Cour{(v_1,\alpha_1),(v_2,\alpha_2)} &= ([v_1,v_2],\,\, \L_{v_1}\alpha_2 - \iota_{v_2}\dr\alpha_1)
\end{align*}
for all vector fields $v_i\in \mathfrak{X}(M)$, and for $\alpha_i\in \Gamma(\mathrm{T}^*M)$.  The anchor map $\a:\T M\to \mathrm{T}M$ is given by projection along T$^*M$.  This example is called the \emph{standard Courant algebroid} over $M$.  Both $\mathrm{T}M$ and $\mathrm{T}^*M$ are examples of Dirac structures in this case.
\end{example}
\begin{example} Given a quadratic Lie algebra $\g$ together with Lie algebra action $\rho:\,\g\to \frak{X}(M)$, the bundle $\A:=M\times\g$ becomes a Courant algebroid if and only if the stabilizers $\ker\rho_m$ are coisotropic in $\A$ \cite{capg}.  In this case, the anchor map is given by $\a(m,\xi) = \rho(\xi)_m$, and the bracket extends the Lie bracket of constant sections via
    \begin{equation} \Cour{\sigma_1,\sigma_2} = [\sigma_1,\sigma_2] + \L_{\rho(\sigma_1)}\sigma_2-\L_{\rho(\sigma_2)}\sigma_1 + \rho^*\<\mathrm{d}\sigma_1,\sigma_2\> \end{equation}
 with $[\cdot,\cdot]$ being the pointwise Lie bracket.  This is called the \emph{action Courant algebroid}.

\end{example}

\subsection{Courant Morphisms}
The notion of \emph{Courant morphism} is due to A.Alekseev and P.Xu \cite{anton} (see also \cite{courmor}).

\begin{dfn} Suppose $\A,\A'$ are Courant algebroids over $M,M'$.  A \emph{Courant morphism} $R_\Phi:\,\A\da\A'$ is a smooth map $\Phi:\,M\to M'$, together with a sub-bundle
    \begin{equation} R_\Phi \subset (\A'\times\bar{\A})\vert_{\mathrm{Gr}(\Phi)} \end{equation}
(where $\bar{\A}$ denotes the Courant algebroid $\A$ with the opposite metric) which satisfies the following properties:
    \begin{enumerate}
        \item $R_\Phi$ is Lagrangian,
        \item $(\a\times\a')(R_\Phi)$ is tangent to the graph of $\Phi$,
        \item If $\sigma_1,\sigma_2 \in \Gamma(\A'\times\bar{\A})$ restrict to sections of $R_\Phi$, then so does $\Cour{\sigma_1,\sigma_2}$.
    \end{enumerate}
\end{dfn}
For a given Courant morphism $R_\Phi:\,\A\da\A'$ with elements $x\in\A,\,x'\in\A'$, we denote $x\sim_{R_\Phi} x'$ if $(x',x)\in R_\Phi$.  For any sections $\sigma_i\in\Gamma(\A)$, $\sigma_i'\in\Gamma(\A')$, we have
    \begin{equation} \sigma_1 \sim_{R_\Phi} \sigma_1', \, \sigma_2 \sim_{R_\Phi} \sigma_2' \quad \Rightarrow\quad \Cour{\sigma_1,\sigma_2}\sim_{R_\Phi}\Cour{\sigma_1',\sigma_2'};\quad \<\sigma_1,\sigma_2\> = \Phi^*\<\sigma_1',\sigma_2'\>. \end{equation}
\begin{example}\label{standardmorphism}
Given any smooth map $\Phi:\,M\to M'$, there is the associated \emph{standard morphism} $R_\Phi:\,\T M\da \T M'$ where
    \begin{equation} (v,\alpha) \sim_{R_\Phi} (v',\alpha') \quad \Leftrightarrow\quad v' = \mathrm{T}\Phi(v),\,\,\alpha = \Phi^*(\alpha'). \end{equation}
\end{example}

\subsection{Coisotropic Reduction of Courant Algebroids, Pullbacks}
One of the best ways of producing Courant algebroids is by reducing existing ones via some compatible quotienting process.

\begin{prop}\label{CAcoisred} \textbf{(Coisotropic Reduction)}  Let $\A\to M$ be a Courant algebroid, $S\subset M$ be a submanifold, and let $C\subset \A|_S$ be subbundle such that
\begin{enumerate}
\item $C$ is coisotropic,
\item $\a(C)\subset \mathrm{T}S$, $\a(C^\perp)=0$,
\item if $\sigma,\,\tau\,\in\Gamma(\A)$ restrict to sections of $C$, then so does $\Cour{\sigma,\tau}$.
\end{enumerate}
Then the bracket, anchor map, and inner product on $C$ descend to $C/C^\perp$, making it a Courant algebroid over $S$.
\begin{proof}
See Proposition 2.1 in \cite{capg}.  Two important aspects of the proof are that $C^\perp$ is involutive, and that for any $\sigma\in \Gamma(C)$, and $\eta\in \Gamma(C^\perp)$, we have $\Cour{\sigma,\eta}\in \Gamma(C^\perp)$, both of which stem from the Courant axiom
    \[ \a(\sigma_1)\< \sigma_2,\sigma_3\> =\< \Cour{\sigma_1,\sigma_2},\,\sigma_3\>+\< \sigma_2,\,\Cour{\sigma_1,\sigma_3}\>.\qedhere \]
\end{proof}
\end{prop}


\begin{lem}\label{redinstage} \textbf{(Reduction in Stages)}  Let $V$ be a metrized vector space, $C\subseteq V$ a coisotropic subspace, and $\pi:\,C\to C/C^\perp$ the quotient map.  Given another coisotropic subspace $D\subseteq V$, we have
    \[ C\cap D\mathrm{\,coisotropic\,} \iff C^\perp\subseteq D \iff D^\perp \subseteq C. \]
Furthermore, in this case,
    \[ \pi(D)/\pi(D)^\perp = C\cap D/ (C\cap D)^\perp.\]
\begin{proof}
$C\cap D$ is coisotropic if and only if $C^\perp+D^\perp \subseteq C\cap D$, which is true if and only if $C^\perp\subset D$ and $D^\perp\subset C$.  The two conditions $C^\perp\subseteq D$ and $D^\perp\subseteq C$, however, are equivalent.  By definition of the induced metric on $C/C^\perp$, for all $v,v'\in C$ we have
    \[ \< \pi(v),\pi(v')\> = \< v,v'\>, \]
and hence
    \begin{align*} \pi(D)^\perp &= \{ \pi(v) \,|\, \<\pi(v),\pi(v')\>=0\,\,\forall\,\pi(v')\in \pi(D)\}\\
                                &= \{\pi(v)\,|\,\<v,v'\>=0 \,\,\forall\,v'\in C\cap D\}\\
                                &= \pi((C\cap D)^\perp)\\
                                &= \pi(C^\perp + D^\perp).
    \end{align*}
Since $\ker(\pi)=C^\perp$, this shows that
    \[ \pi(D)/\pi(D)^\perp = (C\cap D)/(C\cap D)^\perp.\qedhere\]
\end{proof}
\end{lem}
\begin{rmrk} If $\Phi: D\to D/D^\perp$ is the quotient map, we have
    \[ \pi(D)/\pi(D)^\perp = C\cap D/(C\cap D)^\perp = \Phi(C)/\Phi(C)^\perp.\]
\end{rmrk}

\begin{dfn} \textbf{(Pullbacks)}  Given a Courant algebroid $\A$ over $N$, and a smooth map $\Phi:\,M\to N$ whose differential is transverse to $\a$, one can define the pullback of $\A$ as the quotient $\Phi^!\A = C/C^\perp$, with
    \begin{align*}
            C &= \{ (v,\alpha;x) \,|\,\a(x) = \mathrm{T}\Phi(v)\} \subseteq \T M\times\A\\
            C^\perp &= \{ (0,\nu;\a^*\lambda)\,|\,(\mathrm{T}\Phi)^*\lambda + \nu \in \ann((\mathrm{T}\Phi)^{-1}\mathrm{ran}(\a))\}.
    \end{align*}
The rank of the pullback is given by $\mathrm{rk}(\Phi^!\A) = \mathrm{rk}(\A) -2(\dim M-\dim S)$.  The shriek notation $\Phi^!$ is used to distinguish it from the pullback as a vector bundle, $\Phi^*\A$.
\end{dfn}

\begin{example} (Proposition 2.9 in \cite{capg}) For any smooth $\Phi:\,M\to N$, one has a canonical isomorphism
    \begin{equation} \Phi^!(\T N) = \T M. \end{equation}
\end{example}
\subsection{Reduction of Courant Algebroids with Isotropic Generators}
The second sense of Courant reduction comes from \cite{courred}, analogous to Section \ref{laredk} for Lie algebroids.

\begin{dfn} Let $\A\to M$ be a Courant algebroid, and let $K$ be a Lie group which acts on $\A$ by Courant automorphisms.  A $K$-equivariant bundle map $\rho\in\mathrm{Hom}_K(M\times\k,\A)$ is said to define \emph{generators for the K-action} if $\xi_\A=[[\rho(\xi),\cdot]]$ as operators on $\Gamma(\A)$.  These generators are \emph{isotropic} if $\<\rho(\xi),\rho(\xi)\>=0$ for all $\xi$.
\end{dfn}

\begin{thm} \textbf{($K$-Reduction with Isotropic Generators)}  Let $M$ be a manifold with a principal $K$ action, $\A\to M$ be a Courant algebroid such that $K$ acts on $\A$ by Courant automorphisms, and let $\rho\in \mathrm{Hom}_K(M\times\k,\A)$ define generators for the $K$ action which are isotropic.  Then $\Gamma(M\times\k)^K$ is an ideal in $\Gamma((H\times\k)^\perp)^K$, and
\begin{equation}
\A//K = (\rho(M\times\k)^\perp/\rho(M\times\k))/K
\end{equation}
with the induced bracket, anchor map, and inner product is a Courant algebroid over $M/K$.  If $E\subset \A$ is a $K$-invariant Dirac structure such that $E\cap\rho(M\times\k)$ has constant rank, then the reduced bundle
\begin{equation}
E//K= (E\cap \rho(M\times\k)^\perp)/(E\cap \rho(M\times\k))/K \subset \A_{\rd}
\end{equation}
is a Dirac structure of $\A//K$.
\begin{proof} See \cite{courred}.
\end{proof}
\end{thm}

Pullbacks and reductions interact well in the Courant case as well.

\begin{prop}\label{CourKred}  Let $\A$ be a Courant algebroid over $M/K$.  Then with $\pi:\,M\to M/K$, the pullback $\pi^!\A$ has a natural $K$ action with isotropic generators, and $(\pi^!\A)//K \cong \A$.
\begin{proof}
By definition, $\pi^!\A = C/C^\perp$, with $C\subseteq \T M\times\A$ being the fibre product along $\a$ and T$\pi$.  Let $K$ act on $\T M\times\A$ via the natural action on $\T M$ and the trivial action on $\A$.  This action preserves $C$ (and thus $C^\perp$), and has isotropic generators given by the inclusion
    \begin{equation} M\times\k \hookrightarrow \mathrm{T}M\subset \A\times\T M. \end{equation}
Since T$M\cap C^\perp = 0$, the generators descend to the quotient $C/C^\perp$, and hence
    \begin{equation} (\pi^!\A)//K = \big((C\cap(M\times\k)^\perp)/(C^\perp+(M\times\k))\big)/K. \end{equation}
$M\times\k$ injects into $C$ as elements $\{(\kappa,0;0)\,|\,\kappa\in\k\}$, and so $C+(M\times\k)=C$ giving
    \[C^\perp\cap (M\times\k)^\perp = C^\perp.\]
Let $C'\subset C$ be the subset of elements with $\alpha=0$.  Then
    \begin{align*}
        C\cap(M\times\k)^\perp &= \{ (v,\alpha;x) \,|\,\a(x) = \mathrm{T}\pi(v)\}\cap \{ (w,\delta;y)\,|\,\delta\in\ann(\k)\}\\
                                &= C'\oplus C^\perp
     \end{align*}
and hence we get
    \begin{equation} (C\cap(M\times\k)^\perp)/(C^\perp+(M\times\k))) = C'/(M\times\k) = \pi^*\A. \end{equation}
Quotienting by $K$, we obtain $(\pi^!\A)//K \cong \A$.
\end{proof}
\end{prop}
\begin{cor}\label{otherLiered} If $E\subset \A$ is a Dirac structure, then $(\pi^!E)//K \cong E$.
\end{cor}

\section{Dirac Lie Groups}
\subsection{Manin Pairs}
A \emph{Manin pair} $(\A,E)$ is a Courant algebroid $\A\to M$ and a Dirac structure $E\subset\A$.  A \emph{morphism of Manin pairs} (also called a \emph{Dirac morphism})
    \begin{equation} R_\Phi:\, (\A,E) \da (\A',E') \end{equation}
is a Courant morphism with underlying map $\Phi:\,M\to M'$, and the additional property that for every $m\in M$, any element of $E'_{\Phi(m)}$ is $R_\Phi$-related to a unique element of $E_m$.
\begin{example} For any Manin pair $(\A,E)\to M$, there is a morphism of Manin pairs
    \begin{equation} R:\,(\T M,\mathrm{T}M) \da (\A,E) \end{equation}
where $(v,\alpha)\sim_R x\quad \Leftrightarrow \quad v=\a(x),$ and $x-\a^*(\alpha)\in E$.
\end{example}
\begin{example} Suppose $(M,\Pi),(M',\Pi')$ are Poisson manifolds, and $\Phi:\,M\to M'$ is a smooth map.  Then the standard morphism
    \begin{equation} R_\Phi:\,\T M\da \T M' \end{equation}
defines a morphism of Manin pairs $R_\Phi:\,(\T M,\mathrm{Gr}(\Pi))\da (\T M',\mathrm{Gr}(\Pi'))$ if and only if $\Phi$ is a Poisson morphism.
\end{example}

\subsection{$CA$-Groupoids, Multiplicative Manin Pairs}
\begin{dfn}\label{LAg} Let $H\gp H^{(0)}$ be a Lie groupoid.
\begin{enumerate}
  \item An \emph{$LA$-groupoid} over $H$ is a Lie algebroid $E\to H$, together with a groupoid structure such that Gr(Mult$_E)\subset E\times E \times E$ is a Lie sub-algebroid along Gr(Mult$_H$).
  \item  A \emph{$CA$-groupoid} over $H$ is a Courant algebroid $\A\to H$, together with a groupoid structure such that Gr(Mult$_\A)\subset \A\times\bar{\A}\times\bar{\A}$ is a Dirac structure along Gr(Mult$_H$).
 \end{enumerate}
As with the $VB$-groupoid case, it is common to indicate $LA$- and $CA$-groupoids via a diagram such as
\begin{center}
\begin{tikzpicture}

  \node (G)   {$\A$};
  \node (H) [right of=G, node distance=1.5cm] {$\A^{(0)}$};
  \node (I) [below of=G, node distance=1.3cm] {$H$};
  \node (J) [right of=I, node distance=1.5cm] {$H^{(0)}$};

  \path[->]
([yshift= 2pt]G.east) edge node[above] {} ([yshift= 2pt]H.west)
([yshift= -2pt]G.east) edge node[below] {} ([yshift= -2pt]H.west);
\path[->]
([yshift= 2pt]I.east) edge node[above] {} ([yshift= 2pt]J.west)
([yshift= -2pt]I.east) edge node[below] {} ([yshift= -2pt]J.west);
  \draw[->] (G) to node {} (I);
  \draw[->] (H) to node [swap] {} (J);
  \end{tikzpicture}
\end{center}

\end{dfn}
\begin{dfn}
Given two $CA$-groupoids $\A,\A'$ over $H,H'$, a \emph{morphism of CA-groupoids} $R_\Phi:\,\A\da\A'$ consists of a Lie groupoid morphism $\Phi:\,H\to H'$, and a Courant morphism $\R_\Phi\subset \A'\times\bar{\A}$ which is a Lie sub-groupoid along the graph of $\Phi$.
\end{dfn}
One can similarly define morphisms of $LA$-groupoids.
\begin{dfn} A \emph{multiplicative Manin pair} $(\A,E)$ is a Manin pair where $\A\gp\A^{(0)}$ is a $CA$-groupoid over $H\gp H^{(0)}$, and $E\gp E^{(0)}$ is an Dirac structure and subgroupoid of $\A$.  A \emph{morphism of multiplicative Manin pairs} $R_\Phi:\,(\A,E)\da (\A',E')$ is a morphism of Manin pairs which is also a morphism of $CA$-groupoids $\A\da\A'$.
\end{dfn}
When $H$ is a group (so $H^{(0)}=\mathrm{pt}$, the main case we are concerned with), the groupoid multiplication defines a Courant morphism $\text{Mult}_\A$ covering the group multiplication
\begin{center}
\begin{tikzpicture}
  \node (P) {$\A\times\A$};
  \node (B) [right of=P] {$\A$};
  \node (A) [below of=P, node distance=1.5cm] {$H\times H$};
  \node (C) [right of=A] {$H$};

  \draw[->, dashed] (P) to node {$\text{Mult}_\A$} (B);
  \draw[->] (P) to node [swap] {} (A);
  \draw[->] (A) to node [swap] {Mult$_H$} (C);
  \draw[->] (B) to node {} (C);
\end{tikzpicture}
\end{center}

\subsection{Groupoids, Modules, Reductions}
We will briefly link previous sections on $VB$-groupoids, modules, and reductions of Courant algebroids, to yield an important result about reductions of groupoid actions as restrictions of $CA$-groupoid actions.\\
Suppose that $\A\gp \A^{(0)}$ is a $VB$-groupoid over $H\gp H^{(0)}$ with a non-degenerate fibre metric which is multiplicative:
    \[ \<\xi_1\circ \eta_1,\xi_2\circ \eta_2\> = \<\xi_1,\xi_2\> + \<\eta_1,\eta_2\> \]
for appropriately composable elements.  The metric allows us to identify $\A\cong\A^*$, and since the metric is compatible with the multiplication, this is an isomorphism of groupoids.  If $C\subset \A$ is a sub-groupoid with $C^{(0)}=C\cap\A^{(0)}$, then $C^\perp$ is a sub-groupoid of $\A$ with $(C^\perp)^{(0)} = C^\perp\cap\A^{(0)}$.  Hence, if $C$ is coisotropic, then $C/C^\perp$ becomes a groupoid over
\[ (C/C^\perp)^{(0)} = (C\cap\A^{(0)})/(C^\perp\cap\A^{(0)}).\]

\begin{prop}
Suppose $\P$ is an $\A$-module, with a non-degenerate fibre metric which is compatible with the groupoid action.  Suppose $D\subset\P$ is a coisotropic submodule for $C$.  Then $D^\perp$ is a $C^\perp$ submodule, and $D/D^\perp$ is a module for $C/C^\perp$.
\begin{proof}
Let $x\in D_{h\cdot m}$ be arbitrary, then $x$ can be written as $x= \u(x')_h\circ x'$ for $x'\in D_m$. If $\zeta\in {C^\perp}_h$, and $y\in {D^\perp}_m$ with $\s(\zeta)=\u(y)$, then by the compatibility of the groupoid action with the metric:
    \[ \<x,\zeta\circ y\>_{D_{h\cdot m}} = \<\u(x'),\zeta\> + \<x',y\> = 0. \]
This means that $\zeta\circ y \in D^\perp$, which implies $C^\perp\act D^\perp$.  If $\nu$ is any element of $C$ with $s(\nu)=\u(y)$, then $\u(D^\perp)\subseteq (C^\perp)^{(0)}$, and so $D^\perp$ is a $C^\perp$-module.  Let $\bar{\xi}\in C/C^\perp$, $\bar{x}\in D/D^\perp$ be composable elements.  Choose composable lifts $\xi,x$; if $\xi',x'$ are another set of composable lifts, then $\xi-\xi'\in C^\perp$, $x-x'\in D^\perp$, and  $s(\xi-\xi')=\u(x-x')$.  Hence
    \[ \xi'\circ x' = \xi\circ x + (\xi-\xi')\circ(x-x') \]
meaning $\overline{\xi'\circ x'} = \overline{\xi\circ x}$.
\end{proof}
\end{prop}

\begin{example}\label{pullbmod} Let $\A$ be a $CA$-groupoid, with a $CA$-module $\P$.  The base groupoid of $\A$, $H$, acts on the base manifold $M$ of $\P$.  Suppose $H$ also acts on another manifold $N$, and let $\pi:\,N\to M$ be an equivariant submersion ($\u_N=\u_M\circ\pi$).  Then there is an action of $\T H\times\A$ on $\T N\times\P$.\\
Let $C\subset\T H\times \A$ be the pre-image of the diagonal embedding of T$H$ in $\mathrm{T}H\times\mathrm{T}H$ (via projection and anchor), and let $D\subset \T N\times\P$ be the pre-image of the diagonal embedding of T$N$ in $\mathrm{T}N\times\mathrm{T}M$ (via $\pi$ and anchor).  In other words,
\[ C = \{(x,\alpha,\xi) \in \T H\times\A\,|\, x = \a_{\A}(\xi) \}\]
and
\[ D = \{(y,\eta,\zeta)\in \T N\times\P\,|\, \mathrm{T}\pi(y) = \a_{\P}(\zeta)\}.\]

Both $C$ and $D$ are coisotropic (see \cite{capg}, Prop 2.7), and the module action restricts to an action of $C$ on $D$:
\begin{equation}\label{prlgCD} (x,\u_{\mathrm{T}^*M}(\eta),\xi)\circ(y,\eta,\zeta) = (x\cdot y, \eta, \xi\circ\zeta) \end{equation}
with $\s(\xi)=\u(\zeta)$.
Hence we get a $CA$-groupoid action of $C/C^\perp$ on $D/D^\perp$, but $C/C^\perp = {\mathrm{id}_H}^!\A=\A$, and $D/D^\perp = \pi^!\P$.  Thus, if $\P$ is an $\A$-module such that the groupoid action is a Courant morphism, then $\pi^!\P$ is a $CA$-module for $\A$ as well.
\end{example}

\subsection{Dirac Lie Group Structures}\label{prelimdirac}
\begin{dfn} (\cite{lmdir}) A \emph{Dirac Lie group structure} on a Lie group $H$ is a multiplicative Manin pair $(\A,E)$ over $H$ such that the multiplication morphism
    \begin{equation} \text{Mult}_\A:\,(\A,E)\times(\A,E)\da (\A,E) \end{equation}
is a morphism of Manin pairs.
\end{dfn}
This definition differs from Jotz \cite{JotzDir} and Ortiz \cite{ortizd}, as these sources take arbitrary multiplicative Manin pairs over $H$, and only look at the case $\A=\T H$.
\begin{prop}\label{mmpd} A multiplicative Manin pair $(\A,E)$ over $H$ defines a Dirac Lie group structure if and only if $E^{(0)} = \A^{(0)}$ (Proposition 2.7, \cite{lmdir}).
\end{prop}
\begin{example} If $M=\mathrm{pt}$, so $\A$ is a quadratic Lie algebra $\g$, the diagonal $\g_\vartriangle\subset\bar{\g}\times\g$ defines a Dirac Lie group structure on $H=\{e\}$.
\end{example}
\begin{example}
For any Lie group $H\gp \mathrm{pt}$, the standard Courant algebroid $\T H\gp \h^*$ is a $CA$-groupoid, given as the direct sum of the tangent group $\mathrm{T}H\gp \text{pt}$ and the standard symplectic groupoid $\mathrm{T}^*H\gp\h^*$.  The only Dirac Lie group structures for $\A=\T H$ are $E=\mathrm{Gr}(\Pi)$ for Poisson Lie group structures $\Pi$ on $H$ (\cite{JotzDir}, \cite{ortizd}).
\end{example}
\begin{thm} (\cite{lmdir}) Given a Lie group $H$, the Dirac Lie group structures on $H$ are classified by $H$-equivariant Dirac Manin triples $(\d,\g,\h)_\beta$.  Here $\d$ is a Lie algebra with $\g$ and $\h=\text{Lie}(H)$ transverse subalgebras, $\beta$ is a (possibly degenerate) invariant element of $S^2\d$, $\g$ is $\beta$-coisotropic in $\d$, and $H$ acts on $\d$ as by Lie automorphisms integrating the adjoint action of $\h\subset\d$ and restricting to the adjoint action of $H$ on $\h$ (denoted ``Ad" on all of $\d$).
\end{thm}
If $\beta$ is degenerate, or $\g$ is not Lagrangian, we can use $(\d,\g,\h)_\beta$ to construct a new triple $(\q,\g,\r)_\gamma$ which does have a non-degenerate form, and for which $\g$ is Lagrangian.  The first step is constructing the groupoid $\d\times\d^*_\beta\gp\d$, where as a quadratic vector space $\d\times\d^*_\beta = \d\oplus\d^*$ with pairing
    \[ \<(x_1,\alpha_1),(x_2,\alpha_2)\> = \<x_1,\alpha_2\> + \<x_2,\alpha_1\> + \beta(\alpha_1,\alpha_2). \]
and with source, target, and multiplication given by
    \begin{align*}
        &\s(x,\alpha) = x;\quad \t(x,\alpha) = x+\beta^\#(\alpha);\\
        &(x_2+\beta^\#(\alpha_2),\alpha_1)\circ(x_2,\alpha_2) = (x_2,\alpha_1+\alpha_2)
    \end{align*}
(see Chapter 3 for details and analysis of groupoids of this type).  Next, we construct a new triple
    \begin{equation} (\q,\g,\r) = (\s^{-1}(\g)/s^{-1}(\g)^\perp,\g,f^{-1}(\h)) \end{equation}
where $\s$ is the source map $\s:\,\d\times\d_\beta^*\to\d$, and $f:\q\to\d$ is the descent of the target map to the quotient; $\gamma$ is obtained as the dual metric on $\q^*$, and is non-degenerate.  The subalgebra $\g$ is Lagrangian in $\q$. \\
Since $\s:\,E\to \g$ is a fibrewise isomorphism, it can be used to trivialise $E\cong H\times\g$.  The groupoid action of $E$ on $\A$ gives a trivialisation in the following manner:  if $j:\A\to E$ is the projection along $\ker(\t)$, then the map
    \begin{equation} \A\to\q,\,\,x\to j(x)^{-1}\circ x \end{equation}
defines a trivialisation $\A\cong H\times\q$ compatible with the metric.

\begin{prop} (5.1, \cite{lmdir}) Let $(\A,E) = (H\times\q,H\times\g)$ be the Dirac Lie group structure on $H$ associated to the $H$-equivariant Dirac Manin triple $(\d,\g,\h)_\beta$.   The Lie algebra action
    \begin{equation} \rho:\q\to\frak{X}(H), \quad \iota_{\rho(\zeta)}\theta^L_h = \Ad_{h^{-1}}\ph\Ad_hf(\zeta) \end{equation}
where $\theta^L$ is the left Maurer-Cartan form, defines the anchor map for $(\A,E)$ as action Courant/Lie algebroids.  The groupoid structure is given by
    \begin{equation} \s(h,\xi) = \pr^*(\xi);\quad \t(h,\xi) = h\bullet(1-\pr)(\xi) \end{equation}
where $h\bullet\zeta = (1-\ph)\Ad_h\zeta$ for $\zeta\in\g$, $\pr$ is the projection to $\r$ along $\g$, and $\pr^*$ is the projection to $\g$ along $\r^\perp$.  The map $f^*:\,\d^*\to\q$ isomorphically takes $\g^*\cong\ann(\h)$ to $\r^\perp$, and so allows us to define an action of $H$ on $\r^\perp$, also denoted $\bullet$ via:
    \[ h\bullet f^*(\mu) = f^*(\Ad_h\mu),\quad h\in H,\,\mu\in\ann(\h) \]
Given composable elements, the groupoid multiplication is given by
    \begin{equation}(h_1,\zeta_1)\circ(h_2,\zeta_2) = (h_1h_2,\,\zeta_2 + h_2^{-1}\bullet(1-\pr^*)\zeta_1). \end{equation}
\end{prop}

\chapter{Metrized Linear Groupoids, Homogeneous Spaces}
For any Dirac Lie group structure $(\A,E)\cong(H\times\q,H\times\g)$, the identity fibre $\q$ inherits a groupoid structure by restricting the groupoid multiplication to $h=e$. As we shall see in Chapter 5, the global structure of any Dirac homogeneous space being given $(\P,L)\cong (H\times_K\p, H\times_K\l)$  gives a ($K$-equivariant) module structure to the identity coset fibres $(\p,\l)$.  Hence, the most reasonable place to begin the analysis is with linear groupoids $\q\gp\g$ having metrics compatible with the groupoid structure, and their homogeneous spaces.
\section{Metrized Linear Groupoids}
\begin{dfn} A groupoid $\q\rightrightarrows\q^{(0)}$ is called a \emph{linear groupoid} if $\q$ and $\q^{(0)}$ are vector spaces, and all associated structure and multiplication maps are linear.
\end{dfn}
\begin{rmrk} In particular, linear groupoids are a special case of $VB$-groupoids, over $M=\mathrm{pt}$.  Additionally, linear groupoids are a special case of Baez-Crans \emph{2-vector spaces}, being a category inside the category of vector spaces; see \cite{baezcrans}, \cite{Pradinesv}, \cite{alfraj}.
\end{rmrk}

\begin{lem}\label{linmul} (See Lemma 6, \cite{baezcrans}) The groupoid multiplication of a linear groupoid $\q\gp\q^{(0)}$ is determined by the source and target maps, given by
    \begin{equation}\label{linmuleq} \xi\circ \eta = \eta + (1-\s)\xi \end{equation}
for $\xi,\eta\in\q$ with $\s(\xi)=\t(\eta)$.
\begin{proof}
Given such a $\q\rightrightarrows\q^{(0)}$, let $\xi,\eta\in\q$ such that $\s(\xi)=\t(\eta)$. The term $(1-\s)\xi\circ 0$ is defined since $(1-\s)\xi$ is in the kernel of the source map; it equals $(1-s)\xi$ since $0\in\q^{(0)}$ is a unit for the groupoid multiplication. Hence:
\begin{align*}
\xi\circ\eta &= (\s(\xi) + (1-\s)\xi))\circ(\eta + 0)\\
             &= \s(\xi)\circ\eta + (1-\s)\xi\circ 0 \\
             &= \eta + (1-\s)\xi.\qedhere
\end{align*}
\end{proof}
\end{lem}

\begin{thm} (\cite{baezcrans}) Linear groupoids $\q\gp\g$ are classified by a vector space $\q$, a subspace $\g\subseteq\q$, and two projections  $\s,\t:\,\q\to\g$.
\begin{proof}
Given a linear groupoid, the data $(\q,\g,\s,\t)$ is recovered automatically.  Since the multiplication is determined by $\s,\t$, it follows that each linear groupoid produces a set of this data unique to it (i.e.: no two linear groupoids produce the same set of data).\\
Given any set of data $(\q,\g,\s,\t)$, we construct the groupoid $\q\gp\g$ with $\s,\t$ as the source and target maps, and multiplication defined by Equation \eqref{linmuleq}.  This is associative, as for all $(\xi_1,\xi_2,\xi_3)\in \q^{(3)}$,
    \begin{align*} \xi_1\circ(\xi_2\circ\xi_3) &= \xi_1\circ(\xi_3+(1-s)\xi_2) \\
                            &= \xi_3 + (1-s)\xi_2 + (1-s)\xi_1 \\
                            &= \xi_3 + (1-s)(\xi_2 + (1-s)\xi_1)\\
                            &= (\xi_1\circ\xi_2)\circ\xi_3.\qedhere
    \end{align*}
\end{proof}
\end{thm}

\begin{dfn}\label{multmet} A \emph{metrized linear groupoid} $(\q\gp\q^{(0)},\<\cdot,\cdot\>)$ is a linear groupoid $\q\gp\q^{(0)}$, with a metric $\<\cdot,\cdot\>$ on $\q$ compatible with the groupoid multiplication, in the sense that:
    \[ \< \xi_1\circ\eta_1,\xi_2\circ\eta_2\> = \<\xi_1,\xi_2\>+\<\eta_1,\eta_2\>\]
for all composable $\xi_i,\eta_i\in \q$. Such a metric is called a \emph{multiplicative metric}.
\end{dfn}
In keeping with the notation in \cite{lmdir} for the Dirac Manin triples, we will use $\g$ to denote $\q^{(0)}$.
\begin{prop}\label{lgperp} Let $(\q\gp\g, \<\cdot,\cdot\>)$ be a metrized linear groupoid. Then $\g$ is Lagrangian in $\q$, and $\ker(\s)=\ker(\t)^\perp$.
\begin{proof}
For any $\zeta\in\g$, we have the identity $\zeta=\zeta\circ\zeta$.  Hence,
    \[ \<\zeta,\zeta\> = \<\zeta\circ\zeta,\zeta\circ\zeta\> = \<\zeta,\zeta\>+\<\zeta,\zeta\> \]
meaning $\g$ is isotropic $(\g\subseteq\g^\perp)$.  \\
For all $\xi\in\q$, and $\eta\in\g^\perp$
    \begin{align*}
        \<\xi,\eta-\s(\eta)\> &= \<\xi,\eta\> - \<\xi,\s(\eta)\>\\
                              &=\<\t(\xi)\circ\xi,\eta\circ\s(\eta)\> - \<\xi,\s(\eta)\>\\
                              &= \<\t(\xi),\eta\> + \<\xi,\s(\eta)\> - \<\xi,\s(\eta)\>\\
                              &=0
    \end{align*}
and so $\eta = \s(\eta)\in\g$, meaning $\g^\perp\subseteq\g$.  Thus, we conclude $\g=\g^\perp$.\\
If $x\in\ker(\s),\,y\in\ker(\t)$, then $x=x\circ0$ and $y=0\circ y$.  Taking inner products,
    \[ \<x,y\> = \<x\circ 0,0\circ y \> = \<x,0\>+\<0,y\> = 0. \]
This shows $\ker(\s)\subseteq\ker(\t)^\perp$.  Equality follows since $\dim(\ker\s)=\dim(\ker\t)=\dim\g$.
\end{proof}
\end{prop}

\begin{thm}\label{mlincl} Metrized linear groupoids $(\q\rightrightarrows\g,\langle,\rangle)$ are classified by a choice of vector space $\g$ together with a element $\lambda\in S^2\g$.
\begin{proof}   Let $(\q\gp\g,\<,\>)$ be a metrized linear groupoid.  The metric on $\q\cong\q^*$ can be identified with an element of $S^2\q$.  Let $\lambda\in S^2\g$ be its image under the target map $\t$.  Hence for any $\alpha_i\in \g^*$, we have
    \[ \lambda(\alpha_1,\alpha_2) = \<\t^*(\alpha_1),\t^*(\alpha_2)\>, \]
from which we conclude $\t\circ\t^*|_{\g^*} = \lambda^\#:\,\g^*\to\g$. \\
As $\t$ is surjective, $\t^*:\,\g^*\to \q$ is injective, with image $\mathrm{ran}(\t^*) = \ker(\t)^\perp=\ker(\s)$ (Proposition \ref{lgperp}) and so we identify $\ker(\s)$ with $\g^*$. In the resulting decomposition $\q=\g\oplus\g^*$, the metric becomes
    \begin{equation}\label{linme} \<(\zeta_1,\alpha_2),(\zeta_2,\alpha_2)\> = \< \zeta_1,\alpha_2\>+\<\zeta_2,\alpha_1\> + \lambda(\alpha_1,\alpha_2), \end{equation}
and we will use the notation $\g\times\g_\lambda^*$ to denote $\q=\g\oplus\g^*$ with this metric.  Since $\ker(\t) = \ker(\s)^\perp \cong (\g^*)^\perp$, elements of $\ker(\t)$ are of the form $\alpha-\lambda^\#(\alpha)$ for $\alpha\in\g^*$.  The source and targets are uniquely determined by the fact that they are the identity on $\g$, and their kernels are $\g^*$ and $(\g^*)^\perp$ respectively.  From this, together with the result of Lemma \ref{linmul}, the groupoid structure is given by
    \begin{equation}\label{lingpstr}
    \begin{split}  \s(\zeta,\alpha) = \zeta,\quad \t(\zeta,\alpha) = \zeta + \lambda^\#(\alpha)\\
                      ( \zeta_1,\alpha_1)\circ (\zeta_2,\alpha_2) = (\zeta_2,\alpha_1+\alpha_2)
    \end{split}
    \end{equation}
for $\zeta_1 = \t(\zeta_2,\alpha_2)$.  From this we can conclude that each metrized linear groupoid $(\q\gp\g,\<\cdot,\cdot\>)$ determines a pair $(\g,\lambda)$ unique to it amongst metrized linear groupoids.\\
All that remains is to show that every $(\g,\lambda)$ arises this way.  Given any $\g$ and $\lambda\in S^2\g$, construct $\q:=\g\oplus\g^*$ with groupoid structure maps for $\q^{(0)}=\g$ given in Equation \eqref{lingpstr}, and metric given by Equation \eqref{linme}.  Thus, the only thing left to show is that this metric is multiplicative.  By direct computation, for any $\alpha_i\in\g^*,\zeta\in \g$:
    \begin{align*}
        &\< ( \zeta+\lambda^\#(\alpha_2),\alpha_1)\circ (\zeta,\alpha_2), ( \zeta+\lambda^\#(\alpha_2),\alpha_1)\circ (\zeta,\alpha_2)\>\\
         &= \< (\zeta,\alpha_1+\alpha_2),(\zeta,\alpha_1+\alpha_2)\> \\
         &= 2\<\zeta,\alpha_1 + \alpha_2\> + \lambda(\alpha_1+\alpha_2,\alpha_1+\alpha_2,)\\
         &= 2\<\zeta+\lambda^\#(\alpha_2),\alpha_1\> + 2\<\zeta,\alpha_2\> + \lambda(\alpha_1,\alpha_1)+ \lambda(\alpha_2,\alpha_2)\\
         &= \<\zeta+\lambda^\#(\alpha_2),\alpha_1),\zeta+\lambda^\#(\alpha_2),\alpha_1)\> + \< (\zeta,\alpha_2), (\zeta,\alpha_2) \>.\qedhere
    \end{align*}
\end{proof}
\end{thm}

\begin{prop}\label{qggroup} If $(\q\gp\g,\<\cdot,\cdot\>)$ is a metrized linear groupoid, then $\q/\g$ is a group, with group multiplication defined by
    \begin{equation} \overline{\xi_1}\cdot\overline{\xi_2} = \overline{\xi_1\circ\xi_2} \end{equation}
for lifts $\xi_1,\xi_2\in\q$ with $\s(\xi_1)=\t(\xi_2)$.
\begin{proof}
If $\xi_2\in\q$ is any lift of $\overline{\xi_2}$, then there is a unique lift $\xi_1$ of $\overline{\xi_1}$ with $\s(\xi_1)=\t(\xi_2)$.  If $\xi_2'$ is another lift, we can write $\xi_2' = \xi_2 + (\xi_2'-\xi_2)$.  The associated lift is then $\xi_1' = \xi_1 + \t(\xi_2'-\xi_2)$, and we compute
    \begin{align*} \xi_1'\circ \xi_2' - \xi_1\circ\xi_2 &= \xi_1\circ \xi_2 + \t(\xi_2'-\xi_2)\circ(\xi_2'-\xi_2) -\xi_1\circ\xi_2\\
                                                        &= \xi_2'-\xi_2.
    \end{align*}
Since $\xi_2'-\xi_2\in \g$, this means that $\overline{\xi_1\circ\xi_2} = \overline{\xi_1'\circ\xi_2'}$, and so the multiplication is well defined.  It is associative since the groupoid multiplication is, the identity is 0, and the inverses are given by $(\overline{\xi})^{-1} = -\overline{\xi}$.\\
Since $\q/\g\cong\g^*$, and considering Equation \eqref{lingpstr}, this group is just $\g^*$ viewed as a vector space.
\end{proof}
\end{prop}

\section{Linear Groupoid Modules, Homogeneous Spaces}
\begin{dfn} Given a linear groupoid $\q\gp\g$, a \emph{$\q$-module} is a groupoid module $\u:\,\p\to\g$ (Section \ref{gmodules}) such that $\p$ is a vector space, with $\u$ and the groupoid action being linear maps.
\end{dfn}

\begin{thm}\label{linmnm} The $\q$-modules for a given linear groupoid $\q\gp\g$ are classified by a vector space $\p$, and a pair of maps $\u:\,\p\to\g$, and $j:\,\ker(\s)\to\p$ such that the following diagram commutes
\begin{center}
\begin{tikzpicture}

  \node (A)   {$\ker(\s)$};
  \node (B) [right of=A, node distance=1.5cm] {$\p$};
  \node (C) [below of=A, node distance=1.3cm] {$\g$};

  \draw[->] (A) to node {$j$} (B);
  \draw[->] (A) to node [swap] {$\t$} (C);
  \draw[->] (B) to node {$\u$} (C);
  \end{tikzpicture}
\end{center}
For any $\xi\in\q,\,x\in\p$ with $\s(\xi)=\u(x)$, the action is given by
    \begin{equation}\label{linmodwom} \xi\circ x = x + j((1-s)\xi). \end{equation}
\begin{proof}
Let $\u:\,\p\to\g$ be a $\q$-module.  For any $\xi\in\q, x\in\p$ such that $\s(\xi)=\u(x)$, we have
    \[ \xi\circ x = (\s(\xi) + (1-\s)\xi)\circ (x + 0) = \s(\xi)\circ x + (1-\s)\xi\circ 0 = x + (1-\s)\xi\circ 0 \]
and hence the groupoid action is determined by the action of $\ker(\s)$ on $0\in \p$, i.e.: the induced linear map
    \[ j:\,\ker(\s)\times \{0\} \to \p. \]
Since $\u(\xi\circ x) = \t(\xi)$, the $j$ map must satisfy the composition relation $\u\circ j = \t$.  The groupoid action can then be expressed as
    \[ \xi\circ x = x + j((1-\s)\xi). \]
Since the ability to compose relies only on $\u$, and the action depends only on $j$, it is clear that distinct $\q$-modules give distinct data $(\p,\u,j)$.\\
Given any set of data $(\p,\u,j)$, we define a $\q$-module structure on $\p$ with moment map $\u$, and groupoid action defined in Equation \eqref{linmodwom}.  This is a groupoid action, as it satisfies
    \[ \u(\xi\circ x) = \u(x + j((1-\s)\xi)) = \u(x) + \t((1-\s)\xi) = \s(\xi) + \t(\xi) -\s(\xi) = \t(\xi), \]
and is associative
    \begin{align*} \xi_1\circ(\xi_2\circ x) &= \xi_1\circ (x+j((1-\s)\xi_2)) \\
                                            &= x + j((1-\s)\xi_2) + j((1-\s)\xi_1)\\
                                            &= x + j(\xi_1\circ\xi_2 -\s(\xi_2))\\
                                            &= x + j(1-\s(\xi_1\circ\xi_2))\\
                                            &= (\xi_1\circ\xi_2)\circ x
    \end{align*}
using Equation \eqref{linmuleq}, as $\s(\xi_1\circ\xi_2)=\s(\xi_2)$.
\end{proof}
\end{thm}

\begin{dfn} Given a metrized linear groupoid $(\q\gp\g, \<\cdot,\cdot\>)$, a \emph{metrized $\q$-module} is a metrized vector space $(\p,\<\cdot,\cdot\>)$ together with a linear moment map $\u:\,\p\to\g$ which makes $p$ into a groupoid space for $\q$ .  The groupoid action is compatible with the metrics in the sense that for any composable $\xi_i\in\q,x_i\in \p$
    \[ \<\xi_1\circ x_1,\xi_2\circ x_2\> = \<\xi_1,\xi_2\> + \<x_1,x_2\>. \]
\end{dfn}
\begin{prop}\label{uustar}
Let $(\q\gp\g,\<\cdot,\cdot\>)$ be a metrized linear groupoid, and let $(\p,\<\cdot,\cdot\>,\u)$ be a metrized $\q$-module.  Then the groupoid action is given by
    \begin{equation}\label{uustarform} (\zeta,\alpha)\circ x = x+\u^*(\alpha), \end{equation}
for $\s(\zeta,\alpha)=\u(x)$.  The moment map satisfies $\u\circ\u^*=\lambda^\#$.
\begin{proof}
Elements $x\in \p, \,(\zeta,\alpha)\in \q$ are composable if and only if
    \[ \s(\zeta,\alpha) = \zeta = \u(x). \]
In this case, we can compute
\begin{align*}
(\u(x),\alpha)\circ x &= \big((\u(x),0)+(0,\alpha)\big)\circ (x+0)\\
                       &= (\u(x),0)\circ x + (0,\alpha)\circ 0 \\
                       &= x + (0,\alpha)\circ 0.
\end{align*}
Hence, the groupoid action is determined by how $\g^*$ acts on $0\in \p$.  If $(0,\alpha)\circ 0 = y$, then we must have that $\u(y) = \t(0,\alpha) = \lambda^\#(\alpha)$, and so
\begin{align*}
\< x, y \> &= \<(\u(x),0)\circ x, (0,\alpha)\circ 0\>\\
&= \<(\u(x),0),(0,\alpha)\> + \<x,0\>\\
&= \<\u(x),\alpha\>  \\
&= \<x, \u^*(\alpha)\>.
\end{align*}
Since this is true for every $x\in \p$, we conclude that $y = \u^*(\alpha)$, and hence
\begin{equation*}
(\u(x),\alpha)\circ x = x + \u^*(\alpha) \qquad\qquad \forall\,\,x\in \p, \alpha \in \g^*.
\end{equation*}
Lastly, as $\u(y) = \lambda^\#(\alpha)$, and $y = \u^*(\alpha)$, we get $\u\circ\u^*=\lambda^\#$.
\end{proof}
\end{prop}

\begin{thm}\label{linmods} For a given metrized linear groupoid $(\q\rightrightarrows\g,\<\cdot,\cdot\>)$, the metrized $\q$-modules are classified by a metrized vector space $(\p,\<\cdot,\cdot\>)$ together with a linear map $\u:\,\p\to\g$ such that $\u\circ\u^*=\lambda^\#$.
\begin{proof}
Given a metrized $\q$-module, we obtain the unique data $(\p,\u,j)$ by Theorem \ref{linmnm}; Proposition \ref{uustar} gives $j=\u^*$ which satisfies the requirement $\u\circ\u^* =  \t|_{\g^*} = \lambda^\#$, coming from the compatibility of the groupoid action with the metrics.\\
Given any such set of data $(\p,\u)$ with $\u\circ\u^*=\lambda^\#$, we construct the metrized $\q$-module $\u:\,\p\to\g$ with groupoid action defined by
    \[ (\zeta,\alpha)\circ x = x + \u^*(\alpha) \]
with $\s(\zeta,\alpha)=\u(x)$.  This is a groupoid action by Theorem \ref{linmnm}; all that remains is to show that this action is compatible with the metrics:
\begin{align*}
\<(\u(x),\alpha)\circ x,(\u(x),\alpha)\circ x\> &= \<x + \u^*(\alpha),x + \u^*(\alpha)\> \\
&= \<x,x\> + 2\<x,\u^*(\alpha)\> + \<\u^*(\alpha), \u^*(\alpha)\> \\
&= \<x,x\> + 2\<\u(x),\alpha)\> + \<\alpha, \lambda^\#(\alpha)\> \\
&= \<x,x\> + 2\<\u(x),\alpha)\> + \lambda(\alpha,\alpha) \\
&= \<x,x\> + \<(\u(x),\alpha),(\u(x),\alpha)\>.\qedhere
\end{align*}
\end{proof}
\end{thm}

\begin{prop}\label{qgaction} If $(\q\gp\g,\<\cdot,\cdot\>)$  is a metrized linear groupoid, and $(\p,\l,\<\cdot,\cdot\>,\u)$ is a metrized $\q$-module together with a subspace $\l\subset \p$, then the group $\q/\g$ acts on $\p/\l$ via the formula
    \[ \overline{\xi}\cdot\overline{x} = \overline{\xi\circ x} \]
for lifts $\xi\in\q, x\in \p$ with $\s(\xi)=\u(x)$.
\begin{proof}
Let $\overline{\xi}\in\q/\g$ and $\overline{x}\in \p/\l$.  If $x\in \p$ is a lift of $\overline{x}$, then there is a unique lift $\xi\in\q$ of $\overline{\xi}$ with $\s(\xi)=\u(x)$.  Let $x'$ be another lift of $\overline{x}$.  As we can write $x'=x+(x'-x)$, the associated lift of $\overline{\xi}$ becomes $\xi'=\xi + \u(x'-x)$.  Thus, we see
    \begin{align*} \xi'\circ x' - \xi\circ x &= \xi\circ x + \u(x'-x)\circ (x'-x) -\xi\circ x\\
                                             &= x-x',
    \end{align*}
and since $x-x'\in \l$, this means $\overline{\xi\circ x}$ is well defined.  Given $\overline{\xi_1},\overline{\xi_2}\in\q/\g$, and $x\in\p/\l$, we have
    \begin{align*} \overline{\xi_1}\cdot\big(\overline{\xi_2}\cdot \overline{x}\big) &= \overline{\xi_1} \cdot \overline{\xi_2\circ x} \\
                            &= \overline{\xi_1\circ(\xi_2\circ x)}\\
                            &= \overline{(\xi_1\circ\xi_2)\circ x} \\
                            &= \overline{\xi_1\circ\xi_2}\cdot \overline{x}\\
                            &= \big(\overline{\xi_1}\cdot\overline{\xi_2}\big)\cdot \overline{x}
    \end{align*}
since $\u(\xi_2\circ x) = \t(\xi_2)$.  Since the identity element of $\q/\g$ is the coset $\overline{0}=\g$, it clearly acts trivially.
\end{proof}
\end{prop}
Since $\q/\g = \g^*$, then by Proposition \ref{uustar} we have
\begin{equation}\label{gstarac}
\alpha\cdot \overline{x} = \overline{ (\u(x),\alpha)\circ x} = \overline{ x + \u^*(\alpha)} = \overline{x} + \overline{\u^*(\alpha)}
\end{equation}
for any lift $x$ of $\overline{x}$.  \\
\begin{dfn} Let $\g$ be a vector space, and $\lambda\in S^2\g$. A subspace $\mathfrak{v}\subset\g$ is called $\lambda$-coisotropic if $\lambda$ restricted to $\ann(\mathfrak{v})\subset\g^*$ is identically $0$.  Equivalently, $\lambda^\#(\ann(\mathfrak{v}))\subset \mathfrak{v}$.
\end{dfn}




\begin{prop}\label{lininj} Suppose there is a metrized linear groupoid action $(\q\rightrightarrows\g,\<\cdot,\cdot\>)\act(\p,\l,\<\cdot,\cdot\>)$. Then the induced group action $(\q/\g)\act(\p/\l)$ is transitive if and only if the map $\u\vert_\l$ is an injection $\l\hookrightarrow\g$.  If the action is transitive, then $\u(\l)\subseteq\g$ is a $\lambda$-coisotropic subspace.
\begin{proof}
Let $\u_\l:\,\l\to\g$ be the restriction.  By Equation \eqref{gstarac}, the action of $\g^*=\q/\g$ on $\l^*\cong\p/\l$ is by translation via $\u_\l^*:\,\g^*\to\l^*$.  Hence, the group action is transitive if and only if $\u_\l^*$ is surjective, or equivalently, if and only if $\u_l$ is injective.\\
If $\u:\,\l\hookrightarrow\g$ is injective, we have that for all $x\in\l,\alpha\in\ann_{\g^*}(\u(\l))$:
\begin{equation*}
 0 = \<\alpha,\u(x)\> = \<\u^*(\alpha),x\>.
\end{equation*}
Since $\l\subset \p$ is Lagrangian, this means $\u^*(\alpha)\subset\l$ for all $\alpha\in\ann_{\g^*}(\u(\l))$.  Hence
    \[ \u\circ\u^*(\alpha) = \lambda^\#(\alpha)\subset \u(\l)\subset \g,\]
i.e.: $\lambda^\#(\ann_{\g^*}(\l))\subset\l$, meaning $\l$ is $\lambda$-coisotropic.

\end{proof}
\end{prop}


\begin{dfn}
Given a metrized linear groupoid $(\q\rightrightarrows\g,\<\cdot,\cdot\>)$, an \emph{homogeneous space} for $(\q,\g)$ is defined to be a metrized $\q$-module $(\p,\l,\<\cdot,\cdot\>,\u)$ with $\l\subset \p$ Lagrangian, such that $\q/\g\act \p/\l$ transitively.
\end{dfn}

\begin{thm}\label{qlinhom}
For a metrized linear groupoid $(\q\rightrightarrows\g,\<\cdot,\cdot\>)$ with corresponding $\lambda\in S^2\g$, the homogeneous spaces for $(\q,\g)$ are classified by a choice of a $\lambda$-coisotropic subspace $\l\subset\g$.\\
The homogeneous spaces have the normal form
    \[ (\p,\l) = ( \s^{-1}(\l)/\s^{-1}(\l)^\perp,\l) \]
with $\u$ given by the descent of the target map on $\s^{-1}(\u(\l))$ to the quotient.
\begin{proof}
Given any $\lambda$-coisotropic $\l\subset\g$, we construct an homogeneous space in the following way:  define $C:=\s^{-1}(\l) = \l\times\g^*$, which gives
    \[ C^\perp = \mathrm{Gr}(-\lambda^\#)|_{\ann(\l)} \]
as $\l^\perp = \g\times\ann(\l)$, and $(\g^*)^\perp = \mathrm{Gr}(-\lambda^\#)$.  For all $\alpha\in \ann(\l)$, $\lambda^\#(\alpha)\in\l$ since $\l$ is $\lambda$-coisotropic, and thus $C^\perp\subset C$.  Taking the quotient, we have
\begin{align*}
\dim(C/C^\perp) &= \dim C - \dim C^\perp \\
                &= \dim\l + \dim\g - (\dim\g -\dim\l)\\
                &= 2\dim\l.
\end{align*}
and since $(\l\oplus 0)\cap C^\perp = 0$, we can conclude that $\l\subset C/C^\perp$ is Lagrangian.  Since $\ker(\s)\subset C$, this implies $C^\perp\subset\ker(\s)^\perp=\ker(\t)$, and so the target map $\t:C\to\g$ descends to a map $\u:(C/C^\perp)\to\g$.  As the target map is a moment map for the action of $\q$ on itself by left multiplication, we conclude from Proposition \ref{uustar} that $\u\circ\u^*=\lambda^\#$.  Since $\u$ restricted to $\l$ is the identity, we conclude from Theorem \ref{linmods} and Proposition \ref{lininj} that $(C/C^\perp,\l,\<\cdot,\cdot\>,\u)$ is an homogeneous space for $(\q,\g)$, and hence each $\lambda$-coisotropic $\l\subset\g$ determines an homogeneous space $(\p,\l)$.\\
Take any homogeneous space $(\p,\l)$ for $(\q,\g)$. From the groupoid action, we get the vector space $(\p,\l,\<\cdot,\cdot\>)$ and the map $\u:\,\p\to\g$ such that $\u\circ\u^*=\lambda^\#$ (Theorem \ref{linmods}).  From the transitive group action, we get -- via Proposition \ref{lininj} -- that $\u$ embeds $\l$ as a $\lambda$-coisotropic subspace of $\g$.  To show that this $\l\subset\g$ is unique to $(\p,\l)$, construct the homogeneous space $(\p',\l) = (C/C^\perp,\l)$ for $C=\s^{-1}(\u(\l))$.  Consider the map
    \begin{equation}\label{pprim} f:\,\s^{-1}(\u(\l)) \to \p,\quad (\u(x),\alpha)\mapsto x+\u^*(\alpha).\end{equation}
Since $\u:\,\l\hookrightarrow\g$, thus $\u^*:\,\g^*\twoheadrightarrow\l^*$, and since any complement to $\l$ in $\p$ is identified with $\l^\perp$, this means that the image of $f$ in Equation \eqref{pprim} is all of $\p$.  The map $f$ also preserves inner products as
\begin{align*}
\<x+\u^*(\alpha),x+\u^*(\alpha)\> &= 2\<x,\u^*(\alpha)\> + \<\u^*(\alpha),\u^*(\alpha)\>\\
                                    &= 2\<\u(x),\alpha\> + \lambda(\alpha,\alpha)\\
                                    &= \<(\u(x),\alpha),(\u(x),\alpha)\>
\end{align*}
and hence the kernel of $f$ must be $C^\perp$.  This gives us the short exact sequence:
\begin{center}
$\begin{CD}
0 @>>> C^\perp  @>>> \l\times\g^* @>>> \p @>>> 0.
\end{CD}$
\end{center}
Thus, $\p\cong C/C^\perp=\p'$, and by Theorem \ref{linmods}, the module actions coincide.
\end{proof}
\end{thm}

\begin{rmrk}\label{qlinhomH} This theorem holds ``$K$-equivariantly'':  if a Lie group $K$ acts on $\q$ by groupoid automorphisms, and on $\p$ such that $\u$ is $K$-equivariant, and the action preserves the metrics of each, then the homogeneous spaces are classified by $\lambda$-coisotropic subspaces $\l\subseteq\p$ which are $K$-stable. For a given $(\q,\g)$-homogeneous space $(\p,\l)$ in this set-up, $\l\subset\p$ being Lagrangian means it must be $K$-stable, since the metric is $K$-invariant.  The map $f: \s^{-1}(\u(\l))\to\p$ defined by $(\u(x),\alpha)\mapsto x+\u^*(\alpha)$ is $K$-invariant, with kernel $\s^{-1}(\u(\l))^\perp$, and hence
    \[ \bar{f}:\,\s^{-1}(\u(\l))/\s^{-1}(\u(\l))^\perp \xrightarrow[]{\cong} \p, \qquad \overline{(\u(x),\alpha)}\mapsto x + \u^*(\alpha).\]
In other words, $\p$ is equivariantly isomorphic to $\s^{-1}(\u(\l))/\s^{-1}(\u(\l))^\perp$.
\end{rmrk}
\chapter{Vacant $VB$- and $LA$-Groupoid Modules}
Given any Dirac Lie group $(H,\A,E)$, the Dirac structure $E$ is a vacant $LA$-groupoid (Lemma 3.1, \cite{lmdir}).  In this chapter, we classify the vacant groupoids modules over $H$-homogeneous spaces $H/K$ as a preliminary step to classifying the Dirac homogeneous spaces.
\section{Vacant $VB$-groupoids}\label{subsec:vacantvb}

Let $H$ denote a Lie group.  As in Definition \ref{vacantdef}, a $VB$-groupoid over $H$
\begin{center}
\begin{tikzpicture}

  \node (G)   {$E$};
  \node (H) [right of=G, node distance=1.5cm] {$E^{(0)}$};
  \node (I) [below of=G, node distance=1.3cm] {$H$};
  \node (J) [right of=I, node distance=1.5cm] {pt};

  \path[->]
([yshift= 2pt]G.east) edge node[above] {} ([yshift= 2pt]H.west)
([yshift= -2pt]G.east) edge node[below] {} ([yshift= -2pt]H.west);
\path[->]
([yshift= 2pt]I.east) edge node[above] {} ([yshift= 2pt]J.west)
([yshift= -2pt]I.east) edge node[below] {} ([yshift= -2pt]J.west);
  \draw[->] (G) to node {} (I);
  \draw[->] (H) to node [swap] {} (J);
  \end{tikzpicture}
\end{center}
is called \emph{vacant} if it has the property that $E^{(0)}= E\vert_e$ (i.e.: $\mathrm{core}(E)=0$).  Following the Dirac convention in \cite{lmdir}, let $\g:=E^{(0)}$.
\begin{prop}\label{vtriv} The vacant $VB$-groupoid $E$ is the trivial vector bundle $H\times \g$, trivialised via the source map.
\begin{proof}
Consider the map $\psi: E\to H\times \g$, $v \mapsto (h,\s(v))$ for $v\in E_h$.  Since $\s$ is fibrewise surjective, this means for any $h\in H$, $s: E_h \to \g$ must be a bijection, as $E_e=\g$ by the vacant condition.  Hence, $\psi$ is a bundle isomorphism.
\end{proof}
\end{prop}
Using this trivialisation, we have $\s(h,\zeta)=\zeta$ for any $(h,\zeta)\in E_h$.  The groupoid multiplication, along the multiplication of $H$, is then given by
\begin{equation}\label{vacantgrm}
(h_1,\zeta_1)\circ(h_2,\zeta_2)=(h_1h_2,\zeta_2)
\end{equation}
for $\s(h_1,\zeta_1)=\t(h_2,\zeta_2)$.
\begin{prop}\label{vbullet} The target map is a group action of $H$ on $\g$:
    \[ H\times \g \to \g,\quad (h,\zeta) \mapsto \t(h,\zeta). \]
\end{prop}
\begin{proof}
Consider the groupoid multiplication $(h_1,\t(h_2,\zeta))\circ(h_2,\zeta) = (h_1h_2,\zeta)$.  Comparing the targets of each side, we get
\begin{equation*}
\t(h_1,\t(h_2,\zeta)) = \t(h_1h_2,\zeta).
\end{equation*}
Since $(e,\zeta)$ can be viewed as an element of $\g$, we have $\t(e,\zeta)=\zeta$.  In other words, if we define $\t(h,\zeta):=h\bullet \zeta$, then
    \begin{equation}\label{vgpm} h_1\bullet(h_2\bullet \zeta) = (h_1h_2)\bullet \zeta;\quad e\bullet \zeta = \zeta. \end{equation}
Hence, the target map is a group action of $H$ on $\g$.
\end{proof}

From now on, we will use $\bullet$ to denote the action defined by the target.

\begin{thm}  Vacant $VB$-groupoids $E$ over a Lie group $H$ are classified by a choice of vector space $\g$, and an action of $H$ on $\g$.
\begin{proof}
This is just a special case of Mackenzie's classification of vacant $LA$-groupoids in \cite{mack1}, with $\a=0$, $[\cdot,\cdot]=0$.
\begin{center}
\begin{tikzpicture}

  \node (G)   {$H\times\g$};
  \node (H) [right of=G, node distance=1.5cm] {$\g$};
  \node (I) [below of=G, node distance=1.3cm] {$H$};
  \node (J) [right of=I, node distance=1.5cm] {pt};

  \path[->]
([yshift= 2pt]G.east) edge node[above] {} ([yshift= 2pt]H.west)
([yshift= -2pt]G.east) edge node[below] {} ([yshift= -2pt]H.west);
\path[->]
([yshift= 2pt]I.east) edge node[above] {} ([yshift= 2pt]J.west)
([yshift= -2pt]I.east) edge node[below] {} ([yshift= -2pt]J.west);
  \draw[->] (G) to node {} (I);
  \draw[->] (H) to node [swap] {} (J);
  \end{tikzpicture}
\end{center}
with source given by projection to $\g$, target given by the $H$ action, and composition
    \[(h_1,h_2\bullet\zeta)\circ(h_2,\zeta) = (h_1h_2,\zeta)\]
yielding the desired $(\g,\bullet)$.
\end{proof}
\end{thm}
\section{Vacant $VB$-Groupoid Modules over $H/K$}
\begin{dfn} Let $E\gp\g$ be a vacant $VB$-groupoid, for $E$ over $H$.  An $E$-module is a vector bundle $\P\to M$, with an $H$-action on $M$, a bundle map $\u:\,\P\to \g$, and groupoid action
    \[ E _s\times_\u \P \to \P \]
such that the graph of the action is a vector sub-bundle of $\P\times E\times\P$.  This action is represented by the diagram:
\begin{center}
\begin{tikzpicture}
  \node (G)  {$E$};
  \node (H) [right of=G,node distance=1.5cm] {$\g$};
  \node (I) [below of=G,node distance=1.3cm] {$H$};
  \node (J) [right of=I,node distance=1.5cm] {pt};

  \node (ac) [node distance=0.55cm, right of=H, below of=H] {$\act$};

  \node (K) [right of=H, node distance=1.1cm] {$\P$};
  \node (L) [right of=K,node distance=1.5cm] {$\g$};
  \node (M) [below of=K,node distance=1.3cm] {$M$};
  \node (N) [right of=M,node distance=1.5cm] {pt};

  \draw[->] (G) to node {} (I);
  \draw[->] (H) to node [swap] {} (J);
  \path[->]
([yshift= 2pt]G.east) edge node[above] {} ([yshift= 2pt]H.west)
([yshift= -2pt]G.east) edge node[below] {} ([yshift= -2pt]H.west);
\path[->]
([yshift= 2pt]I.east) edge node[above] {} ([yshift= 2pt]J.west)
([yshift= -2pt]I.east) edge node[below] {} ([yshift= -2pt]J.west);
  \path[->]
([yshift= 2pt]G.east) edge node[above] {} ([yshift= 2pt]H.west)
([yshift= -2pt]G.east) edge node[below] {} ([yshift= -2pt]H.west);
\path[->]
([yshift= 2pt]I.east) edge node[above] {} ([yshift= 2pt]J.west)
([yshift= -2pt]I.east) edge node[below] {} ([yshift= -2pt]J.west);

  \draw[->] (K) to node {} (M);
  \draw[->] (L) to node {} (N);
  \draw[->] (K) to node {$\u$} (L);
  \draw[->] (M) to node {} (N);
\end{tikzpicture}
\end{center}

\end{dfn}
For this section, we will be concerned with the case that $M=H/K$, an $H$-homogeneous space.

\begin{prop}\label{EandP} Let $E$ be a vacant $VB$-groupoid over $H$, and let $\P$ be an $E$-module over $H/K$ for some closed Lie subgroup $K\leqslant H$. The groupoid action of $E$ on $\P$ induces a group action of $H$ on $\P$, via:
    \begin{equation}\label{HactonPL}
    H\times\P \to \P,\qquad    (h,y) \mapsto (h,\u(y))\circ y.
    \end{equation}
This restricts to a group action of $K$ on $\P_{eK}=:\p$, identifying $\P\cong H\times_K\p$.
\end{prop}
\begin{proof} Given any $y\in \P$ and $h\in H$, there is a unique element of $E_h$ composable with $y$, namely $(h,\u(y))$, and so the proposed action is well defined.  For brevity, we will denote $(h,\u(y))\circ y$ by $h.y$ where appropriate.  First, we note that given any $y\in\P$, $e.y := (e,\u(y))\circ y = y$, since $(e,\u(y))$ is the unit at the point $\u(y)\in \g$. So to conclude it is an action, we will need to show that $(h_1h_2).y = h_1.(h_2.y)$.  Computing, we see
\begin{align*}
h_1.(h_2.y) &= \big((h_1,h_2\bullet\u(y))\circ(h_2,\u(y))\big)\circ y \\
              &= (h_1h_2, \u(y))\circ y \\
              &= (h_1h_2).y
\end{align*}
via the formula for the groupoid multiplication in $E$, Equation \eqref{vgpm}.\\
From the construction, if $y \in \P_{h_2K}$, then for any $h_1\in H$, we have $h_1.y = (h_1,\u(y))\circ y \, \, \in \P_{(h_1h_2)K}$, and hence $K\times \P_{eK} \to \P_{eK}$ under this action.  So with $\p:=\P_{eK}$, we have an action $K\circlearrowright \p$, giving us $\P\cong H\times_K\p$, $y \mapsto [h, h^{-1}.y]$ for $y\in \P_{hK}$.
\end{proof}
\begin{prop}\label{vacfor} Given $E$ and $\P$ as in Proposition \ref{EandP}, $h_i\in H$, $x\in \p$, the groupoid action is given by
    \begin{equation}\label{VBac} (h_1,\xi)\circ[h_2,x] = [h_1h_2,x] \end{equation}
for $\s(h_1,\xi)=\u([h_2,x])$, and the moment map given by $\u([h,x]) = h\bullet \u(x)$ for the $K$-equivariant induced linear map $\u:\,\p\to \g$.
\begin{proof}
For the groupoid action, we directly compute:
\begin{align*}
(h_1,\u([h_1,x]))\circ[h_2,x] &= h_1.[h_2,x] \\
                         &= h_1.(h_2.x) \\
                         &=(h_1h_2).x \\
                         &= [h_1h_2,x].
\end{align*}
Our map $\u:\, \P \to \g$ restricts to a linear map $\u:\, \p \to \g$ via $\u(x):= \u([e,x])$ for $x\in \p$; since $\u:\,\P\to\g$ is $H$-equivariant, $\u:\,\p\to\g$ is $K$-equivariant.  Using the formula that $(h,\u(x))\circ[e,x] = [h,x]$, we get
\begin{equation}\label{vacu}\u([h,x])=\t(h,\u(x)) = h\bullet\u(x).\end{equation}
\end{proof}
\end{prop}

\begin{thm}\label{vacantVBmd} Given a vacant $VB$-groupoid $E$ over a Lie group $H$, and an $H$-homogeneous space $H/K$, the $E$-modules $\P\to H/K$ are classified by a vector space $\p$, a $K$-action $K\act \p$, and a $K$-equivariant linear map $\u:\, \p\to \g$.
\begin{proof}
Beginning with such a $\P$, we obtain the data $\p$ and $\u$ with the $K$ action from the previous analysis in Propositions \ref{EandP} and \ref{vacfor}.  Now using this data, construct $\P':=H\times_K\p$, and extend $\u:\,\p\to \g$ to $\u:E\to \g$ by $\u([h,\xi])=h\bullet \u(\xi)$.  Since $\u$ is $K$-equivariant, this is well defined:  picking another representative $[hk^{-1},k.\xi]$, we have
\begin{align*}  \u([hk^{-1},k.\xi]) &= (hk^{-1})\bullet \u(k.\xi)\\
                                     &= (hk^{-1})\bullet k\bullet \u(\xi)\\
                                     &= h\bullet \u(\xi).
\end{align*}
Define the groupoid action to be $(h_1,\xi)\circ[h_2,x]=[h_1h_2,x]$ if $\s(h_1,\xi)=\u([h_2,x])$.  This is clearly the $\P$ with which we began, $\P'=\P$.\\
Conversely, starting with $K\act(\p,\u)$ and constructing $\P:=H\times_K\p$ with the extension of $\u$ in Equation \eqref{vacu} and action in Equation \eqref{VBac}, the recovery of the initial $(\p,\u)$ is automatic.
\end{proof}
\end{thm}
\section{Vacant LA-groupoids, LA-Modules}
Let $H$ be a Lie group, and $E\gp\g$ be an $LA$-groupoid (Definition \ref{LAg}) over $H$, which is vacant.
\begin{thm} (\cite{mack1}, \cite{mack2}) There is a 1-1 correspondence between $H$-equivariant Lie algebra triples $(\d,\g,\h)$ (Definition \ref{lietrip}), and vacant $LA$-groupoids over $H$.
\end{thm}
Similar to the $VB$-case Proposition \ref{vtriv}, we have $E\cong H\times\g$, with formulae
    \[s(h,\zeta)=\zeta, \quad t(h,\zeta)=h\bullet\zeta,\]
where the latter is a group action of $H$ on $\g$.  Indeed, this group action is given by
    \begin{equation} h\bullet \zeta = \pg(\Ad_h\zeta) \end{equation}
(Proposition B.3, \cite{lmdir}).  The groupoid multiplication is given by Equation \eqref{vacantgrm}, and the anchor (for left-trivialisation of the tangent bundle) is given by
\begin{equation}\label{anchorvac}
\a(h,\zeta) = (h, \Ad_{h^{-1}}\ph \Ad_h \zeta).
\end{equation}
\begin{dfn}\label{LAmdl} Let $E\gp \g$ be a vacant $LA$-groupoid over a Lie group $H$. An $E$-module is a Lie algebroid $L$ over an $H$-manifold $M$, with a groupoid action of $E$ on $L$ (with moment map $\u$) along the group action $H\act M$ such that the graph of the groupoid action is a Lie subalgebroid of $L\times E\times L$.  This is represented by the diagram
\begin{center}
\begin{tikzpicture}
  \node (G)  {$E$};
  \node (H) [right of=G,node distance=1.5cm] {$\g$};
  \node (I) [below of=G,node distance=1.3cm] {$H$};
  \node (J) [right of=I,node distance=1.5cm] {pt};

  \node (ac) [node distance=0.55cm, right of=H, below of=H] {$\act$};

  \node (K) [right of=H, node distance=1.1cm] {$L$};
  \node (L) [right of=K,node distance=1.5cm] {$\g$};
  \node (M) [below of=K,node distance=1.3cm] {$M$};
  \node (N) [right of=M,node distance=1.5cm] {pt};

  \draw[->] (G) to node {} (I);
  \draw[->] (H) to node [swap] {} (J);
  \path[->]
([yshift= 2pt]G.east) edge node[above] {} ([yshift= 2pt]H.west)
([yshift= -2pt]G.east) edge node[below] {} ([yshift= -2pt]H.west);
\path[->]
([yshift= 2pt]I.east) edge node[above] {} ([yshift= 2pt]J.west)
([yshift= -2pt]I.east) edge node[below] {} ([yshift= -2pt]J.west);
  \path[->]
([yshift= 2pt]G.east) edge node[above] {} ([yshift= 2pt]H.west)
([yshift= -2pt]G.east) edge node[below] {} ([yshift= -2pt]H.west);
\path[->]
([yshift= 2pt]I.east) edge node[above] {} ([yshift= 2pt]J.west)
([yshift= -2pt]I.east) edge node[below] {} ([yshift= -2pt]J.west);

  \draw[->] (K) to node {} (M);
  \draw[->] (L) to node {} (N);
  \draw[->] (K) to node {$\u$} (L);
  \draw[->] (M) to node {} (N);
\end{tikzpicture}
\end{center}
\end{dfn}
As in the previous section, we will concern ourselves with the case where $M=H/K$, an homogeneous space.  This gives us $L\cong H\times_K\l$, for $\l:=L_{eK}$ via Proposition \ref{EandP}.
\begin{prop}  Let $E$,$L$ be as in Definition \ref{LAmdl}.  The moment map $\u:\,L\to \g$ is a Lie algebroid morphism.
\begin{proof}
The inclusion $\{e\}\hookrightarrow H$ is covered by an $LA$-morphism $\g\to E$; the diagonal inclusion of $H/K$ into $H/K\times H/K$ is covered by an $LA$ morphism $L\to L\times L$.  Taken together, we have a $LA$-morphism $\g\times L\to L\times E\times L$.  The graph of the groupoid action $E\times L \da L$ is a Lie subalgebroid of $L\times E\times L$.  Over elements of the form $(hK,e,hK)$, it restricts to a section of $\g\times L$, which is precisely the set of elements $\{(\u(x),x)\,|\,x\in L\}$, which is thus a Lie subalgebroid of $\g\times L$.
\end{proof}
\end{prop}
\begin{prop}\label{anchorid}
The anchor map of $L$ is entirely determined by its value at the identity coset.
\begin{proof}
The anchor of the graph of the groupoid action is tangent to the graph of the group action on $H/K$ (Definition \ref{LAmdl}). Using the groupoid action formula from Equation \eqref{VBac} when $h_2=e$:
    \[ (h,\u(x))\circ [e,x] = [h,x] \]
for $x\in \l$, and applying the anchor maps (in left-trivialisation), we get
    \begin{align*}
        \a[h,x] &= \a(h,\u(x))\cdot \a[e,x] \\
                  &= (h, \Ad_{h^{-1}}\ph \Ad_h \u(x))\cdot [e,\a(x)]\\
                  &= [h, \overline{\Ad_{h^{-1}}\ph \Ad_h (\u(x))} + \a(x)]
\end{align*}
for the group action of $\mathrm{T}H$ on $\mathrm{T}(H/K)$ (here the overline denotes the image in $\h/\k$).
\end{proof}
\end{prop}

\subsection{The Trivial Case $K=\left\{e\right\}$}
When the homogeneous space $H/K$ is just $H$, we get the trivialisation $L=H\times \l$.
\begin{prop}\label{trivLiealgstr} For the $E$-module $L=H\times\l$ over $H$, $L_e=:\l$ has a Lie algebra structure.
\begin{proof}
(See \cite{lmdir}, Proposition 3.9 for a similar argument.) As a note, the $H$-action given by $\bullet$ is not by $LA$-automorphisms.  We identify $\l$ with the space of $H$-invariant sections $\Gamma(L)^H$, i.e.: the constant sections relative to the trivialisation of $L$ by $E$.  Let $\mathrm{pr}_2: E\times L \da L$ be $LA$-morphism defined by
    \[ (x,y)\sim y' \quad \leftrightarrow \quad y = y',\,x\in E \]
and $\mathrm{Act}:\,E\times L\da L$ be the groupoid action.  By definition, a section $\sigma\in \Gamma(L)$ is $H$-invariant if and only if there exists $\tilde{\sigma}\in \Gamma(E\times L)$ with
    \[ \tilde{\sigma} \sim_{\mathrm{Act}} \sigma,\quad \tilde{\sigma} \sim_{\mathrm{pr}_2} \sigma. \]
If $\sigma_1, \sigma_2\in \Gamma(L)^H$, with associated sections $\tilde{\sigma_1},\tilde{\sigma_2}$, then
    \[ [ \tilde{\sigma_1},\tilde{\sigma_2}] \sim_{\mathrm{Act}} [\sigma_1,\sigma_2],\quad [ \tilde{\sigma_1},\tilde{\sigma_2}] \sim_{\mathrm{pr}_2} [\sigma_1,\sigma_2] \]
since both are Lie algebroid morphisms.  Thus, $\Gamma(L)^H\cong\l$ is a Lie subalgebra.
\end{proof}
\end{prop}

$L$ comes with the linear maps $\u:\,\l\to \g$, and $\a:\, \l\to \h$, both determined by the values of the bundle maps at the identity (Propositions \ref{vacfor}, \ref{anchorid}).

\begin{prop}\label{uplusa}
The map $\varphi:\,\l\to\d,\,\,x\mapsto (\u(x),\a(x))$ is a Lie algebra morphism.
\begin{proof}
We will check $\varphi([x,y]) = [\varphi(x),\varphi(y)]$ for all $x,y\in \l$.
\begin{enumerate}
\item For the right hand side:
\begin{align*}
[\varphi(x),\varphi(y)] &= [(\u(x),\a(x)),(\u(y),\a(y))]\\
       &= \bigg([\u(x),\u(y)] + \L_{\a(x)}\u(y)-\L_{\a(y)}\u(x), [\a(x),\a(y)]+\L_{\u(x)}\a(y)-\L_{\u(y)}\a(x)\bigg).
\end{align*}
\item For the left hand side, consider $\tilde{x},\tilde{y}$, constant sections of $H\times\l$ corresponding to $x,y$.  Using Equation \eqref{anchorvac} for the anchor, and viewing the tangent bundle $\mathrm{T}H$ as an action Lie algebroid (Example \ref{acLA}) for $\h\to \mathfrak{X},\,\xi\mapsto \xi^L$, we compute
\begin{align*}
\a([x,y]) &= \a([\tilde{x},\tilde{y}])\vert_e \\
             &= [\a(\tilde{x}),\a(\tilde{y})]|_e\\
             &= [\a(x),\a(y)] + \a(x)^L(\a(y)+\Ad_{h^{-1}}\ph\Ad_h(\u(y)))\vert_e - \a(y)^L(\a(x)+\Ad_{h^{-1}}\ph\Ad_h(\u(x)))\vert_e\\
             &=[\a(x),\a(y)] + \ph[\a(x),\u(y)] - \ph[\a(y),\u(x)] + [\a(x),\ph(\u(y))] -[\a(y),\ph(\u(x))] \\
             &=[\a(x),\a(y)] -\L_{\u(y)}\a(x) +\L_{\u(x)}\a(y).
\end{align*}
We therefore get
\begin{align*}
\varphi([x,y]) &= \bigg(\u([x,y]),\a([x,y]\bigg) \\
                      &= \bigg([\u(x),\u(y)] + \L_{\a(x)}\u(y)-\L_{\a(y)}\u(x), [\a(x),\a(y)] +\L_{\u(x)}\a(y)-\L_{\u(y)}\a(x)\bigg).
\end{align*}
\end{enumerate}
Hence we conclude $\varphi([x,y]) = [\varphi(x),\varphi(y)]$.\\
\end{proof}
\end{prop}

\begin{thm}\label{vacantLAH} Let $E\gp\g$ be a vacant $LA$-groupoid over the Lie group $H$, with associated $H$-equivariant Lie algebra triple $(\d,\g,\h)$.  The $E$-modules over $M=H$, with $H$-action given by left multiplication, are classified by a choice of Lie algebra $\l$, and a Lie algebra morphism $\varphi:\,\l\to\d$.
\begin{proof}
Given an LA-groupoid module $L\to H$, from the above we get a Lie algebra $\l$ and the Lie algebra morphism $\varphi=\u\oplus \a$.\\
Now start with any Lie algebra $\l$ and a $\varphi:\,\l\to\d$. Define $L =H\times\l$ as a vector bundle over $H$, and let $\u:=\mathrm{pr}_\g\circ\varphi:\,\l\to\g$, and $\a:=\mathrm{pr}_\h\circ\varphi:\,\l\to\h$.  Define the structure map $\u:\,L\to\g$ to be $\u(h,\xi)=h\bullet\u(\xi)$, and the anchor $\a:L\to TH=H\times\h$ to be the bundle map induced from the map $\a$ above.  Define a bracket on the sections of $L$ which extends the Lie bracket on the constant sections of $L$ (induced by the bracket of $\l$), making $L$ into an action Lie algebroid.  The groupoid action formula is the same from the VB case:  $(h,g\bullet\u(\xi))\circ(g,\xi)=(hg,\xi)$.  Taking this $L$ now, we clearly get the same Lie algebra $\l$ and Lie algebra morphism $\varphi:\,\l\to\d$ we started with.\\
\end{proof}
\end{thm}
\subsection{For $K\neq\{e\}$}
For the non-trivial case $K\neq\left\{e\right\}$, we know from Proposition \ref{EandP} that the $E$-module $L\to H/K$ will be $H\times_K \l$ for some vector space $\l$.  This time, however, there is no Lie algebra structure on $\l$.  For this case, we consider the pullback Lie algebroid of $L$ (Definition \ref{pbLA}) under the quotient map $\pi:\,H\to H/K$,

\[
\pi^!L =\left\{(x,v)\in L\times \mathrm{T}H\, \vert\, \a(x) = \mathrm{T}\pi(v)\,\right\}.
\]

\begin{prop}\label{pbLm} Let $E$ be a vacant $LA$-groupoid over $H$, and $L$ an $E$-module over $H/K$ with moment map $\u$.  For this Proposition, let $\a_E,\a_L$ denote the anchors of $E,L$, respectively.  If $\pi:\,H\to H/K$ is the natural projection, then $\pi^!L$ is an $E$-module; by Proposition \ref{EandP}, it trivialises as $\pi^!L= H\times\tilde{\l}$.
\begin{proof}
The direct product groupoid $(\mathrm{T}H\times E)\to H\times H$ acts on $(\mathrm{T}H\times L)\to H\times H/K$ with the structure map $\u':\, (\mathrm{T}H\times L) \to \g$ defined as $\u'(v,x) = \u(x)$.  Diagrammatically:
\begin{center}
\begin{tikzpicture}
  \node (A) {$\mathrm{T}H$};
  \node (B) [right of=A,node distance=1.2cm] {$0$};
  \node (C) [below of=A,node distance=1.2cm] {$H$};
  \node (D) [right of=C,node distance=1.2cm] {pt};

  \node (t) [node distance=0.55cm, right of=B, below of=B] {$\times$};

  \node (G) [right of=B, node distance=1cm] {$E$};
  \node (H) [right of=G,node distance=1.2cm] {$\g$};
  \node (I) [below of=G,node distance=1.2cm] {$H$};
  \node (J) [right of=I,node distance=1.2cm] {pt};

  \node (ac) [node distance=0.55cm, right of=H, below of=H] {$\act$};

  \node (K) [right of=H, node distance=1cm] {$\mathrm{T}H$};
  \node (L) [right of=K,node distance=1.2cm] {$0$};
  \node (M) [below of=K,node distance=1.2cm] {$H$};
  \node (N) [right of=M,node distance=1.2cm] {pt};

  \node (s) [node distance=0.55cm, right of=L, below of=L] {$\times$};

  \node (O) [right of=L, node distance=1cm] {$L$};
  \node (P) [right of=O,node distance=1.2cm] {$\g$};
  \node (Q) [below of=O,node distance=1.2cm] {$H/K$};
  \node (R) [right of=Q,node distance=1.2cm] {pt};

  \draw[->] (A) to node {} (C);
  \draw[->] (B) to node [swap] {} (D);
  \draw[->] (A.10) to node {} (B.160);
  \draw[->] (A.350) to node [swap] {} (B.200);
  \draw[->] (C.20) to node {} (D.160);
  \draw[->] (C.340) to node [swap] {} (D.200);
  \draw[->] (G) to node {} (I);
  \draw[->] (H) to node [swap] {} (J);
  \draw[->] (G.16) to node {} (H.160);
  \draw[->] (G.344) to node [swap] {} (H.200);
  \draw[->] (I.20) to node {} (J.160);
  \draw[->] (I.340) to node [swap] {} (J.200);
  \draw[->] (K) to node {} (M);
  \draw[->] (L) to node [swap] {} (N);
  \draw[->] (K.10) to node {} (L.160);
  \draw[->] (K.350) to node [swap] {} (L.200);
  \draw[->] (M.20) to node {} (N.160);
  \draw[->] (M.340) to node [swap] {} (N.200);
  \draw[->] (O) to node {$\u$} (P);
  \draw[->] (O) to node [swap] {} (Q);
  \draw[->] (P) to node [swap] {} (R);
  \draw[->] (Q) to node {} (R);

\end{tikzpicture}
\end{center}
The $LA$-groupoid $E$ injects into the direct sum groupoid $\mathrm{T}H\times E\to H\times H$ via the map $\a_E\oplus \mathrm{id}$.  This is clearly injective, and retains the groupoid structure with the maps $\s(\a_E(\zeta),\zeta) = \s(\zeta)$, $\t(\a_E(\zeta),\zeta) = \t(\zeta)$, and for $\zeta_1,\zeta_2\in E$ such that $\s(\zeta_1)=\t(\zeta_2)$,
\[
(\a_E(\zeta_1),\zeta_1)\circ (\a_E(\zeta_2),\zeta_2) = ( \a_E(\zeta_1)\cdot \a_E(\zeta_2), \zeta_1\circ \zeta_2) = (\a_E(\zeta_1\circ \zeta_2), \zeta_1\circ \zeta_2).
\]
This copy of $E$ acts as a groupoid on $\pi^!L \hookrightarrow \mathrm{T}H\times L$: if $(v,x)\in \pi^!L$, and $(\a_E(\zeta),\zeta)\in E$ such that $s(\zeta) = \u(x)$, then
\[
(\a_E(\zeta),\zeta)\circ(v,x) = (\a_E(\zeta)\cdot v, \zeta\circ x).
\]
Since $\mathrm{T}\pi$ is equivariant for the $\mathrm{T}H$ actions, we have
\[
\a_L(\zeta\circ x) = \a_E(\zeta) \cdot \a_L(x) = \mathrm{T}\pi(\a_E(\zeta)\cdot v)
\]
meaning $(\a_E(\zeta),\zeta)\circ(v,x) \in \pi^!L$.  Thus $\pi^!L\cong H\times\tilde{\l}$ with $\tilde{\l}$ a Lie algebra.  There is also a clear injection of $\k\hookrightarrow \l$, since $\k$ is the fibrewise kernel of $\mathrm{T}\pi$.
\end{proof}
\end{prop}


\begin{thm}\label{LAmodthm} For a vacant LA-groupoid $E\to H$, and a $H$-homogeneous space $H/K$, the $E$-modules $L$ over $H/K$ are classified by the choice of a Lie algebra $\tilde{\l}$ with an injection $\k\hookrightarrow\tilde{\l}$ as a Lie subalgebra, an action of $K$ on $\tilde{\l}$ by Lie algebra automorphisms, and a K-equivariant Lie algebra homomorphism $\omega:\,\tilde{\l}\to \d$ such that $\omega|_\k=\text{id}$.
\begin{proof}
First, let $L\to H/K$ with $\u:\,L\to\g$ be an $E$-module, and $\pi:\,H\to H/K$.  Then $L\cong H\times_K\l$ via Proposition \ref{EandP}, with $\l$ a vector space.  By Example \ref{pullbmod}, the pullback Lie algebroid $\pi^!L$ carries the trivialisation $\pi^!L\cong H\times\tilde{\l}$, with $\tilde{\l}$ a Lie algebra.  Since $\k=\ker\mathrm{T}\pi$, we have an inclusion of Lie algebras $\k\hookrightarrow \tilde{\l}=\pi^!L_e$, $\kappa\mapsto (0,\kappa)$.  Define
    \[ \omega:= \a\oplus\u:\, \tilde{\l}\to \d. \]
By Proposition \ref{uplusa}, this is a $K$-equivariant Lie algebra morphism for which $\omega(\kappa) = \kappa$ for all $\kappa\in\k$, since the anchor of $\pi^!L$ is the projection to the $\mathrm{T}H$ component.\\
Each such $L$ produces data unique to it: suppose $L_1,L_2$ yield the same $\tilde{\l}$; then by Proposition \ref{CourKred},
    \[ L_1\cong (\pi^!L_1)//K \cong (H\times\tilde{\l})//K \cong (\pi^!L_2)//K \cong L_2. \]
All that remains to be shown is that each $(\tilde{\l},\omega)$ arises from some $L$.  Form $H\times\tilde{\l}\to H$ as an action Lie algebroid for the action
    \[ \a:\,\tilde{\l}\to\mathfrak{X}(H), \quad \a(x)_h = (h, \Ad_{h^{-1}}\ph\Ad_h\omega(x)) \]
with $\mathrm{T}H$ in left trivialisation.  The map $\u':=\pg\circ\omega$ is a moment map, making $H\times\tilde{\l}$ into an $E$-module.  The Lie group $K$ acts on $H\times\l$ with generators $\rho(\kappa)=\kappa^L$, allowing for the reduction
    \[ L:= (H\times\tilde{\l})//K \cong H\times_K(\tilde{\l}/\k). \]
The map $\u'$ descends to $L$, since it is $0$ on $H\times\k$, and so defines a moment map for an $E$ action on $L$.  Hence, $L$ is an $E$-module which produces $(\tilde{\l},\omega)$ as
    \[ \pi^!L = \pi^!((H\times\tilde{\l})//K) \cong H\times\tilde{\l} \]
by Proposition \ref{LAlgpbk}.
\end{proof}
\end{thm}

For the vacant $LA$-groupoid $E$, Theorem \ref{vbdualgrpd} gives us that $E^*$ is a groupoid over $(\mathrm{core}(E))^* = 0$.  The Lie algebroid structure of $E$ makes $E^*$ a Poisson manifold by Theorem \ref{EdualPoisson}, and since the groupoid multiplication $E\times E\da E$ is a comorphism of Lie algebroids, this means the group multiplication $E^*\times{E^*}\to{E^*}$ is a Poisson morphism, making $E^*$ a Poisson Lie group.  We recall that vacant $LA$-groupoids over $H$ are classified by $H$-equivariant Lie algebra triples $(\d,\g,\h)$ (\cite{mack1}); since $E^*$ is a Poisson Lie group, it is classified by a Manin triple, given (\cite{davidpavol}) by
    \[ (\d\ltimes\d^*,\g\ltimes\ann(\g),\h\ltimes\ann(\h) ). \]
If $L\to M$ is an $E$-module such that the groupoid action is a comorphism of Lie algebroids, then the dual group action
    \[ E^*\times L^*\to L^* \]
is a Poisson action.  This allows us to define vacant $LA$ homogeneous spaces.

\begin{dfn}  Let $E$ be an $LA$-groupoid over $H$.  An $LA$ homogeneous space for $E$ is a Lie algebroid $L\to H/K$, which is an $E$-module such that $L^*$ is a Poisson homogeneous space for $E^*$.
\end{dfn}
As a note: we omit the condition that the action of $E$ on $L$ is a comorphism, as Equation \eqref{VBac} ensures this for the case that $M=H/K$.

\begin{thm}\label{LAhomgsp} Let $E$ be a vacant $LA$-groupoid over $H$, classified by the $H$-equivariant Lie algebra triple $(\d,\g,\h)$.  The $LA$ homogeneous spaces $L\to H/K$ are classified by $K$-invariant subalgebras $\c\subset\d$ such that $\c\cap\h=\k$.
\begin{proof}
By Theorem \ref{LAmodthm}, any $LA$ module $L\to H/K$ for $E$ is classified by a Lie algebra $\tl$ with $\k\hookrightarrow\tl$, an action of $K$ on $\tl$ by Lie automorphisms, and a $K$-equivariant Lie algebra morphism $\omega:\,\tl\to\d$ such that $\omega|_\k=\mathrm{id}$.  The group action of $E^*\cong H\times\g^*$ on $L^*\cong H\times_K\l^*$ is
    \[ (h_1,\alpha)\cdot [h_2,\eta] = [h_1h_2, \eta + \u^*(h_2^{-1}\bullet\alpha)] \]
via Equation \eqref{Vdmdeq}.  Thus, the action of $E^*$ on $L^*$ is transitive if and only if $\u^*$ is surjective, if and only if $\u:\,\l\to\g$ is injective.  For any $\bar{x}\in \tl/\k\cong\l$, $\u$ is defined as
    \[ \u(\bar{x}) = \pg(\omega(x)) \]
for any lift $x\in \tl$.  So $\u(\bar{x})=0$ if and only if $\omega(x) \in \h$; we conclude that $\u$ is injective if and only if $\omega(\lt)\cap\h=\k$.  We also conclude that if $\u$ is injective, then $\omega$ is injective: since $\omega(x)=\pg(x)+\ph(x)$ for any $x\in \lt$, if $\u$ is injective and $\pg(x)=0$, then $\ph(x)$ must be in $\k$, but $\omega$ is the identity on $\k$, and hence if $\omega(x) = 0$, then $x=0$.  Define
    \[ c:= \omega(\lt) \subset \d,\]
which is a $K$-invariant subalgebra, since $\omega$ is a $K$-equivariant Lie al
It is clear that each data set $K\act(\tl,\omega)$ produces a $\c$ unique to it.\\
Given a $K$-invariant $\c\subset\d$ with $\c\cap\h=\k$, define $\lt:=\c$, and $\omega=\mathrm{id}$.  This gives an $LA$-module $L$ for $E$ by Theorem \ref{LAmodthm}; since $\c\cap\h=\k$, this means $\u$ is injective, and hence the action of $E^*$ on $L^*$ is transitive.  So $L\to H/K$ is the $LA$ homogeneous space for $E$ which gives $\c$.
\end{proof}
\end{thm}


\chapter{Almost-Dirac Lie Groups, Homogeneous Spaces}
In the next step towards the Dirac Lie group case, we consider $VB$-groupoids $\A\gp\g$ over $H\gp\mathrm{pt}$ with a multiplicative bundle metric, containing a vacant Lagrangian sub-groupoid $E\subset \A$. In essence, this is the Dirac case, but ``forgetting" the Courant bracket structure, which we will add in later.
\section{Metrized $VB$-Groupoids, Almost-Dirac Lie Groups}\label{subsec:vbg}
\begin{dfn} An \emph{almost-Dirac structure} on a manifold $M$ is a pair $(\A,E)$, with $\A\to M$ being a metrized vector bundle, and $E\subseteq \A$ a Lagrangian sub-bundle.  A \emph{morphism of almost-Dirac structures}
    \[ R:\,(\A,E)\da (\A',E') \]
with underlying map $\Phi:\,M\to M'$ is a Lagrangian relation $\A\to \A'$ along the graph of $\Phi$ with the property that with the property that each element of $E'_{\Phi(m)}$ is $R$-related to a unique element of $E_m$.
\end{dfn}
\begin{dfn}  A \emph{metrized VB-groupoid} is a $VB$-groupoid (Definition \ref{VBgpdf}) $\A$ over a Lie group $H$
\begin{center}
\begin{tikzpicture}

  \node (G)   {$\A$};
  \node (H) [right of=G, node distance=1.5cm] {$\A^{(0)}$};
  \node (I) [below of=G, node distance=1.3cm] {$H$};
  \node (J) [right of=I, node distance=1.5cm] {pt};

  \path[->]
([yshift= 2pt]G.east) edge node[above] {} ([yshift= 2pt]H.west)
([yshift= -2pt]G.east) edge node[below] {} ([yshift= -2pt]H.west);
\path[->]
([yshift= 2pt]I.east) edge node[above] {} ([yshift= 2pt]J.west)
([yshift= -2pt]I.east) edge node[below] {} ([yshift= -2pt]J.west);
  \draw[->] (G) to node {} (I);
  \draw[->] (H) to node [swap] {} (J);
  \end{tikzpicture}
\end{center}
with a multiplicative bundle metric $\<\cdot,\cdot\>$ (Definition \ref{multmet}).
\end{dfn}
\begin{dfn}  Let $H$ be a Lie group.  An \emph{almost-Dirac Lie group} is a triple $(H,\A,E)$, with $(\A,E)\to H$ an almost-Dirac structure such that $\A$ is a metrized $VB$-groupoid, $E\subseteq\A$ is a vacant $VB$-subgroupoid, and the groupoid multiplication Mult: $\A\times\A\da\A$ is a morphism of almost-Dirac structures.  A \emph{morphism of almost-Dirac Lie group structures}
    \[ R:\, (\A,E)\da (\A',E') \]
with underlying Lie group morphism $\Phi:\,H\to H'$ is a morphism of almost-Dirac structures such that the bundle morphism $\A\to\A'$ is a $VB$-groupoid morphism.
\end{dfn}
As a note, the condition that Mult: $\A\times\A\da\A$ is a Lagrangian relation is equivalent to the metric on $\A$ being multiplicative.\\
Let $(H,\A,E)$ be an almost-Dirac Lie group.  By Proposition \ref{vtriv} and Equation \ref{vacantgrm}, $E=H\times \g$ equipped with an $H$-action on $\g$ denoted by $\bullet$; we have the formulae
\begin{align*}
\s(h,\zeta)=\zeta,\quad \t(h,\zeta)=h\bullet\zeta\\
(h_1,\zeta_1)\circ(h_2,\zeta_2) = (h_1h_2,\zeta_2)
\end{align*}
for $\zeta_i\in\g$ such that $\s(h_1,\zeta_1)=\t(h_2,\zeta_2)$.  Since the source and target maps $\s,\t:\,E\to \g$ are fibrewise isomorphisms, we have the following equalities as vector bundles:
\begin{equation*}
\A = E\oplus\ker(\s)=E\oplus\ker(\t).
\end{equation*}
Let $j:\,\A\to E$ be the projection along $\ker(\t)$, and let $\q=\A_e$.  Via Proposition 3.9 in \cite{lmdir}, the trivialisation of $E$ extends to a trivialisation $\A = H\times\q$ defined by the map
\[\A\to\q,\,x\mapsto j(x)^{-1}\circ x.\]
As a direct consequence of this trivialisation, any element $(h,\xi)\in\A$ can be uniquely written as:
\begin{equation}\label{jtriveqn} (h,\xi) = j(h,\xi)\circ(e,\xi).\end{equation}
Let $\s,\t$ also denote the source and target at the identity of $\A$, i.e.: the projection maps $\s,\t:\,\q \to \g$.  We have the vector space decompositions
    \[\q\cong \g\oplus\ker(\s)\cong \g\oplus\ker(\t).\]
Given any $(h,\xi)\in\A$, and decomposing as in Equation \eqref{jtriveqn}, we have $\s(j(h,\xi))=\t(e,\xi)=\t(\xi)$, meaning that $j(h,\xi) = (h,\t(\xi))\in E$.  Hence
    \[\t(h,\xi)=\t(h,\t(\xi))=h\bullet \t(\xi),\quad\s(h,\xi)=\s(e,\xi)=\s(\xi).\]
Given an element $(e,\xi)\in\q$, there exist unique elements $\zeta,\eta$ in $E_h$ and $E_{h^{-1}}$ respectively such that $(e,\xi)\circ\eta$ and $\zeta\circ(e,\xi)$ are defined.  Specifically, $\zeta=(h,\t(\xi))$, and $\eta=(h^{-1},h\bullet \s(\xi))=(h,\s(\xi))^{-1}$.
\begin{prop}\label{conjact1}  Given any $h\in H$, there is an associated map $\q \to \q$ given by
\begin{equation}\label{conjact}
\xi \mapsto (h,\t(\xi))\circ(e,\xi)\circ(h^{-1},h\bullet \s(\xi)).
\end{equation}
This is a group action $H\act \q$ by groupoid automorphisms, which extends the $\bullet$ action on $\g\subset\q$, and preserves the metric.
\begin{proof}
For a given $h\in H$, denote the map in Equation \eqref{conjact} by $h\cdot \xi = \mu\circ\xi\circ\nu$, with $\xi\in\q$, $\mu\in E_h,\,\nu\in E_{h^{-1}}$.  For $h_1,h_2\in H$, we have
\begin{align*}
h_1\cdot (h_2\cdot \xi) &= h_1\cdot  ( \mu_2\circ\xi\circ\nu_2) \\
        &= \mu_1\circ(\mu_2\circ\xi\circ\nu_2)\circ\nu_1\\
        &= (\mu_1\circ\mu_2)\circ \xi \circ (\nu_2\circ\nu_1)\\
        &= (h_1h_2)\cdot \xi
\end{align*}
since $\mu_1\circ\mu_2\in E_{h_1h_2}$, and $\nu_2\circ\nu_1\in E_{(h_1h_2)^{-1}}$. It is clear that $e\cdot \xi=\xi$ for any $\xi\in \q$.  Hence, this is a group action.\\
Given any $\zeta\in \g$, we have
\begin{align*}
h\cdot \zeta &= (h,\zeta)\circ(e,\zeta)\circ(h^{-1},h\bullet \zeta)\\
    &= (h,\zeta)\circ (h^{-1},h\bullet \zeta)\\
    &= (e, h\bullet \zeta)\\
    &= h\bullet \zeta.
\end{align*}
So the action restricted to $\g$ is just $\bullet$. We have
\begin{align*}
\s(h\cdot \xi) &= s\left((h,\t(\xi))\circ(e,\xi)\circ(h^{-1},h\bullet \s(\xi))\right) \\
         &= \s(h^{-1},h\bullet \s(\xi))\\
         &= h\bullet \s(\xi), \\
\t(h\cdot \xi) &= t\left((h,\t(\xi))\circ(e,\xi)\circ(h^{-1},h\bullet \s(\xi))\right)\\
         &= \t(h,\t(\xi))\\
         &= h\bullet \t(\xi).
\end{align*}
Hence, given two composable elements $\xi_1,\xi_2\,\,\in\q$, and $h\in H$, it's clear that $(h\cdot \xi_1),(h\cdot \xi_2)$ are again composable: $\s(h\cdot \xi_1)=h\bullet \s(\xi_1) = h\bullet \t(\xi_2)=\t(h\cdot \xi_2)$.  We have
\begin{align*}
h\cdot (\xi_1\circ\xi_2) &= \mu\circ(\xi_1\circ\xi_2)\circ\nu \\
                    &= \mu\circ\xi_1\circ(\tau^{-1}\circ\tau)\circ\xi_2\circ\nu\\
                    &= (\mu\circ\xi_1\circ\tau^{-1})\circ(\tau\circ\xi_2\circ\nu)\\
                    &= (h\cdot \xi_1)\circ(h\cdot \xi_2)
\end{align*}
where $\tau\in E_h$, $\tau^{-1}\in E_{h^{-1}}$. From this, it is clear that $h\cdot (\xi^{-1}) = (h\cdot \xi)^{-1}$.  Hence, this is an action of $H$ on $\q$ by groupoid automorphisms which extends the $\bullet$ action on $\g$, given by the target map.\\
The inner product is invariant under the action, as if $\xi,\eta\,\in\q$ and $h\in H$, then:
\begin{align*}
\langle h\bullet\xi,h\bullet\eta\rangle &= \langle(h,\t(\xi))\circ(e,\xi)\circ(h^{-1},h\bullet \s(\xi)),(h,\t(\eta))\circ(e,\eta)\circ(h^{-1},h\bullet \s(\eta))\rangle\\
  &= \langle \t(\xi),\t(\eta)\rangle+\langle \xi,\eta \rangle + \langle h\bullet \s(\xi), h\bullet \s(\eta)\rangle \\
  &= \langle \xi,\eta\rangle
\end{align*}
where the first and third term vanish since $\g\subset\q$ is Lagrangian, and $\g$ is $H$-stable under the $\bullet$ action.
\end{proof}
\end{prop}

We will denote the full action as $\bullet$ for the remainder.  This action allows us to determine an exact formula for the groupoid multiplication in $\A$ based on that in $\q$.
\begin{lem} Let $(H,\A,E)$ be an almost-Dirac Lie group.  Let $(h_1,\xi),(h_2,\eta)\in \A$ such that $\s(h_1,\xi)=\t(h_2,\eta)$.  Then the groupoid multiplication is given by
\begin{equation}\label{metvbmult}
(h_1,\xi)\circ(h_2,\eta) = \big(h_1h_2, \eta+(1-\s)(h_2^{-1}\bullet\xi)\big) = \big(h_1h_2, \eta + h_2^{-1}\bullet (1-\s)\xi\big).
\end{equation}
\begin{proof}
Given an element $(h,\xi)\in \A$, we can use the action to decompose it in two different ways as the product of something in $E$ and something in $\A_e$, as
\begin{equation*}
(h,\xi) = (h,\t(\xi))\circ(e,\xi) = (e,h\bullet\xi)\circ(h,\s(\xi)).
\end{equation*}
Consider now $(h_1,\xi),(h_2,\eta)\in\A$ such that $\s(h_1,\xi)=\t(h_2,\eta)$, i.e.:$\s(\xi)=h_2\bullet\t(eta)$.  The multiplication is then given as:
\begin{align*}
(h_1,\xi)\circ(h_2,\eta) &= (h_1,\t(\xi))\circ(e,\xi)\circ(e,h_2\bullet\eta)\circ(h_2,\s(\eta))\\
                      &= (h_1,\t(\xi))\circ(e,(\xi\circ(h_2\bullet\eta)))\circ(h_2,\s(\eta))\\
                      &= (h_1,\t(\xi))\circ(e,h_2\bullet((h_2^{-1}\bullet\xi)\circ\eta))\circ(h_2,\s((h_2^{-1}\bullet\xi)\circ\eta))\\
                      &= (h_1,\t(\xi))\circ(h_2, (h_2^{-1}\bullet\xi)\circ\eta)\\
                      &= (h_1h_2, (h_2^{-1}\bullet\xi)\circ\eta).
\end{align*}
Since the identity fibre is a metrized linear groupoid, $\q\rightrightarrows\g$, we use Equation \eqref{linmuleq} to deduce:
\[ (h_1,\xi)\circ(h_2,\eta) = \big(h_1h_2, \eta+(1-\s)(h_2^{-1}\bullet\xi)\big) = \big(h_1h_2, \eta + h_2^{-1}\bullet (1-\s)\xi\big).\qedhere\]
\end{proof}
\end{lem}


Remembering via Theorem \ref{mlincl} that metrized linear groupoids $\q\gp\g$ are classified by a choice of vector space $\g$, and $\lambda\in S^2\g$, we now have sufficient data to classify the almost-Dirac Lie group case.
\begin{thm}\label{almDirLieG} Almost-Dirac Lie groups
\begin{center}
\begin{tikzpicture}

  \node (G)   {$(\A,E)$};
  \node (H) [right of=G, node distance=1.5cm] {$\g$};
  \node (I) [below of=G, node distance=1.3cm] {$H$};
  \node (J) [right of=I, node distance=1.5cm] {pt};

  \path[->]
([yshift= 2pt]G.east) edge node[above] {} ([yshift= 2pt]H.west)
([yshift= -2pt]G.east) edge node[below] {} ([yshift= -2pt]H.west);
\path[->]
([yshift= 2pt]I.east) edge node[above] {} ([yshift= 2pt]J.west)
([yshift= -2pt]I.east) edge node[below] {} ([yshift= -2pt]J.west);
  \draw[->] (G) to node {} (I);
  \draw[->] (H) to node [swap] {} (J);
  \end{tikzpicture}
\end{center}
are classified by the data: a vector space $\g$ together with an element $\lambda\in S^2\g$ (yielding a metrized linear groupoid $\q=\g\times\g^*_\lambda\gp\g$ by Theorem \ref{mlincl}) and a Lie group action $H\act\q$ by groupoid automorphisms compatible with the metric.
\begin{proof}
Start with such an almost-Dirac Lie group  $(H,\A,E)$.  By the trivialisation induced by Equation \eqref{jtriveqn}, $(\A,E)\cong(H\times\q, H\times\g)$, and the classification of Theorem \ref{mlincl}, $(\A,E)$ yields $\g$ with $\lambda\in S^2\g$.  By Proposition \ref{conjact1}, $H$ acts on $\q\gp\g$ by groupoid automorphisms, denoted $\bullet$, and this action is compatible with the metric.  Theorem \ref{mlincl} and Equation \eqref{metvbmult} guarantee that no two distinct almost-Dirac Lie groups yield the same $(\g,\lambda,\bullet)$.\\
Let $(\g,\lambda,\bullet)$ be any set of such data.  The first two data together uniquely determine a metrized linear groupoid $\q\gp\g$; let $\s,\t$ denote its source and target, and $\circ$ the groupoid multiplication.  We can form the $VB$-groupoid $(\A,E)=(H\times\q,H\times\g)$ with structure maps
    \begin{align*}
     \s(h,\xi) &= \xi,\quad \t(h,\xi) = h\bullet\t(\xi)\\
     (h_1,\xi_1)&\circ (h_2,\xi_2) = \big(h_1h_2,\xi_2 + {h_2}^{-1}\bullet((1-\s)\xi_1)\big)
     \end{align*}
for $\s(h_1,\xi_1)=\t(h_2,\xi_2)$, and with a bundle metric induced by the metric on $\q$, $\<(h,\xi_1),(h,\xi_2)\>=\<\xi_1,\xi_2\>$.  These structure maps ensure that $(H,\A,E)$ is an almost-Dirac Lie group:  $E$ is vacant, and must be Lagrangian by Proposition \ref{lgperp}. Since $\A_e\gp E_e$ is a metrized linear groupoid whose metric is $\bullet$ invariant, this means the metric induced on $\A$ is multiplicative.  Finally the groupoid multiplication guarantees that for any $(h_1h_2,\zeta)\in E_{h_1h_2}$, there is a unique element of $(E\times E)_{(h_1,h_2)}$ which is Mult-related to it, namely $((h_1,h_2\bullet\zeta),(h_2,\zeta))$.
Therefore, this $(H,\A,E)$ clearly yields the desired data $(\g,\lambda,\bullet)$.
\end{proof}
\end{thm}

Given an almost-Dirac Lie group $(H,\A,E)$, the quotient $\A/E\cong H\times\g^*$ is a group:  since $E$ is Lagrangian, $(\A/E)\cong E^*$, which by Theorem \ref{vbdualgrpd} is a groupoid over $(\mathrm{core}(E))^* = 0$.  The multiplication in $\A/E$ is given by:
    \begin{equation}\label{AEgmulteq} (h_1,\alpha_1)\cdot(h_2,\alpha_2) = (h_1h_2, \alpha_2+h_2^{-1}\bullet\alpha_1) \end{equation}
for $\alpha_i\in \g^*$; by Remark \ref{annr}, the graph of the groupoid multiplication in $E^*$ must be the annihilator of the graph of the groupoid multiplication in $E$.  By Equation \eqref{vacantgrm}, the graph of the multiplication in $E$ is
    \[ \{\big((h_1h_2,\zeta), (h_1,h_2\bullet\zeta),(h_1,\zeta)\big)\} \subset E\times E\times E\]
for $\zeta\in \g$, $h_i\in H$, which annihilates the graph given by Equation \eqref{AEgmulteq}.

\section{Almost-Dirac Homogeneous Spaces}\label{subsec:vbhs}
\begin{dfn} Let $\A\gp\g$ be a metrized $VB$-groupoid over a Lie group $H$.  A \emph{metrized $\A$-module} is a metrized vector bundle $\P$ over an $H$-manifold $M$, with a $VB$-groupoid action by $\A$ with moment map $\u:\,\P\to\g$
\begin{center}
\begin{tikzpicture}
  \node (P) {$\A\times\P$};
  \node (B) [right of=P] {$\P$};
  \node (A) [below of=P, node distance=1.5cm] {$H\times M$};
  \node (C) [right of=A] {$M$};

  \draw[->, dashed] (P) to node {} (B);
  \draw[->] (P) to node [swap] {} (A);
  \draw[->] (A) to node [swap] {} (C);
  \draw[->] (B) to node {} (C);
\end{tikzpicture}
\end{center}
along the group action $H\act M$, which is compatible with the metrics in the sense that for any composable $\xi_i\in \A, x_i\in \P$ in appropriate fibres:
    \[ \< \xi_1\circ x_1, \xi_2\circ x_2\> = \<\xi_1,\xi_2\> + \<x_1,x_2\>. \]
We may also represent this type of action with the diagram
\begin{center}
\begin{tikzpicture}
  \node (G)  {$\A$};
  \node (H) [right of=G,node distance=1.5cm] {$\g$};
  \node (I) [below of=G,node distance=1.3cm] {$H$};
  \node (J) [right of=I,node distance=1.5cm] {pt};

  \node (ac) [node distance=0.55cm, right of=H, below of=H] {$\act$};

  \node (K) [right of=H, node distance=1.1cm] {$\P$};
  \node (L) [right of=K,node distance=1.5cm] {$\g$};
  \node (M) [below of=K,node distance=1.3cm] {$M$};
  \node (N) [right of=M,node distance=1.5cm] {pt};

  \draw[->] (G) to node {} (I);
  \draw[->] (H) to node [swap] {} (J);
  \path[->]
([yshift= 2pt]G.east) edge node[above] {} ([yshift= 2pt]H.west)
([yshift= -2pt]G.east) edge node[below] {} ([yshift= -2pt]H.west);
\path[->]
([yshift= 2pt]I.east) edge node[above] {} ([yshift= 2pt]J.west)
([yshift= -2pt]I.east) edge node[below] {} ([yshift= -2pt]J.west);
  \path[->]
([yshift= 2pt]G.east) edge node[above] {} ([yshift= 2pt]H.west)
([yshift= -2pt]G.east) edge node[below] {} ([yshift= -2pt]H.west);
\path[->]
([yshift= 2pt]I.east) edge node[above] {} ([yshift= 2pt]J.west)
([yshift= -2pt]I.east) edge node[below] {} ([yshift= -2pt]J.west);

  \draw[->] (K) to node {} (M);
  \draw[->] (L) to node {} (N);
  \draw[->] (K) to node {$\u$} (L);
  \draw[->] (M) to node {} (N);
\end{tikzpicture}
\end{center}
\end{dfn}

\begin{dfn}  Let $(H,\A,E)$ be an almost-Dirac Lie group.  If $(M,\P,L)$ is an almost-Dirac structure on an $H$-manifold $M$, such that $\P$ is a metrized $\A$-module, and $L\subseteq\P$ is a Lagrangian sub-bundle, the action of $(\A,E)$ on $(\P,L)$ is called an \emph{almost-Dirac Lie group action} if every element of $L_{h\cdot m}$ is uniquely related by the groupoid action to an element of $E_h\times L_m$, i.e.: the action map Act: $\A\times\P\da\P$ is a morphism of almost-Dirac structures.  We call $(M,\P,L)$ an \emph{almost-Dirac module} for $(H,\A,E)$.
\end{dfn}
From the definition, we see that $E$ acts on $L$. For any $h\in H,\,m\in M$, let $\zeta\in E_h,\,x\in L_m$ such that $\zeta\circ x$ is defined. From the definition, any $y\in L_{h.m}$ can be uniquely written as $\zeta'\circ y'$ for $\zeta'\in E_h$, $y'\in L_m$.  Thus,
    \[ \< y, \zeta\circ x\> = \< \zeta'\circ y',\zeta\circ x\> = \<\zeta',\zeta\> + \< y',x\> = 0 \]
and since $L$ is Lagrangian, this means $\zeta\circ x \in L_{h.m}$.\\
We concern ourselves with the case that $M=H/K$.  Let $(H/K,\P,L)$ be an almost-Dirac module for $(H,\A,E)$.  Specifically, this makes $\P$ an $E$-module, and so by Proposition \ref{EandP} we have $\P= H\times_K \p$ for some vector space $\p$.  Since $E$ acts on $L$, Proposition \ref{EandP} gives $L=H\times_K \l$ for a Lagrangian $\l\subset \p$.   The map $\u:\,\P\to \g$ takes the form $\u([h,y])=h\bullet\u(y)$ for the linear map $\u:\,\p \to \g$.  Given any $[h,y]\in \P$, we can decompose $[h,y] = (h,\u(y))\circ[e,y]$. \\
\begin{prop}\label{AEeqqg}
The action of $\A$ on $\P$ is determined by the action of $\q$ on $\p$, with formula
    \begin{equation}\label{aldirac} (h_1,(\zeta,\alpha))\circ[h_2,x] = [h_1h_2,\,x+\u^*({h_2}^{-1}\bullet\alpha)] \end{equation}
such that $\s(h,(\zeta,\alpha))=\u([h_2,x])$.
\begin{proof}
For brevity, write $\xi := (\zeta,\alpha)\in \q=\g\times\g^*_\lambda$.  Let $(h_1,\zeta)\in \A$, and $[h_2,x]\in \P$ such that $\s(h_1,\zeta) = \u([h_2,x])$.  Using the formula for the $E$ action on $\P$ from Equation \eqref{VBac}, and the formula for the $\q$ action on $\p$ by Equation \eqref{uustarform}, we have:
\begin{align*}
 (h_1,\xi)\circ [h_2,x] &= (h_1,\xi)\circ(h_2,\u(x))\circ [e,x]\\
                          &= (h_1h_2, \u(x) + {h_2}^{-1}\bullet(1-\s)\xi)\circ [e,x]\\
                          &= (h_1h_2, (h_1h_2)\bullet(\u(x)+\t({h_2}^{-1}\bullet(1-\s)\xi))\circ(e, \u(x)+{h_2}^{-1}\bullet(1-\s)\xi)\circ[e,x]\\
                          &= (h_1h_2, (h_1h_2)\bullet(\u(x)+\t({h_2}^{-1}\bullet(1-\s)\xi)) \circ [e, x + \u^*({h_2}^{-1}\bullet(1-\s)\xi)]\\
                          &= [h_1h_2, x + \u^*({h_2}^{-1}\bullet(1-\s)\xi)].
\end{align*}
For $\xi = (\zeta,\alpha)$, $(1-\s)\xi = \alpha$, and hence Equation \eqref{aldirac} follows.
\end{proof}
\end{prop}

\begin{thm}\label{almDmod}
Given an almost-Dirac Lie group $(H,\A,E)$, its almost-Dirac modules $(H/K,\P,L)$ are classified by the following data:
\begin{enumerate}
    \setlength{\itemsep}{0pt}
    \setlength{\parskip}{0pt}
    \setlength{\parsep}{0pt}
    \item A choice of metrized vector space $\p$ with a $K$-action preserving the metric;
    \item A Lagrangian subspace $\l\subseteq \p$ stable under this $K$-action;
    \item A $K$-equivariant map $\u:\,\p\to\g$ such that $\u\circ\u^*=\lambda^\#$ (with $\lambda\in S^2\g$ classifying $\q=\A^{(0)}\vert_e$).
\end{enumerate}
\begin{proof}
By Theorem \ref{almDirLieG}, $(\A,E) = ( H\times\q,H\times\g)$ with $\q=\g\times\g^*_\lambda$ for $\lambda\in S^2\g$.  Begin with an $(H,\A,E)$- module $(H/K,\P,L)$.  As $(\P,L)$ are $E$-modules, Theorem \ref{vacantVBmd} gives the structure
    \[ (\P,L) = (H\times_K\p, H\times_K\l), \]
with $\p$ a vector space with a $K$ action, $\l\subseteq \p$ subspace, and $K$-equivariant moment map $\u:\,\p\to \g$. By Proposition \ref{AEeqqg}, the action of $\q$ on $\p$ determines an action of $\A$ on $\P$, so with Theorem \ref{linmods}, we conclude that $\p$ is a metrized vector space with $\l\subset\p$ Lagrangian, $\u\circ\u^*=\lambda^\#$, and the induced metric on $\P$ is $K$-invariant. It is clear that this set of data is unique to $(H/K,\P,L)$.\\
Suppose $K\act(\p,\l,\u)$ are given.  We construct
    \[(H/K,\P,L)=(H/K,H\times_K\p,H\times_K\l)\]
with $\u\:,\P\to\g$ defined by $\u([h,x]) = h\bullet\u(x)$.  Theorem \ref{linmods} gives us that $(\p,\l,\u)$ is a metrized module for $(\q,\g)$, and then we define the full groupoid action via Equation \eqref{aldirac}.  This is compatible with the metrics, as
    \begin{align*} \<x+\u^*({h_2}^{-1}\bullet\alpha),x+\u^*({h_2}^{-1}\bullet\alpha)\> &= \<x,x\> + 2\<h_2\bullet\u(x),\alpha\> + \lambda(\alpha,\alpha)\>\\
                &= \< (h_2\bullet\u(x),\alpha),(h_2\bullet\u(x),\alpha)\> + \<x,x\>.
    \end{align*}
Lastly, any element $[h_1h_2,x]\in L_{h_1h_2K}$ is uniquely Act-related to an element of $(E\times L)_{(h_1,h_2K)}$, namely $\big((h_1,h_2\bullet\u(x)),[h_2,x]\big)$.  Hence, $(H/K,\P,L)$ is an almost Dirac module associated to the data $K\act(\p,\l,\u)$.

\end{proof}
\end{thm}

\begin{prop}
Let $(H,\A,E)$ be an almost-Dirac Lie group, and let $(H/K,\P,L)$ be an $(H,\A,E)$- module. The group $\A/E$ acts on $\P/L$ by the formula
    \begin{equation}\label{AEPLeq} (h_1,\alpha)\cdot[h_2,\overline{x}] = [h_1h_2, \overline{x} + \overline{\u^*({h_2}^{-1}\bullet\alpha)}], \end{equation}
where the bar denotes the image of an element of $\p$ in the quotient $\p/\l$.
\begin{proof}
The action follows from Lemma \ref{VBquotactt}.  For any $(h_i,\alpha_i)\in \A/E,\,[h_3,\overline{x}]\in \P/L$, we have
    \begin{align*}  (h_1,\alpha_1)\cdot((h_2,\alpha_2)\cdot[h_3,\overline{x}] &= (h_1,\alpha_1)\cdot [h_2h_3,\overline{x} + \overline{\u^*(\h_3^{-1}\bullet\alpha_2)}] \\
                        &= [h_1h_2h_3,\overline{x} + \overline{\u^*(\h_3^{-1}\bullet\alpha_2)} + \overline{\u^*((h_2h_3)^{-1}\bullet\alpha_1)}]\\
                        &= (h_1h_2,\alpha_2 + \h_2^{-1}\bullet\alpha_1)\cdot [h_3,\overline{x}],
    \end{align*}
and
    \[ (e,0)\cdot[h,\overline{x}] = [h,\overline{x} + \overline{\u^*(0)}] = [h,x]. \]
\end{proof}
\end{prop}
This allows us to define homogeneous spaces for $(\A,E)$ analogously to the metrized linear groupoid case.
\begin{dfn}
If $(H,\A,E)$ is an almost-Dirac Lie group, then an \emph{homogeneous space} for $(H,\A,E)$ is an almost-Dirac module $(H/K,\P,L)$ such that $\A/E\act \P/L$ transitively.
\end{dfn}

\begin{prop}\label{transiff} A $VB$-group (i.e.: $V^{(0)}=\mathrm{pt}$) acts transitively on a vector bundle over an $H$-homogeneous space
\begin{center}
\begin{tikzpicture}
  \node (G)  {$V$};
  \node (H) [right of=G,node distance=1.5cm] {$0$};
  \node (I) [below of=G,node distance=1.3cm] {$H$};
  \node (J) [right of=I,node distance=1.5cm] {pt};

  \node (ac) [node distance=0.55cm, right of=H, below of=H] {$\act$};

  \node (K) [right of=H, node distance=1.1cm] {$W$};
  \node (L) [right of=K,node distance=1.5cm] {$0$};
  \node (M) [below of=K,node distance=1.3cm] {$H/K$};
  \node (N) [right of=M,node distance=1.5cm] {pt};

  \draw[->] (G) to node {} (I);
  \draw[->] (H) to node [swap] {} (J);
  \path[->]
([yshift= 2pt]G.east) edge node[above] {} ([yshift= 2pt]H.west)
([yshift= -2pt]G.east) edge node[below] {} ([yshift= -2pt]H.west);
\path[->]
([yshift= 2pt]I.east) edge node[above] {} ([yshift= 2pt]J.west)
([yshift= -2pt]I.east) edge node[below] {} ([yshift= -2pt]J.west);
  \path[->]
([yshift= 2pt]G.east) edge node[above] {} ([yshift= 2pt]H.west)
([yshift= -2pt]G.east) edge node[below] {} ([yshift= -2pt]H.west);
\path[->]
([yshift= 2pt]I.east) edge node[above] {} ([yshift= 2pt]J.west)
([yshift= -2pt]I.east) edge node[below] {} ([yshift= -2pt]J.west);

  \draw[->] (K) to node {} (M);
  \draw[->] (L) to node {} (N);
  \draw[->] (K) to node {$\u$} (L);
  \draw[->] (M) to node {} (N);
\end{tikzpicture}
\end{center}
if and only if $V_e$ acts transitively on $W_{eK}$.
\begin{proof}
The action $V\act W$ is transitive if and only if the bundle map $V\to W$, $v\mapsto v.0_{eK}$ is surjective, if and only if the map $V_e\to W_{eK}$, $v\mapsto v.0_{eK}$ is surjective, if and only if $V_e\act W_{eK}$ is transitive.
\end{proof}
\end{prop}
\begin{rmrk} In particular, if $(H,\A,E)=(H,H\times\q,H\times\g)$ is an almost-Dirac Lie group, and $(H/K,\P,L)=(H/K,H\times_K\p,H\times_K\l)$ is an almost-Dirac module, then $\A/E\act \P/L$ transitively if and only if $\q/\g\act \p/\l$ transitively.
\end{rmrk}

Now we have all the pieces required to classify our full homogeneous spaces.
\begin{lem}\label{alDirh}
Let $(H,\A,E)$ be an almost-Dirac Lie group, classified by a vector space $\g$ together with an element $\lambda\in S^2\g$ (giving a metrized linear groupoid $\q=\g\times\g^*_\lambda\gp\g$), and an action of $H$ on $\q$ by groupoid automorphisms. The homogeneous spaces $(H/K,\P,L)$ for $(H,\A,E)$ are classified by $\lambda$-coisotropic subspaces $\l\subset\g$ which are stable under the induced $K$ action.  As such, they have the normal form
    \[ (\P,L) = \big(H\times_K(\s^{-1}(\l)/\s^{-1}(\l)^\perp), H\times_K\l\big). \]
\begin{proof}
Begin with such an homogeneous space $(H/K,\P,L)$.  Since $(H/K,\P,L)$ is an almost-Dirac module for $(H,\A,E)$, Theorem \ref{almDmod} gives a unique pair $(\p,\l)$ with $\p$ a metrized vector space, and $\l\subset \p$ a Lagrangian subspace.  In addition, the theorem yields a $K$ action on $\l$, and a $K$-equivariant map $\u:\,\l\to\g$.  $\A/E$ acting transitively on $\P/L$ means $\q/\g \act \p/\l$ transitively (Proposition \ref{transiff}), and so by Remark \ref{qlinhomH}, $\u$ injects $\l$ into $\g$ as a $K$-stable $\lambda$-coisotropic subspace.  By Theorem \ref{qlinhom}, we have
    \[ \p:=\s^{-1}(\u(\l))/\s^{-1}(\u(\l))^\perp. \]
To show that every $K$-stable $\lambda$-coisotropic $\l\subset\g$ arises in this way, we construct the bundle pair
    \[ (\P,L) = \big(H\times_K(\s^{-1}(\l)/\s^{-1}(\l)^\perp),H\times_K\l\big) \]
for $\s:\,\q\to\g$, with bundle metric induced by that on $\q$ and with moment map $\u:\,\p\to\g$ given by the target map descending to the quotient (making it $K$ equivariant as the bullet action is by groupoid automorphisms).  The full moment map is then defined by $\u([h,x]) = h\bullet\u(x)$, and hence $(H/K,\P,L)$ is an almost-Dirac module by Theorem \ref{almDmod}; it is an homogeneous space as $\l$ is $\lambda$-coisotropic, by Remark \ref{qlinhomH}.
\end{proof}
\end{lem}

\chapter{Dirac Homogeneous Spaces}
In this section, we give the classification of homogeneous spaces for the Dirac case; given a Dirac Lie group $(H,\A,E)$, and an $H$-homogeneous space $H/K$, the classification is given in terms of $K$-invariant coisotropic subalgebras $\c\subset\d$, with $\c\cap\h=\k:=\mathrm{Lie}(K)$.

\section{Preparations}
\subsection{Identifications of $\lambda$, Preimages of Coisotropic Subspaces}
Let $(H,\A,E)$ be a Dirac Lie group with associated $H$-equivariant Dirac Manin triple $(\d,\g,\h)_\beta$, with $\beta\in S^2\d$.  Let $(\q,\g,\r)_\gamma$ be the Dirac Manin triple constructed from $(\d,\g,\h)_\beta$ (Section \ref{prelimdirac}) such that $\g\subset\q$ is Lagrangian, and $\gamma\in S^2\q$ is non-degenerate, with Lie algebra morphism $f:\q\to\d$.  From the construction, we have $f(\gamma)=\beta$, and $\r^\perp\cong\g^*$.\\
The identity fibre of $\A$ is a metrized linear groupoid $\q\gp\g$, with $\lambda\in S^2\g$ defined as the image of $\gamma$ under $\pg:\,S^2\q\to S^2\g$, (where $\pg:\,\q\to\g$ is the projection along $\r$), i.e.: $\lambda^\#=\pg\gamma^\#|_{\g^*}$.
\begin{lem}\label{lamequiv} Let $\pg:\,\q\to\g$ be the projection along $\r$, and (for this lemma only) let $\pg'$ denote the projection $\d\to\g$ along $\h$. If $\lambda = \pg(\gamma)$, then also $\lambda=\pg'(\beta)$.
\begin{proof}
Since $f(\gamma)=\beta$, and $\pg = \pg'\circ f$, we also have $\lambda = \pg'(\beta)$.
\end{proof}
\end{lem}



\begin{lem}\label{fstarann} Let $V_i$ be vector spaces with bilinear forms $\beta_i\in S^2 V_i$, and let $f:V_1\to V_2$ be a linear map such that $f(\beta_1)=\beta_2$.
    \begin{enumerate}
        \item A subspace $W_2\subset V_2$ is $\beta_2$-coisotropic if and only if $f^{-1}(W_2)$ is $\beta_1$-coisotropic.
        \item If $W_1\subset V_1$ is $\beta_1$-coisotropic, this implies $f(W_1)$ is $\beta_2$-coisotropic.
    \end{enumerate}
\begin{proof}
\begin{enumerate}
\item Since $f(\beta_1)=\beta_2$, we have $\beta_2(\alpha,\alpha') = \beta_1(f^*(\alpha),f^*(\alpha'))$ for any $\alpha,\alpha'\in V_2^*$.  We also have $f^*(\ann(W_2)) = \ann(f^{-1}(W_2))$, since $\xi\in f^{-1}(W_2)$ if and only if $\forall \alpha\in\ann(W_2)$
        \[ 0 = \<\alpha,f(\xi)\> = \< f^*(\alpha),\xi\> = 0 \]
if and only if $\xi\in \ann(f^*(\ann(W_2)))$.\\  Hence
\begin{align*} W_2 \mathrm{\,\,coisotropic\,\,} &\Leftrightarrow\,\, \forall\, \alpha,\alpha'\in \ann(W_2), \,\beta_2(\alpha,\alpha')=0\\
                            &\Leftrightarrow \,\,\forall\, \alpha,\alpha'\in \ann(W_2), \,\beta_1(f^*(\alpha),f^*(\alpha'))=0\\
                            &\Leftrightarrow \,\,\forall\,\eta,\eta'\in \ann(f^{-1}(W_2)),\,\beta_1(\eta,\eta')=0\\
                            &\Leftrightarrow \,\,f^{-1}(W_2) \mathrm{\,\,coisotropic\,\,}
\end{align*}
    \item We have $f(W_1)$ is $\beta_2$-coisotropic if and only if $f^{-1}(f(W_1))$ is $\beta_1$-coisotropic, by Part 1.  But $f^{-1}(f(W_1))\supseteq W_1$ is $\beta_1$-coisotropic. \qedhere
\end{enumerate}
\end{proof}
\end{lem}

\begin{cor} Let $\c\subseteq\d$.  Then $\c$ is $\beta$-coisotropic if and only if $f^{-1}(\c)$ is $\gamma$-coisotropic. Furthermore, if $\c\subset\d$ is $\beta$-coisotropic, then $\pg(\c)\subset\g$ is $\lambda$-coisotropic.
\end{cor}
\subsection{A Useful $LA$-Isomorphism}
\begin{prop}\label{THgHd} Let $H$ be a Lie group, let $(\d,\g,\h)$ be an $H$-equivariant Lie algebra triple (Definition \ref{lietrip}), and let $H\times\d$ be the action Lie algebroid for the action
    \[  \rho:\, \d \to \mathfrak{X}(H)\]
defined for any $\xi\in\d$ as
    \begin{equation}\label{diracanchor} \rho(\xi)_h = (h, \Ad_{h^{-1}}\ph\Ad_h \xi)\end{equation}
(in left trivialisation, T$H\cong H\times\h$).  Then $H\times\d$ and $\mathrm{T}H\times\g$ are isomorphic as Lie algebroids.\\
Here, $\mathrm{T}H\times\g$ is given the Lie algebroid structure determined by the bracket
    \[ [(\sigma_1,\zeta_1),(\sigma_2,\zeta_2)] = ([\sigma_1,\sigma_2],\,[\zeta_1,\zeta_2] + \L_{\sigma_1}\zeta_2-\L_{\sigma_2}\zeta_1)\]
for $(\sigma_i,\zeta_i)\in \Gamma(\mathrm{T}H\times\g)$, with anchor map as the projection to the first factor.
\begin{proof}
Let $\u:\,H\times\d\to\g$ be the map given by $\u(h,\xi) = \pg\Ad_h\xi$, and define
    \[\varphi:=(\a,\u):\, H\times\d \to \mathrm{T}H\times\g; \]
in left-trivialisation, $\varphi$ is given by
    \begin{align*} \varphi:\, H\times\d &\to \mathrm{T}H\times\g\\
                        (h,\xi) &\mapsto  ((h,\, \Ad_{h^{-1}}\ph\Ad_h \xi),\pg\Ad_h\xi)
    \end{align*}
which is invertible, meaning $\varphi$ is an isomorphism of vector bundles.  \\
For $\zeta,\xi\in \d$, we have
   \begin{align*}
   [\varphi(\zeta),\varphi(\xi)] &= [(\a(\zeta),\pg(\Ad_h\zeta)),(\a(\xi),(\pg(\Ad_h\xi))]\\
                                                      &= \big(\a([\zeta,\xi]),[\pg(\Ad_h\zeta),\pg(\Ad_h\xi)]_\q + \L_{\a(\zeta)}(\pg(\Ad_h\xi)) - \L_{\a(\xi)}(\pg(\Ad_h\zeta))\big).
   \end{align*}
   To compute the last two terms, we note
   \begin{equation*}
   \L_{\a(\zeta)}\pg(\Ad_h\xi)=\iota_{\a(\zeta)}\mathrm{d}(\pg(\Ad_h\xi)) = \iota_{\a(\zeta)}\pg[\theta^R,\Ad_h\xi] = \pg[\ph(\Ad_h\zeta),\Ad_h\xi].
   \end{equation*}
   But $\Ad_h\xi = \pg\Ad_h\xi + \pg\Ad_h\xi$, and $[\ph\Ad_h\zeta,\ph\Ad_h\xi]\in\h$, thus
   \begin{equation*}
   \L_{\a(\zeta)}\pg(\Ad_h\xi) = \pg[\ph\Ad_h\zeta,\pg\Ad_h\xi].
   \end{equation*}
   Hence, we can compute
   \begin{align*}
   [\pg&(\Ad_h\zeta),\pg(\Ad_h\xi)]+ \L_{\a(\zeta)}(\pg(\Ad_h\xi)) - \L_{\a(\xi)}(\pg(\Ad_h\zeta))\\
   &=\pg\Big([\pg(\Ad_h\zeta),\pg(\Ad_h\xi)] + +[\ph(\Ad_h\zeta),\pg(\Ad_h\xi)] +[\pg(\Ad_h\zeta),\ph(\Ad_h\xi)]\Big)\\
   &= \pg\big([\Ad_h\zeta,\Ad_h\xi] - [\ph(\Ad_h\zeta),\ph(\Ad_h\xi)]\big)\\
   &= \pg[\Ad_h\zeta,\Ad_h\xi]\\
   &= \pg\Ad_h[\zeta,\xi].
   \end{align*}
   We conclude
   \begin{equation*}
   [\varphi(\zeta),\varphi(\xi)] = \big(\a([\zeta,\xi]),\pg\Ad_h[\zeta,\xi]\big) = \varphi([\zeta,\xi]).\qedhere
   \end{equation*}
\end{proof}
\end{prop}

\subsection{Structure Maps}
In this section we look at compositions of specific structure maps for $CA$-groupoids and modules; this leads towards Proposition \ref{fpbeta}, producing a map $f_P$ which becomes an important module moment map for the pullbacks of Dirac homogeneous space $(H/K,\P,L)$.\\
Suppose $\A\gp\g$ is a $CA$-groupoid (Definition \ref{LAg}) over a group $H\gp\mathrm{pt}$, and let $\P\to M$ be a $CA$-module for $\A$ with moment map $\u:\,\P\to \g$.  From this we obtain maps $\u_m^*:\,\g^*\to\P_m$, and $\s_h^*,\t_h^*:\,\g^*\to\A_h$.  Since $\ker(\t_h)$ and $\ker(\s_h)$ are orthogonal by Proposition \ref{lgperp}, we have $\ker(\s_h) = \mathrm{ran}(\t_h^*)$.  Thus $\s_h\circ \t_h^* = 0$, implying that for all $\alpha\in \A^*_h$, $\t_h^*(\alpha)\circ 0_m \in \P_{h\cdot m}$ is defined.
\begin{lem} If $\P\to M$ is a $CA$-module for the $CA$-groupoid $\A\gp\g$ over $H\gp\{\mathrm{pt}\}$, with moment map $\u:\,\P\to \g$, then for any $\alpha\in \A^*\cong \A$
\[ \t_h^*(\alpha) \circ 0_m = \u_{h\cdot m}^*(\alpha). \]
\begin{proof}
Let $\xi\in \A_h$, $x\in \P_m$ such that $\s(\xi) = \u(x)$.  Then $\xi\circ x\in \P_{h\cdot m}$, and
\[ \<\xi\circ x, \u_{h\cdot m}^*(\alpha)\> = \< \u_{h\cdot m}(\xi\circ x),\alpha\> = \<\t_h(\xi),\alpha\> = \<\xi,\t^*_h(\alpha)\>. \]
This shows that $(\u^*_{h\cdot m}(\alpha),\t_h^*(\alpha),0_m)\in \P\times\overline{\A}\times\overline{\P}$ is orthogonal to the graph of the action map.  But the latter is Lagrangian by assumption, which proves the claim.
\end{proof}
\end{lem}

\subsection{Some Important Maps}
Let $H\times \d$ over $H$ be the action Lie algebroid for the action $\rho_1:\d \to \mathfrak{X}(H)$ given by Equation \eqref{diracanchor}.  Let $H\times\r$ over $H$ be the action Lie algebroid for the action $\rho_2:\r \to \mathfrak{X}(H)$, given by $\rho_2(\tau) = -f(\tau)^R$.  Define maps
\begin{align*}
\varphi:\, H\times \d&\to \mathrm{T}H\times\q, \qquad \qquad(h,\zeta)\mapsto(\a(h,\zeta),\pg(\Ad_h\zeta))\\
     \psi:\,H\times\r& \to \mathrm{T}H\times\q,\qquad\qquad (h,\tau)\mapsto (-f(\tau)^R,\tau).
\end{align*}
These maps are integral to the construction of the homogeneous spaces, so we will begin by proving a number of properties they have.
\begin{prop} The induced maps $\varphi:\,\d\to \Gamma(\mathrm{T}H\times\q)$, $\psi:\,\r\to\Gamma(\mathrm{T}H\times\q)$ are Lie algebra maps.
\begin{proof}  The result for $\varphi$ follows from Proposition \ref{THgHd}.\\
Secondly, for $\psi$: let $\nu,\eta\in\r$, then
\begin{align*}
  [\psi(\nu),\psi(\eta)] &= [(-f(\nu)^R,\nu),(-f(\eta)^R,\eta)]\\
                 &= ([f(\nu)^R,f(\eta)^R],\,[\nu,\eta]_\q - \L_{f(\nu)^R}\eta + \L_{f(\eta)^R}\nu)\\
                 &= (-f([\nu,\eta])^R,\, [\nu,\eta] + 0)  \\
                 &= \psi([\nu,\eta]).
  \end{align*}
  The 0 arises since $\nu,\eta$ are constant sections of $H\times\q$.
\end{proof}
\end{prop}
\begin{prop}\label{phipsimix} Let $\zeta\in \d$, and $\nu\in\r$.  Then
\begin{equation*}
[\psi(\nu), \varphi(\zeta)] = \psi\big(\pr[\nu,\pg(\Ad_h\zeta)]\big).
\end{equation*}
\begin{proof}
Computing, we have
\begin{align*}
[\psi(\nu),\varphi(\zeta)]&= [(-f(\nu)^R,\nu),(a(\zeta),\pg(\Ad_h\zeta))] \\
             &=\big([-f(\nu)^R,a(\zeta)], [\nu,\pg(\Ad_h\zeta)] + \L_{-f(\nu)^R}\pg(\Ad_h\zeta)\big)
\end{align*}
We see first that
\begin{align*}
\iota_{[-f(\nu)^R,\a(\zeta)]}\theta^R &= -\L_{f(\nu)^R}\iota_{\a(\zeta)}\theta^R +\iota_{\a(\zeta)}\L_{f(\nu)^R}\theta^R\\
              &=-\iota_{f(\nu)^R}\mathrm{d}\ph\Ad_h\zeta + \iota_{\a(\zeta)}[f(\nu),\theta^\R]\\
              &=-\iota_{f(\nu)^R}\ph[\theta^R,\Ad_h\zeta] + [f(\nu),\ph\Ad_h\zeta]\\
              &=-\ph\big([f(\nu),\Ad_h\zeta]-[f(\nu),\ph\Ad_h\zeta]\big)\\
              &=-\ph[f(\nu),\pg\Ad_h\zeta].
\end{align*}
We note that $-\ph[f(\nu),\pg\Ad_h\zeta] = f(\pr[\nu,\pg\Ad_h\zeta])$.  Secondly we see
\begin{align*}
 \L_{-f(\nu)^R}\pg(\Ad_h\zeta) &= -\iota_{f(\nu)^R}\pg[\theta^R,\Ad_h\zeta]\\
                    &=-\pg[f(\nu),\Ad_h\zeta]\\
                    &=-\pg[f(\nu),\pg(\Ad_h\zeta)].
 \end{align*}
 Hence, we have
 \begin{align*}
 [\psi(\nu),\varphi(\zeta)] &= ((-\ph[f(\nu),\pg\Ad_h\zeta])^R, [\nu,\pg\Ad_h\zeta] - \pg[f(\nu),\pg\Ad_h\zeta])\\
        &= (-f(\pr[\nu,\pg\Ad_h\zeta))^R, \pr[\nu,\pg\Ad_h\zeta])\\
        &=\psi(\pr([\nu,\pg\Ad_h\zeta])).
 \end{align*}
\end{proof}
\end{prop}

\begin{cor} If $\zeta\in\h$, then $\varphi(\zeta)=(\zeta^L,0)$, and hence $[\psi(\nu),\varphi(\zeta)] = 0$.
\end{cor}

\begin{cor} Let $\c\subset\d$ be a $\beta$-coisotropic subalgebra.  Then the sub-bundle
    \[ F:=\varphi(H\times \c)\oplus\psi(H\times\r)\subset \mathrm{T}H\times\q \]
is involutive.
\begin{proof}
The images of $\varphi$ and $\psi$ do not intersect, as the $\q$ component of the image of $\varphi$ is solely in $\g$, whereas the image of $\psi$ is solely in $\r$.  The previous propositions show that $F$ is involutive: checking the images of constant sections, we have
\begin{equation*}
[\varphi(\zeta_1)+\psi(\nu_1),\varphi(\zeta_2)+\psi(\nu_2)] = \varphi([\zeta_1,\zeta_2])+\psi\Big([\nu_1,\nu_2]+\pr\big([\nu_1,\pg\Ad_h\zeta_2]-[\nu_2,\pg\Ad_h\zeta_1]\big)\Big).
\end{equation*}
\end{proof}
\end{cor}
\section{Dirac Homogeneous Space Construction}
In this section, we will look at both deriving the classifying data for Dirac homogeneous spaces, and then reconstructing homogeneous spaces from this data.
\begin{dfn}  Let $(H,\A,E)$ be a Dirac Lie group.  A \emph{Dirac module} $(M,\P,L)$ for $(H,\A,E)$ is a Manin Pair $(\P,L)$ over an $H$-manifold $M$, with $\P$ a $CA$-module for $\A$ such that the groupoid action
    \[ \mathrm{Act}:\,(\A,E)\times(\P,L)\da (\P,L) \]
is a Dirac morphism.  A \emph{Dirac homogeneous space} $(H/K,\P,L)$ for $(H,\A,E)$ is a Dirac module over an $H$-homogeneous space $H/K$, such that the group $\A/E$ acts transitively on $\P/L$.
\end{dfn}
\begin{rmrk} This definition differs from the notion of ``Dirac homogeneous space" presented in \cite{JotzDir}, being a Dirac structure in $\T(H/K)$ such that $(\T H,E)$ acts on it via a forward Dirac map.  For details and classification of this type of object, see Section 4 in \cite{JotzDir}.
\end{rmrk}
The main result of this chapter will be Theorem \ref{dirachom}, giving a classification of Dirac homogeneous spaces in terms of $K$-invariant $\beta$-coisotropic subalgebras $\c\subset\d$ such that $\c\cap\h=\k$.  This follows closely from Theorem \ref{LAhomgsp}:  since $\A/E\cong E^*$, and $\P/L\cong L^*$, the condition that $\A/E\act\P/L$ transitively is equivalent to $L$ being an $LA$ homogeneous space for $E$, which are classified by subalgebras $\c\subset\d$ with $\c\cap\h=\k$.

\begin{lem}\label{CAisDirac} Let $(H,\A,E)$ be a Dirac Lie group, and let $(\P,L)$ be a Manin Pair over $H/K$ such that $\P$ is a $CA$-module for $\A$, and $E$ acts on $L$.  Then the groupoid action
    \[ \mathrm{Act}:\,(\A,E)\times(\P,L)\da (\P,L) \]
is a Dirac morphism.
\begin{proof}
All that remains to show is that for every $x\in L_{h_1h_2K}$, there exists a unique $(\zeta,x')\in (E\times L)_{(h_1,h_2K)}$ such that $(\zeta,x')\sim_{\mathrm{Act}}x$.  By Proposition \ref{EandP},
    \[ (\P,L) = (H\times_K\p,H\times_K\l).\]
Hence by Equation \eqref{VBac}, any $[h_1h_2,x]\in L_{h_1h_1K}$ is uniquely Act-related to
    \[ (\zeta,x') = \big((h_1,\u([h_2,x]),[h_2,x]\big).\qedhere\]
\end{proof}
\end{lem}

Suppose $(H/K,\P,L)$ is a Dirac module for $(H,\A,E)$, and let $\pi:\,H\to H/K$ be the natural projection.  Then $(\tilde{\P},\tilde{L})=(\pi^!\P,\pi^!L)$ is a module for $(\A,E)$ as shown in Example \ref{pullbmod}.

\begin{prop}\label{fpbeta} Let $(H,\A,E)$ be a Dirac Lie group , and $(H,\pi^!\P,\pi^!L)=(H,H\times\tilde{\p},H\times\tilde{\l})$ the pullback of the Dirac homogeneous space $(H/K,\P,L)$ along the projection $\pi:\,H\to H/K$.  Then the map $f_{\tp}:\,\tilde{\p}\to \d$, given by $f_{\tp}:=\a_e\oplus\u_e$ satisfies
\[f_{\tp}\circ {f_{\tp}}^* = \beta^\#.\]
\begin{proof}
By Theorem \ref{almDmod} and Lemma \ref{lamequiv}, we have $\u\circ\u^*=\pg\circ\beta^\#|_{\g^*}$.  Next, we compute $\a_e\circ\u_e^*$ and $\u_e\circ \a_e^*$.\\  For this computation, let $\a_\A$ and $\a_\P$ denote the anchors of $\A$ and $\P$ respectively. For any $\alpha\in\g^*$, we have
\begin{align*}
        \a_{\P,e}(\u_e^*(\alpha)) &= \a_{\P,e}(\t_e^*(\alpha)\circ 0_e)\\
                                  &= \a_{\A,e}(\t_e^*(\alpha))\cdot \a_{\P,e}(0_e)\quad (\text{product\,in\, }\mathrm{T}H)\\
                                  &= \a_{\A,e}(\t_e^*(\alpha))\\
                                  &= \ph(\beta^\#(\alpha)),
\end{align*}
and hence $\a_e\circ\u_e^* = \ph\circ\beta^\#|_{\g^*}$.  Next, for any $\alpha\in\g^*,\,\nu\in\h^*$, we see
\begin{align*}
\<\u_e\circ\a_e^*(\nu),\alpha\> &= \<\a_e^*(\nu),\u_e^*(\alpha)\>\\
                               &=  \< \nu, \ph(\beta^\#(\alpha))\> \\
                               &= \< \beta^\#(\ph^*(\nu)),\alpha\>\\
                               &= \< \beta^\#(\nu),\alpha\>,
\end{align*}
where we used that $\ph^*:\,\h^*\hookrightarrow\d^*$ is the inclusion.  Since $\beta^\#$ takes $\h^*=\ann(\g)$ to $\g$, and $\alpha\in\g^*$ was arbitrary, this gives us $\u_e\circ\a_e^* = \beta^\#|_{\h^*}$.\\
Finally, let $\eta\in\d^*\cong\g^*\oplus\h^*$, and decompose as $\eta = \eta_{\g^*} + \eta_{\h^*}$; then
\begin{align*}
f_{\tp}\circ {f_{\tp}}^*(\eta) &= (\a_e+\u_e)\circ(\a_e^*+\u_e^*)(\eta_{\g^*} + \eta_{\h^*})\\
                     &= \a_e\circ\u_e^*(\eta_{\g^*}) + \u_e\circ\a_e^*(\eta_{\h^*}) + \u_e\circ\u_e^*(\eta_{\g^*}) \\
                     &= \pg\beta^\#(\eta_{\g^*}) + \beta^\#(\eta_{\h^*}) + (1-\ph)\beta^\#(\eta_{\g^*})\\
                     &= \beta^\#(\eta).\qedhere
\end{align*}
\end{proof}
\end{prop}

\subsection{From Homogeneous Spaces to $\c$}\label{Homtoc}
Let $(H,\A,E) = (H,H\times\q,H\times\g)$ be the Dirac Lie group associated to the Dirac Manin triple $(\d,\g,\h)_\beta$.  Let $K$ be a closed Lie subgroup of $H$, and let $(H/K,\P,L)$ be a Dirac homogeneous space for $(H,\A,E)$, with moment map $\u$.  Let $\p=\P_{eK}$, and $\l=L_{eK}$ so that $\p$ is $K$-invariant a metrized vector space with $\l\subset\p$ a $K$-invariant Lagrangian subspace. From Proposition \ref{EandP}, we know that
\[ (\P,L) \cong (H\times_K\p, H\times_K \l). \]
The moment map $\u$ injects $\l$ into $\g$ by Proposition \ref{lininj}.  If $\pi:\,H\to H/K$ is the natural projection, we consider the pullback Courant and Lie algebroids $(\pi^!\P,\pi^!L)$.  As these are also modules for $(H,\A,E)$ by Example \ref{pullbmod}, they must be of the form
\[ (\pi^!\P,\pi^!L) \cong (H\times\tilde{\p},H\times\tilde{\l}) \]
as a result of Theorem \ref{almDmod}.  By Proposition \ref{trivLiealgstr}, $\tl$ is a Lie algebra; indeed, the same argument shows that $\tp$, viewed as $H$-invariant sections of $\pi^!\P$, is closed under the Courant bracket.  As these are the sections which are constant relative to the trivialisation of $\pi^!\P$ by $E$, we can conclude from property 3. in Definition \ref{courdef} that the Courant bracket restricted to $\Gamma(\pi^!\P)^H$ is a Lie bracket, and hence $\tp$ is a Lie algebra.  Let $f_{\tp}:= \u_e\oplus\a_e:\, \tilde{\p}\to \d$; by Proposition \ref{uplusa}, it is a Lie algebra morphism, and by Proposition \ref{fpbeta}, $f_{\tp}\circ {f_{\tp}}^* = \beta^\#$.  Recall the formula for the pullback Lie algebroid:
    \[ \pi^!L := \{ (x,\zeta)\in \mathrm{T}H\times L\,|\, \mathrm{T}\pi(x) = \a(\zeta)\}. \]
Since $\k=\ker((\mathrm{T}\pi)_e)$, we have $\k\subset\tilde{\l}$ and $\tilde{\l}/\k=\l$; the maps $\u:\l\to\g$ and $\a_e:\,\k\to\h$ are injective, we conclude that $f_p$ injects $\tilde{\l}$ into $\d$ as a Lie subalgebra.  Hence, by Theorem \ref{qlinhom}, $(\tilde{\p},\tilde{\l})$ is a linear homogeneous space for the metrized linear groupoid $\d\times\d^*_\beta\gp\d$. In particular, this means $f_{\tp}$ injects $\tilde{\l}$ as a $\beta$-coisotropic subalgebra of $\d$.  We can see that $f_{\tp}(\tilde{\l})\cap\h=\k$, since the anchor map on $\pi^!L$ is the projection to the T$H$ factor, and the kernel of of $\mathrm{T}\pi_e$ is precisely $\k$.  Since $\l$ has a $K$-action, $\tilde{\l}$ inherits a $K$-action; since $f_{\tp}$ is $K$-equivariant, we conclude $f_{\tp}(\tilde{\l})$ is $K$-invariant.\\
We define
    \[\c:=f_{\tp}(\tilde{\l})\subset\d.\]

\subsection{From $\c$ to Dirac Homogeneous Spaces}
We now use $K$-invariant $\beta$-coisotropic subalgebras $\c\subset\d$ with $\c\cap\h=\k$ to construct Dirac homogeneous spaces $(H/K,\P,L)$.  We proceed by a reduction in stages: parallel to the construction of $\A$ as the coisotropic reduction $\A=C/C^\perp$, we construct a Courant algebroid of the form $D/D^\perp$ from $D\subseteq \T H\times\bar{\q}$, and then further reduce to $D_1/{D_1}^\perp$ for $D_1\subseteq \T(H/K)\times\bar{\q}$ defined as $D_1 = D//K$.  If $K=\{e\}$, then $(\P,L) = (D/D^\perp,H\times\c)$ is the homogeneous space associated to $\c$, and if $K$ is nontrivial then $(\P,L)=(D_1/{D_1}^\perp,H\times_K(\c/\k))$ is the homogeneous space.  In this section, we will use the identifications
    \begin{equation}\label{tangid} \T H = H\times(\h\oplus\h^*), \quad \T(H/K) = H\times_K(\h/\k\oplus\ann(\k)).\end{equation}
Let $(H,\A,E)$ be a Dirac Lie group, with associated Dirac Manin triple $(\d,\g,\h)_\beta$.  Let $K$ be a closed Lie subgroup of $H$ with Lie algebra $\k$, and let $\c\subseteq\d$ a $K$-invariant $\beta$-coisotropic subalgebra such that $\c\cap\h=\k$.  By Proposition 3.8 in \cite{lmdir}, $\A$ is constructed by considering the $CA$-groupoid
    \[\T H\times\bar{\q}\times\q\gp\h^*\times\q,\]
and taking the quotient $C/C^\perp$ where
   \begin{equation}\label{defC} C = \{(v,\alpha;\zeta,\zeta')\,|\,\iota_v\theta^L = f(\zeta')-\Ad_{h^{-1}}f(\zeta)\} \end{equation}
is an involutive, coisotropic subgroupoid, with $C^{(0)}=\h^*\times\q$.  The orthogonal $C^\perp$ is a subgroupoid with $(C^\perp)^{(0)} = \d^*$, embedded by the map
    \[ \mu \mapsto (-\mathrm{pr}_{\h^*}\mu,f^*(\mu)),\]
and hence $C/C^\perp$ is a groupoid over $(\h^*\times\q)/\d^*\cong \g$ as required.  The reduction
    \[ E:= (\mathrm{T}H\times\g\oplus\g)\cap C / (\mathrm{T}H\times\g\oplus\g) \cap C^\perp \]
is a Dirac structure in $\A$, making $(H,\A,E)$ a Dirac Lie group.\\
The Courant algebroid $\T H\times\bar{\q}$ is a $CA$-module for $\T H\times\bar{\q}\times\q$ by the moment map $\u((v,\alpha),\xi) = (\Ad_h\alpha,\t(\xi))$ for $((v,\alpha),\xi)\in \T_hH\times\bar{\q}$.  $\T H\times\bar{\q}$ contains an isomorphic\footnote{Using a restriction of the isomorphism $H\times\d\cong \mathrm{T}H\times\g$ from Proposition \ref{THgHd}} copy of $H\times\c$ given by
    \[ H\times\c \cong \{(v,\zeta)\in \mathrm{T}H\times\g\,|\,\Ad_{h^{-1}}\zeta+\iota_v\theta^L\in \c\} . \]
This can be written as $D\cap(\mathrm{T}H\times\g)$ with
    \begin{equation}\label{Ddef} D = \{(v,\alpha,\xi)\in \T H\times\bar{\q}\,|\,\Ad_{h^{-1}}f(\xi)+\iota_v\theta^L\in\c\}, \end{equation}
since we have
    \begin{equation}
    \label{DcapTHg}
    \begin{split}
    D\cap \big(\mathrm{T}H\times\g\big) &= \{ (v,0,\zeta)\in \T H\times\bar{\q}\,|\,\Ad_{h^{-1}}f(\zeta)+\iota_v\theta^L\in \c,\,\zeta\in\g\}\\
                                    &= \{ (v,0,\zeta)\,|\,\Ad_{h^{-1}}\zeta+\iota_v\theta^L\in \c\} \\
                                    &\cong H\times\c,
    \end{split}
    \end{equation}
as $f|_\g = \mathrm{id}$.
\begin{prop}\label{DredKe} $D$ is a coisotropic, involutive sub-bundle of $\T H\times\bar{\q}$, of rank $(\dim\c+\dim\d)$.  It is a $VB$-module over $C\gp\h^*\times\q$ with $\u(D^\perp)\subseteq (C^\perp)^{(0)}$, and it follows $\tilde{\P}:=D/D^\perp$ is a $CA$-module over $\A=C/C^\perp$, of rank $2\dim\c$.  The reduction of $\mathrm{T}H\times\g$ with respect to $D$ is Dirac structure $\tilde{L}\subset \tilde{\P}$, canonically isomorphic to $H\times\c$, which is an $E$-module for the Dirac structure $E\subseteq\A$.
\begin{proof}
\begin{enumerate}
\item First, we will show that $D$ is coisotropic.  By Equation \eqref{DcapTHg}, we have
    \[ D\cap\big(\mathrm{T}H\times\g\big)=H\times\c.\]
    Furthermore, $D$ contains a copy of T$^*H$ as elements $(0,\alpha,0)$, as well as a copy of $H\times\r$ with $\r=f^{-1}(\h)\subset\q$ embedded as elements $(v,0,\tau)$ with $\tau\in \r$ and $\iota_v\theta^L_h= -\Ad_{h^{-1}}f(\tau)$.  Since every element of $D$ can be decomposed uniquely as a sum of elements in each of these subsets, we can write $D$ as a direct sum

    \begin{equation}\label{Ddecomp} D\cong (H\times\c)\oplus\mathrm{T}^*H\oplus (H\times\r) \end{equation}
and so $D$ is a sub-bundle of rank $\dim\c+\dim\h+\dim\r=\dim\c+\dim\d$.  Hence, $D^\perp$ has rank $\dim\d-\dim\c$. We claim it is described by
    \begin{equation}\label{Dperpdf} D^\perp = \left\{ \big(0, -\<\theta^L_h,\delta\>,f^*(\Ad_h\delta)\big)\,|\,\delta\in\ann_{\d^*}(\c)\right\}. \end{equation}
Indeed, for any $(v,\alpha,\zeta)\in D$, we have
    \[ -\<\iota_v\theta^L_h,\delta\> - \<\zeta,f^*(\Ad_h\delta)\> = -\big(\<\iota_v\theta^L_h + \Ad_{h^{-1}}f(\zeta),\delta\>\big) = 0. \]
Furthermore, the right hand side of Equation \eqref{Dperpdf} is the correct rank as $\<\theta_h^L,\delta\>=0$ if and only if $\delta\in\ann_{\d^*}(\h)$, and hence $\Ad_h\delta\in \ann_{\d^*}(\h)$.  The map $f^*:\d^*\to\q^*$, however, restricts to an isomorphism $\ann_{\d^*}(\h)\to \g^* = r^\perp$, so
    \[ \delta \mapsto (-\<\theta^L_h,\delta\>,f^*(\Ad_h\delta)) \]
is injective, and rk$D^\perp = \dim(\ann_{\d^*}(\c)) = \dim\d - \dim\c$ as required.  We have $D^\perp\subset D$ since $\c$ is $\beta$-coisotropic, and $\beta^\#$ is $H$-equivariant:
    \[ \Ad_{h^{-1}}f(f^*(\Ad_h\delta)) = \Ad_{h^{-1}}\beta^\#(\Ad_h(\delta)) = \beta^\#(\delta)\in \beta^\#(\ann_{\d^*}(\c)) \subset \c. \]
\item Next, we'll show that $D$ is a module for the groupoid action of $C$:  let $(v,\alpha;\zeta,\zeta')\in C_h$ and $(v_1,\alpha_1;\zeta_1)\in D_{h_1}$ be composable.  Thus, we must have
    \[ \zeta' = \zeta_1,\qquad \alpha = \Ad_{h_1}\alpha_1. \]
Using left trivialisation, the composition is
    \[ (v,\alpha;\zeta,\zeta_1)\circ(v_1,\alpha_1;\zeta_1) = (\Ad_{h_1^{-1}}v+v_1,\alpha_1; \zeta). \]
To verify that this is in $D$, we note (using $\Ad_{h^{-1}}f(\zeta) = f(\zeta')-v$ by definition of $C$)
    \begin{equation}
    \label{CactonD}
    \begin{split}
        \Ad_{h_1^{-1}}\Ad_{h^{-1}}f(\zeta) + \Ad_{h_1^{-1}}v+v_1 &= \Ad_{h_1^{-1}}(f(\zeta')-v)+\Ad_{h_1^{-1}}v + v_1\\
                                    &= \Ad_{h_1^{-1}}f(\zeta_1)+v_1 \,\,\in\c.
    \end{split}
    \end{equation}
We also notice that $\u(D^\perp)\subset (C^\perp)^{(0)}$ as
    \[ \u\big(0, -\<\theta^L_h,\delta\>,f^*(\Ad_h\delta)\big) = (-\mathrm{pr}_{\h^*}\Ad_h\delta, f^*(\Ad_h\delta)). \]
Consequently, $D/D^\perp$ is a $C/C^\perp$ module (by Lemma \ref{VBquotactt}).
\item To show that $D$ is involutive, we note by Equation \eqref{Ddecomp} that the space of sections of $D$ is spanned by sections of the form
    \[ \varphi(\zeta) = (\a(\zeta),\pg(\Ad_h\zeta)) \]
for $\zeta\in \c$, sections of the form
    \[\psi(\tau) = (-f(\tau)^R,\tau) \]
for $\tau\in\r$, and sections $\mu\in \Gamma(\mathrm{T}^*H)$ .  As shown in Proposition \ref{phipsimix}, both $\varphi,\,\psi$ are bracket preserving.  In $\T H\times\q$, the bracket of any section of $\mathrm{T}^*H$ with any other section is again a section of $\mathrm{T}^*H$:
    \[ \Cour{(0,\mu_1,0),(v,\mu_2,\xi)} = (0, -\iota_v \mathrm{d}\mu_1, 0) \]
and hence, we only need consider the Courant brackets of the form $\Cour{\psi(\tau), \varphi(\zeta)}$.  But by Proposition \ref{phipsimix}, this is
    \[\Cour{\psi(\tau), \varphi(\zeta)} = \psi\big(\pr[\tau,\pg(\Ad_h\zeta)]\big) \in \Gamma(D).\]
Similarly, $\Cour{\varphi(\zeta), \psi(\tau)} = -\Cour{\psi(\tau), \varphi(\zeta)} + \a^*(\mathrm{d}\<\varphi(\zeta),\psi(\tau)\>)\, \in\, \Gamma(D)$, since $D$ contains the range of $\a^*$.  Hence $D$ is involutive, and so $D/D^\perp$ is a Courant algebroid by Proposition \ref{CAcoisred}. The graph of the groupoid action of $\A=C/C^\perp$ on $\P=D/D^\perp$ is obtained by reduction from the graph of the groupoid action of $\T H\times\bar{\q}\times\q$ on $\T H\times\bar{\q}$, and hence it is a $CA$-morphism $\A\times\P\da\P$.

\item Lastly, we consider the reduction of $\mathrm{T}H\times\g$.  The map $\psi$ injects $H\times\r$ into $C$, and the sum of the sub-bundles gives
\[\mathrm{T}H\times\g + \psi(H\times\r) = \mathrm{T}H\times\q.\]
Thus, we have $D+(\mathrm{T}H\times\g) = \T H\times\q$, which gives us $D^\perp\cap(\mathrm{T}H\times\g)=0$.  But the intersection of $D$ with $\mathrm{T}H\times\g\cong H\times\d$ is $H\times\c$.  This means that the reduction of T$H\times\g$ is $\tilde{L}\cong H\times\c$, embedded as a Dirac structure of $\tilde{\P}$.  $\tilde{L}$ carries an $E$-action, as if $(v,0,\zeta,\zeta')\in ((\mathrm{T}H\times\g\oplus\g)\cap C)_h$, $(v_1,\zeta')\in ((\mathrm{T}H\times\g)\cap D)_{h_1}$, then in left trivialisation
    \[ (v,0,\zeta,\zeta')\circ (v_1,\zeta') = (v_1 + \Ad_{{h_1}^{-1}}v, \zeta);\]
this is in $((\mathrm{T}H\times\g)\cap D)_{hh_1}$ by a similar computation to Equation \eqref{CactonD}, and hence the reduction yields an action of $E$ on $\tilde{L}$.
\end{enumerate}
\end{proof}
\end{prop}

\begin{cor}\label{DDpDirmod} The module $(H,\tilde{\P},\tilde{L}) = (H, D/D^\perp,H\times\c)$ is a Dirac module for $(H,\A,E)$.
\begin{proof}
This follows from Lemma \ref{CAisDirac} since the action $\A\act\tilde{P}$ is a $CA$-morphism, and $E$ acts on $L$.
\end{proof}
\end{cor}

We will now use the action of the subgroup $K$ to further reduce: $K$ acts on $\T H\times\bar{\q}$, lifting the action $k.h = hk^{-1}$ on $H$, and the trivial action on $\q$. The generators for this action come from the inclusion
    \begin{equation}\label{Kgeninj} \k\to\Gamma(\T H\times\bar{\q}), \quad \kappa\mapsto ((\kappa^L,0);0) = \varphi(\kappa)\end{equation}
and are clearly isotropic.  Using this action, we can consider the reduction of $\T H\times\bar{\q}$.  Since, in left trivialisation, $\varphi(H\times\k) = (H\times(\k\times \{0\}))\times \{0\}$, we can conclude
    \[ \varphi(H\times\k)^\perp = (H\times(\h\times\ann(\k)))\times\bar{\q} \]
and hence $(\T H\times\bar{\q})//K = \T(H/K)\times\bar{\q}$.  Considering the reduction of $D\subset\T H\times\bar{\q}$, we notice that $\varphi(H\times\k)\subset D$, and
    \[ D\cap\varphi(H\times\k)^\perp = \{((h,v,\alpha);\zeta)\,|\,\alpha\in\ann(\k),\,v+\Ad_{h^{-1}}f(\zeta)\in\c\}. \]
Thus, the reduction gives us
    \[ D_1:= D//K = \{([h,(-\Ad_{h^{-1}}f(\zeta))_{\h/\k},\alpha];\zeta)\,|\,f(\zeta)\in\Ad_h(\c+\h)\} \subset \T(H/K)\times\bar{\q} \]
which, by virtue of the reduction, is an involutive sub-bundle. As a note, the subscript ${\h/\k}$ denotes the image under the map
    \[(\c+\h) \to (\c+\h)/\c \cong \h/(\c\cap\h) = \h/\k.\]
\begin{prop}
${D_1}$ is a coisotropic subbundle of $\T(H/K)\times\bar{\q}$ of rank $\dim\h/\k+\dim\g+\dim \c/\k$.  Its perpendicular is given by
  \begin{equation*}
  {D_1}^\perp = \{ ([h,0,-\Ad_{h^{-1}}\mathrm{pr}_{\h^*}(\delta)];f^*(\delta))\,|\,\delta\in \ann_{\d^*}(\Ad_h\c)\}
  \end{equation*}
and satisfies $\a({D_1}^\perp) = 0$.
\begin{proof}
\begin{enumerate}
\item The rank and coisotropicity of $D_1$ follows from the reduction:  $D$ is coisotropic, and since $\varphi(H\times\k)\cap D^\perp=0$, we have $\varphi(H\times\k)^\perp+D=\T H\times\bar{\q}$.  Hence,
        \begin{align*} \rk\big(D\cap\varphi(H\times\k)^\perp\big) &= \rk(D)+\rk\big(\varphi(H\times\k)^\perp\big) - \rk(\T H\times\bar{\q})\\
                            &=  \rk(D) - \rk(H\times\k),
        \end{align*}
    and
        \[ \rk\big(D\cap\varphi(H\times\k)\big) = \rk(H\times\k),\]
    and so we conclude that $\rk(D^\perp) = \rk(D) - 2\dim\k$.  Since $D$ is coisotropic, $D_1$ must be coisotropic by reduction.

\item Next, we give a description of $D^\perp$. Since rk$D_1 = \dim\d + \dim \c - 2\dim\k$, we have
   \begin{align*}
   \mathrm{rk}{D_1}^\perp &= \mathrm{rk}(\T(H/K)\times\bar{\q}) - \mathrm{rk}(D_1) \\
                        &= 2\dim(h/k)+2\dim\g - (\dim\d + \dim \c - 2\dim\k) \\
                        &= \dim\d - \dim \c.
   \end{align*}
  Let $B=\varphi(H\times\k)^\perp\subseteq \T H\times\bar{\q}$, and $\Phi: B\to B/B^\perp=\T(H/K)\times\bar{\q}$ the quotient map.  We have
            \[ {D_1}^\perp = \Phi(D)^\perp = \Phi(B^\perp + D^\perp) = \Phi(D^\perp).\]
  With the description
    \[ D^\perp = \left\{ \big(0, -\<\theta^L_h,\delta\>,f^*(\Ad_h\delta)\big)\,|\,\delta\in\ann_{\d^*}(\c)\right\}, \]
  from Equation \eqref{Dperpdf}, we have
    \begin{equation}\label{D1perp}
       {D_1}^\perp = \{ ([h,0,-\Ad_{h^{-1}}\mathrm{pr}_{h^*}(\delta)];f^*(\delta))\,|\,\delta\in \ann_{\d^*}(\Ad_h\c)\} .
    \end{equation}

  \item Finally, since the anchor map on $\T(H/K)\times\bar{\q}$ is the projection to the tangent bundle factor, it is clear that $\a({D_1}^\perp)=0$.
  \end{enumerate}
 \end{proof}
\end{prop}

\begin{prop}  Given $D\subset \T H\times\bar{\q}$ in Equation \eqref{Ddef}, with $D_1:= (D//K)$, then:
    \[ D_1 / D_1^\perp = (D/D^\perp)//K. \]
\begin{proof}
For brevity, define $B:= \varphi(H\times\k)^\perp$.  Let $\Phi:\,B\to B/B^\perp$ and $\mu:\,D\to D/D^\perp$ be the natural projections.  We have $B^\perp\subset D$, and $D^\perp\cap B^\perp=0$, and so by Lemma \ref{redinstage} we have
    \begin{align*} (D//K)/(D//K)^\perp &= \big(\Phi(D)/\Phi(D)^\perp\big) / K \\
                                       &= \big((D\cap B) / (D\cap B)^\perp \big)/ K\\
                                       &= \big( \mu(B) / \mu(B)^\perp\big) / K\\
                                       &= (D/D^\perp)//K. \qedhere
    \end{align*}
\end{proof}
\end{prop}

\begin{cor} The coisotropic reduction ${D_1}/{D_1}^\perp$ contains $L:=\tilde{L}//K$ as a Dirac structure, characterised as follows:
    \begin{equation}\label{Linrd} L = \Big\{\Big([h,\,-(\Ad_{h^{-1}}\xi)_{\h/\k}],\,\xi\Big) \,\big|\,\, \xi\in \g\cap\Ad_h(\c+\h)\,\,\Big\}. \end{equation}
\begin{proof}
Since the Dirac structure $\tilde{L}\subset D/D^\perp$ contains $\varphi(H\times\k)$, its reduction $\tilde{L}//K$ is a Dirac structure in $(D/D^\perp)//K$ by \cite{iscourred}.  Since $\tilde{L}\cong H\times\c$ as Lie algebroids, this gives us that
    \[ L:= \tilde{L}//K \cong (H\times\c)//K = H\times_K(\c/\k) \]
which embeds into $\T(H/K)\times\bar{\q}$ via Equation \eqref{Linrd}, since
    \[ \c/\k = \c/(\c\cap\h) \cong (\c+\h)/\h.\]
\end{proof}
\end{cor}

\begin{prop}\label{deq}  The Manin pair $(\P,L)=(D_1/{D_1}^\perp, \tilde{L}//K)$ over $H/K$ is a Dirac module for $(H,\A,E)$.
\begin{proof}
By Corollary \ref{DDpDirmod}, $(H,D/D^\perp,\tilde{L})$ is a Dirac module for $(H,\A,E)$.  The graph of the groupoid action of $\A$ on $D_1/{D_1}^\perp$ is obtained by the reduction of the groupoid action of $\A$ on $D/D^\perp$, hence defining a $CA$-action
    \[ \A\times \P \da \P.\]
Alternatively, since $D_1\subset \T(H/K)\times\bar{\q}$, the action of $\A$ on $\P$ is obtained by reduction of the graph of the groupoid action of $\T H\times\bar{\q}\times\q$ on $\T(H/K)\times\bar{\q}$: $C$ acts on $D$ by Proposition \ref{DredKe}, and so the reduction yields an action of $C$ on $D_1$ by the formula
\begin{align*}
  ((h_1,f(\zeta')&-\Ad_{h_1^{-1}}f(\zeta),\Ad_{h_2}\alpha);\zeta,\zeta')\circ([{h_2},-\Ad_{{h_2}^{-1}}f(\zeta')\md \c,\alpha],\zeta')\\
   &=([h_1h_2, \Ad_{{h_2}^{-1}}\big(f(\zeta')-\Ad_{{h_1}^{-1}}f(\zeta)\big)-\Ad_{{h_2}^{-1}}f(\zeta')\md \c,\alpha];\zeta)\\
   &=([h_1h_2, -\Ad_{(h_1h_2)^{-1}}f(\zeta)\md \c,\alpha];\zeta)
  \end{align*}
using the identifications in Equation \eqref{tangid}.  Since $\zeta'$ satisfies $f(\zeta')\in \Ad_{h_2}(\c+\h)$, and $f(\zeta')-\Ad_{h_1^{-1}}f(\zeta)\in \h$, it follows that $f(\zeta)\in \Ad_{h_1h_2}(\c+\h)$, and hence
    \[\big([h_1h_2, -\Ad_{(h_1h_2)^{-1}}f(\zeta)\md \c,\alpha];\zeta\big)\in D_1.\]
In this way, parallel to Proposition \ref{DredKe}, we get an action of $\A=C/C^\perp$ on $D_1/{D_1}^\perp$, and an action of $E$ on $L$.  The result then follows from Lemma \ref{CAisDirac}.
\end{proof}
\end{prop}

\begin{lem}\label{D1hom} Consider the conditions in Proposition \ref{deq}.  Then $(H/K,D_1/{D_1}^\perp,H\times_K(\c/\k))$ is a Dirac homogeneous space for the Dirac Lie group $(H,\A,E)$.
\begin{proof}
Let $\p:=\P_{eK}$, and $\l:=\pg(\c)\cong\c/\k=L_{eK}$.  Since $\u_e:\,\pg(\c)\hookrightarrow\g$ with ${\u_e}\circ{\u_e^*}=\pg\beta^\#$, this means $\q/\g$ acts on $\p/\l$ transitively (Proposition \ref{lininj}), and so $\P/L$ is acted upon transitively by $\A/E$ (Proposition \ref{transiff}).
Hence, $(H/K, \P,L)$ is a Dirac homogeneous space for $(H,\A,E)$.
\end{proof}
\end{lem}

\begin{lem}\label{keeh} If $K=\{e\}$, and thus $\c\cap\h=0$, the Dirac module
    \[(H,\tilde{\P},\tilde{L}) = (H,D/D^\perp,H\times\c)\]
is an homogeneous space for $(H,\A,E)$.
\begin{proof}
By the construction of $(\tilde{\P},\tilde{L})$ in Proposition \ref{DredKe}, the moment map on $H\times\c$ at the identity is given by $\u_e=\pg$, the projection to $\g$ along $\h$.  If $\c\cap\h=0$, then we have
    \[ \c \cong \c/\{0\} = \c/\c\cap\h \cong (\c+\h)/\h \cong \pg\c. \]
and hence, $\u_e$ is an injection of $\c$ into $\g$, and satisfies $\u_e\circ\u_e^* = \pg\beta^\#$.  Thus, by Proposition \ref{lininj} $\tilde{\p}/\c$ is an homogeneous space for $\q/\g$, and so $\A/E$ acts transitively on $\tilde{\P}/\tilde{L}$ (Proposition \ref{transiff}).
\end{proof}
\end{lem}

\begin{prop}\label{tcs} Let $(H,\A,E)$ be a Dirac Lie group, with Dirac Manin triple $(\d,\g,\h)_\beta$, and let $(H,\P',L')$ be a $CA$-module for $(H,\A,E)$ such that $f_{\tp}:=\a_e\oplus\u_e$ is injective on $\l':=L'_e$.  If $\c:=f_{\tp}(\l')$, and $D$ defined in Equation \eqref{Ddef}, then $(\P',L')$ and $(D/D^\perp, H\times\c)$ are isomorphic as action Courant/Lie algebroids.
\begin{proof}
Since $L'$ is an $E$-module, then $L'=H\times\l'$ as an action Lie algebroid whose action is determined by its value at the identity via the groupoid action of $E$.  The map $f_{\tp}:\,\l'\to \d$ is an injective Lie algebra morphism (Proposition \ref{uplusa}), and hence $L' \cong H\times\c$ as action Lie algebroids via the map $f_{\tp}$.\\
The $E$ action trivialises the Courant algebroid $\P'$ as $H\times \p'$ by Proposition \ref{EandP}.  Since $f_{\tp}:\l'\hookrightarrow \d$ with $f_{\tp}\circ{f_{\tp}}^*=\beta^\#$, the pair $(\p',\l')$ must be a linear homogeneous space for the metrized linear groupoid $\d\times\d^*_\beta\gp\d$, meaning
    \[\p' = \s^{-1}(f_{\tp}(\l'))/\s^{-1}(f_{\tp}(\l'))^\perp.\]
We obtain a moment map $\tilde{\u}:({D/D^\perp},\c)\to \d$ via $\tilde{\u}=\a_e\oplus\u_e$; we know that on $\c$, $\a_e = \ph$ and $\u_e = \pg$, meaning this new moment map on $\c$ is just the inclusion into $\d$.  By the same token with $(\p',\l')$, we obtain a groupoid action of $\d\times\d^*_\beta\gp \d$ on $\tilde{\u}:{D/D^\perp}_e\to \d$.  Since $\tilde{\u}\circ\tilde{\u}^* = \beta^\#$, and $\tilde{\u}$ is injective on $\c$, this makes $({(D/D^\perp)}_e,\c)$ a homogeneous space, and so
    \[ {(D/D^\perp)}_e = \s^{-1}(\c)/\s^{-1}(\c)^\perp\ = \p' \]
with $\tilde{\u}(\overline{(x,\alpha)}) = x + \beta^\#(\alpha)$ for $\overline{(x,\alpha)}$ the image of $(x,\alpha)\in \s^{-1}(f_{\tp}(\l'))$ in the quotient, and with the anchor being recovered as the projection to $\h$.  Hence, $D/D^\perp$ and $\P'$ are equivalent as bundles, and share the same anchor map at the identity fibres; since the anchor map is determined completely by the value at the identity for these action Courant modules, this means that $\P' = D/D^\perp$ as Courant algebroids.
\end{proof}
\end{prop}

\section{The Classification}
We now have all of the pieces to classify the Dirac homogeneous spaces.
\begin{thm}\label{dirachom} Let $(H,\A,E)$ be the Dirac Lie group associated to the $H$-equivariant Dirac Manin triple $(\d,\g,\h)_\beta$, and let $H/K$ be an $H$-homogeneous space.  The Dirac homogeneous spaces $(H/K,\P,L)$ for $(H,\A,E)$ are classified by $K$-invariant $\beta$-coisotropic subalgebras $\c\subseteq\d$ such that $\c\cap\h=\k=\mathrm{Lie}(K)$.
\begin{proof}
\begin{enumerate}
\item When $K=\{e\}$:  start with such a Dirac homogeneous space $(\P,L)$ over $H$; as in Proposition \ref{EandP}, the action of $E$ trivialises $(\P,L)=(H\times \p,H\times\l)$, with
        \[\p=\s^{-1}(\l)/\s^{-1}(\l)^\perp.\]
    From Proposition \ref{uplusa} the map $f_{\p}:=\a_e\oplus\u_e$ is an injective Lie algebra morphism of $\l$ into $\d$, which satisfies $f_{\p}\circ f_{\p}^*=\beta^\#$.  This defines a moment map for an action of the groupoid $\d\times\d^*_\beta\gp\d$ on $(\p,\l)$, and by Proposition \ref{lininj}, $(\d\times\d^*_\beta)/\d$ acts transitively on $\p/\l$, which implies $\c:=f_{\p}(\l)$ must be a $\beta$-coisotropic subalgebra of $\d$.  Since $\u_e$ is an injection of $\l$ into $\g$, this means $f_{\p}(\l)\cap\h = 0$.  Taking this $\c$, and constructing the action Lie algebroid $L'=H\times\c$, and $\P'=D/D^\perp$ (via Equation \eqref{Ddef}), Proposition \ref{tcs} gives $(\P',L')\cong(\P,L)$, the homogeneous space we started with.\\
    Every such $\c$ arises in this way by Lemma \ref{keeh}, coming from the homogeneous space $(H,D/D^\perp,H\times\c)$.
\item When $K$ is nontrivial:  Start with a Dirac homogeneous space
        \[(H/K,\P,L)=(H/K,H\times_K\p,H\times_K\l),\]
    and consider the algebroid pullbacks $(\pi^!\P,\pi^!L)$ over $H$, for $\pi:\,H\to H/K$. By Example \ref{pullbmod}, $(H,\pi^!\P,\pi^!L)$ is a Dirac module for $(H,\A,E)$, and so by Proposition \ref{EandP} trivialises as
        \[  (H,\pi^!\P,\pi^!L) = (H,H\times\tilde{\p},H\times\tilde{\l})\]
    where $\tilde{\l}/\k\cong\l$.  The map $f_{\tp}$ injects $\tilde{\l}$ as a $K$-invariant $\beta$-coisotropic subalgebra of $\d$ by Proposition \ref{fpbeta}, with $f_{\tp}(\tilde{\l})\cap\h=\k$.  Define $\c:=f_{\tp}(\tilde{\l})$.  With $D$ from Equation \eqref{Ddef} we have $(D/D^\perp,H\times\c) = (\tilde{\P},\tilde{L})$, by Proposition \ref{tcs}.  We know that the reduction by $K$ gives
       \[ (D/D^\perp,H\times\c)//K = ( (D/D^\perp)//K, (H\times\c)//K ) = (D_1/D_1^\perp, H\times_K(\c/\k)) \]
    by Proposition \ref{deq}, and gives an homogeneous space by Lemma \ref{D1hom}.  Reducing $(\tilde{\P},\tilde{L})$ yields:
        \[ (\tilde{\P},\tilde{L})//K = ( (\pi^!\P)//K, (\pi^!L)//K) = (\P,L) \]
    via Proposition \ref{CourKred}.  Hence, $(D_1/D_1^\perp,H\times_L(\c/\k)) = (\P,L)$, the homogeneous space we began with.\\
    Every such $\c$ arises from the homogeneous space $(H/K,D_1/{D_1}^\perp,H\times_K(\c/\k))$ by Lemma \ref{D1hom}.
\end{enumerate}
\end{proof}
\end{thm}
In the statement of Theorem \ref{dirachom}, we have implicitly chosen a basepoint of the $H$-homogeneous space $M$, identifying $M\cong H/K$.  We thus get a map from $M$ to the variety of coisotropic subalgebras
    \begin{equation}\label{coisomap}
       M \longrightarrow \mathcal{C}(\d), \qquad m \mapsto \c_m
    \end{equation}
with $\c_m$ being the classifying subalgebra from Theorem \ref{dirachom} under the identification $M\cong H/K_m$ with $K_m = \mathrm{stab}(m)$; this map is the analogue to the Drinfeld map $M\to\L(\d)$ in the Poisson case \cite{Drinfeldp}.

\begin{thm}  Let $(M,\P,L)$ be a Dirac homogeneous space for $(H,\A,E)$.  For any point $m\in M$ with corresponding classifying subalgebra $\c_m\subset\d$ from Equation \eqref{coisomap}, and any $g\in H$,
    \[\Ad_g\c_m = \c_{g\cdot m}.\]
Hence, the isomorphism class of classifying subalgebras for $(M,\P,L)$ is the orbit $(H\cdot\c_m)$.
\begin{proof}
Fix $m\in M$ and $g\in H$, and let $\c:= \c_m$, $\c' := \Ad_g\c$, $K:= \mathrm{stab}(m)$, and $K'=\mathrm{stab}(g\cdot m)$.  Since $K_{g\cdot m} = \Ad_g K_m$, it follows that $\c'\cap\h = \k' = \Ad_g\k$, and so $\c'$ defines a Dirac homogeneous space on $H/K' \cong M$.  Define the action Lie algebroids $H\times\c$ and $H\times\c'$ with the action for each given by Equation \eqref{diracanchor} in left-trivialisation; i.e.: for $\xi\in \c$
    \[ \rho(\xi)_h = (h, \Ad_{h^{-1}}\ph\Ad_h \xi) \]
and similarly for $\xi'\in\c'$.  Define $\mu: H\times\c \to H\times\c'$ by $\mu(h,\xi) = (hg^{-1},\Ad_g\xi)$; since $\mu_*(h,x) = (hg^{-1},\Ad_gx)$ in left trivialisation, we have
    \[ \a'\circ \mu = \mu_* \circ \a \]
giving an isomorphism of action Lie algebroids $H\times\c\cong H\times\c'$, as $\Ad_g$ is a Lie automorphism of $\d$.  The groups $K$ and $K'$ act on $H\times\c$ and $H\times\c'$ respectively; since the generators for this action are the inclusions of $H\times\k$ and $H\times\Ad_g\k$, the map $\mu$ clearly intertwines the generators, giving us the Lie algebroid isomorphism
    \[ (H\times\c)//K \cong E \cong (H\times\c')//K'.\]
Similarly, defining $\mu:\T H\times\q\to \T H\times\q$ via $\mu((h,v,\alpha);\xi) = (hg^{-1},\Ad_gv,\Ad_g\alpha);\xi)$ in left-trivialisation, and recalling
    \[D = \{(v,\alpha,\xi)\in \T H\times\bar{\q}\,|\,\Ad_{h^{-1}}f(\xi)+\iota_v\theta^L\in\c\}  \]
(with $D'$ defined similarly for $\c'$), we have $\mu(D) = D'$.  As the inner product is $\Ad$-invariant, and the generators for the $K$ action on $D/D^\perp$ are the inclusion of $H\times\k$, we again conclude that
    \[ (D/D^\perp)//K \cong \P \cong (D'/D'^\perp)//K',\]
and hence $\c_{g\cdot m} = \Ad_g\c_m$.
\end{proof}
\end{thm}
\subsection{Examples}
In this section we will consider examples of Theorem \ref{dirachom} for specific Dirac Lie groups.
\begin{example} \textbf{Poisson Lie Groups} (adapted from Section 4, \cite{lwx2}). (Let $(H,\Pi)$ be a Poisson Lie group with associated $H$-equivariant Dirac Manin triple $(\d,\g,\h)_\beta = (\h\bowtie\h^*,\h^*,\h)_\beta$, where $\beta$ is given by the pairing.  Let $(H/K,\Pi_{H/K})$ be a Poisson homogeneous space with corresponding Manin pair $(\T(H/K),\mathrm{Gr}_{\Pi_{H/K}})$. The Poisson action
    \[(H,\Pi)\times (H/K,\Pi_{H/K}) \longrightarrow (H/K,\Pi_{H/K}) \]
 lifts to a Dirac morphism
    \[ (H\times\d,H\times\h^*)\times (\T(H/K),\mathrm{Gr}_{\Pi_{H/K}})\da (\T(H/K),\mathrm{Gr}_{\Pi_{H/K}}) \]
 Considering pullbacks $(\pi^!\T(H/K),\pi^!\mathrm{Gr}_{\Pi_{H/K}})=(\T H,L')$, which by Example \ref{pullbmod} trivialises to
    \[(\T H,L')=(H\times\d,H\times\tilde{\l}).\]
 Hence, the $\c$ we recover is $\tilde{\l}$, which is necessarily Lagrangian in $\d$ as $L'$ is Lagrangian in $\T H$.  This coincides with Drinfeld's theorem.
\end{example}
\begin{example}\label{standardex} \textbf{The Standard Dirac Lie Group} (Example 5.2.1, \cite{lmdir}). For any Lie group $H$, we have the $H$-equivariant Dirac Manin triple $(\h\ltimes\h^*,\h^*,\h)_\beta$, with $\beta$ given by the pairing; in essence, this is the Poisson case for $\Pi_H=0$.  Let $K$ be a closed Lie subgroup of $H$, with Lie algebra $\k$.
    \begin{enumerate}
        \item The Lie subalgebra $\c:=\k\ltimes\ann(\k)$ is Lagrangian in $\d$ and satisfies $\c\cap\h=\k$, and so generates the homogeneous space
            \[ (\P,L) \cong \big( H\times_K (\s^{-1}(\ann(\k))/\s^{-1}(\ann(\k))^\perp), H\times_K(\ann(\k))\big). \]
        \item Since $\h^*\subset\d$ is an abelian Lie algebra, the coadjoint action of $\h^*$ on $\h$ is trivial, and hence $\c:= \k\ltimes\h^*$ is a subalgebra of $\d$.  It is coisotropic because it contains the Lagrangian subalgebra $\k\ltimes\ann(\k)$, and it satisfies $\c\cap\h=\k$.  This gives rise to the homogeneous space
            \[ (\P,L) \cong (H\times_K \d, H\times_K\h^*) = (\A,E)/K. \]
        \item If $\mathfrak{j}\subseteq\k$ is an ideal (in $\k$), then $\c:= \k\ltimes\ann(\mathfrak{j})$ is a subalgebra of $\d$.  It contains the Lagrangian subalgebra $\k\ltimes\ann(\k)$, and satisfies $\c\cap\h=\k$, generating the homogeneous space
                \[ (\P,L) \cong \big(H\times_K (\s^{-1}(\ann(\mathfrak{j}))/\s^{-1}(\ann(\mathfrak{j}))^\perp),H\times_K(\ann(\mathfrak{j}))\big). \]
    \end{enumerate}
\end{example}
\begin{example} \textbf{The Cartan-Dirac structure} (Example 5.2.2, \cite{lmdir}).  Let $G$ be a Lie group with an invariant metric on $\g$. The Dirac Manin triple   \[(\overline{\g}\oplus\g,\g_\triangle,0\oplus\g)_\beta\]
with $\beta$ given by the metric on $\overline{\g}\oplus\g$ gives rise to a Dirac Lie group structure on $H=\{e\}\oplus G$.
    \begin{enumerate}
        \item Let $\g$ be a semi-simple Lie algebra with the Killing form, and let $\k\subset\g$ be an ideal.  Then the subalgebra
            \[ \c:= \g_\triangle + (\k\oplus\k) \subset \overline{\g}\oplus\g\]
        is coisotropic as it contains the Lagrangian subalgebra $\g_\triangle$.  Then we have
            \[ \c\cap \h = \c\cap (0\oplus\g) = 0\oplus \k,\]
        and exponentiating $0\oplus\k$ gives us a closed, connected Lie subgroup $K\leqslant H$.  Hence, we get a Dirac homogeneous space over $H/K$:
            \[(H/K,\P,L) = (H/K, H\times_K\p, H\times_K(\c/(0\oplus\k)).\]
        \item Let $\g$ be a complex, semi-simple Lie algebra with the Killing form.  Decompose $\g$ as
                \[ \g = \mathfrak{t}\oplus\bigoplus_\alpha \g_\alpha = \mathfrak{t}\oplus\mathfrak{n}_-\oplus\mathfrak{n}_+ \]
             for $\g_\alpha$ the root spaces.  Take $\c\subseteq \overline{\g}\oplus\g$ to be
                \[ \c:=(\mathfrak{n}_+\oplus 0) + (0\oplus \mathfrak{n}_-) + (\mathfrak{t}_\triangle),\]
             then $\c$ is a subalgebra, since
             \[ [ (\mathfrak{n}_+\oplus 0), \mathfrak{t}_\triangle] \subseteq (\mathfrak{n}_+\oplus 0)\]
             and similarly
                \[[(0\oplus \mathfrak{n}_-),\mathfrak{t}_\triangle]\subseteq (0\oplus \mathfrak{n}_-).\]
             It is also Lagrangian, since
                \[ \<\mathfrak{n}_+,\mathfrak{t}\> = 0 = \< \mathfrak{n}_-,\mathfrak{\t}\>, \]
             and $\dim\c = \dim\g$.  Then $\c\cap\h = \c\cap(0\oplus\g) = \mathfrak{n}_-$, exponentiating to $N_-=:K$.  Hence, $H/K = TN_+$, giving us the Dirac homogeneous space
                \[ (H/K,\P,L) = (TN_+, H\times_{N_-} \p,  H\times_{N_-}\l),\]
             for $\l = (\mathfrak{n}_+\oplus 0) + \mathfrak{t}_\triangle$, and $\p= \s^{-1}(\l)/\s^{-1}(\l)^\perp$.
        \item In fact, Evens and Lu -- together with work by Karolinsky in \cite{Karol} --  provided a classification of all Lagrangian subalgebras of $\g\oplus\bar{\g}$ in \cite{evenslu2}.  Given any such Lagrangian subalgebra $\l\subset \g\oplus\bar{\g}$, define
                \[\k:= \l\cap(0\oplus\g) \]
            and exponentiate it to the connected, closed Lie subgroup $K\subset H$.  This $\l$ then gives a Dirac homogeneous space on $H/K$.
        \item Let $\g$ be a real semi-simple Lie algebra with complexification $\g^\C$.  The Iwasawa decomposition gives us
            \[ \g = \mathfrak{u}\oplus\mathfrak{a}\oplus\mathfrak{n} \]
           where $\mathfrak{u}$ is a compact real form of $\g$, $\mathfrak{a} = i\mathfrak{t}$, and $\mathfrak{n}$ is nilpotent; define $\mathfrak{b}:=\mathfrak{a}\oplus\mathfrak{n}$.  Lu and Weinstein showed in \cite{luwein1} that the triple $(\g,\mathfrak{u},\mathfrak{b})$ forms a Manin triple, with the imaginary part of the Killing form as the metric.  Hence, $\mathfrak{b}$ is Lagrangian in $\g$, meaning $\mathfrak{t}\oplus\mathfrak{b}$ must be coisotropic in $\g$.  Defining
            \[ \c: = (\mathfrak{t}\oplus\mathfrak{b})\oplus(\mathfrak{t}\oplus\mathfrak{b}) \subset \overline{\g}\oplus\g \]
           we obtain a Dirac homogeneous space on $H/K$ with $K=\mathrm{exp}(0\oplus(\mathfrak{t}\oplus\mathfrak{b}))$.
    \end{enumerate}
\end{example}

\section{Coisotropic Dirac Lie Subgroups, Homogeneous Spaces}
Lu gave a method in \cite{luh} for producing Poisson homogeneous spaces from certain Poisson Lie subgroups.  Let $(H,\Pi_H)$ be a connected Poisson Lie group, and $(K,\Pi_K)$ a closed, connected Poisson Lie-subgroup with Lie algebra $\k$ such that $\ann(\k)\subseteq\h^*$ is a subalgebra\footnote{Such subgroups are referred to as ``coisotropic", perhaps a little confusingly.}. Then there exists a unique Poisson structure on $H/K$ making $\pi:\,H\to H/K$ a Poisson map, and the left action of $H$ on $H/K$ a Poisson action.  Indeed, this is the Poisson homogeneous space for $(H,\Pi_H)$ corresponding to the Lagrangian subalgebra $\l=\k\oplus\ann(\k)\subseteq\d$.\\
This construction generalises to the Dirac case as well.
\begin{dfn} Let $(H,\A,E)$ be a Dirac Lie group with associated Dirac Manin triple $(\d,\g,\h)_\beta$.  A \emph{Dirac Lie subgroup} of $(H,\A,E)$ is a Dirac Lie group $(K,\A',E')$ with associated Dirac Manin triple $(\d',\g',\k)_{\beta'}$, and a Dirac morphism (Chapter 4 in \cite{lmdir})
    \[ R:\, (\d',\g',\k)_{\beta'} \da (\d,\g,\h)_\beta \]
such that the associated group morphism $K\hookrightarrow H$ is an injective immersion.  $(K,\A',E')$ is called a \emph{closed Dirac Lie subgroup} if $K$ is closed in $H$. If $\k\subset\d'$ is $\beta'$-coisotropic, then we call $(K,\A',E')$ a \emph{coisotropic Dirac Lie subgroup}.
\end{dfn}

\begin{prop} Let $(K,\A',E')$ be a coisotropic, closed Dirac Lie subgroup of $(H,\A,E)$.  Then $(K,\A',E')$ generates a Dirac homogeneous space for $(H,\A,E)$ over $H/K$ associated to $\c=\k\oplus\ker\psi$, for $\psi:\,\g\to\g'$ coming from the Dirac morphism.
\begin{proof}
By Lemma 4.2 in \cite{lmdir}, the Dirac morphism $R:\,(\d',\g',\k)_{\beta'} \da (\d,\g,\h)_\beta$ is the direct sum of the graphs of $\phi:\,\k\hookrightarrow\h$, $\psi:\,\g\to\g'$:
    \begin{equation}\label{Rgraphs} R = \{ (\zeta + \phi(\kappa),\psi(\zeta)+\kappa)\,|\,\kappa\in\k,\,\zeta\in\g \}. \end{equation}
Since $R$ is a coisotropic relation of algebras, and $\k\subset\d'$ is a $\beta'$-coisotropic, it follows that the forwards image $\k\circ R$ is a $\beta$-coistropic subalgebra of $\d$.  By Equation \eqref{Rgraphs}, $\kappa \sim_R (\phi(\kappa)+\zeta)$ if and only if $\psi(\zeta) = 0$.  Since $\phi(\k)\cong\k$ as algebras, we have $\k\circ R = \k\oplus\ker\psi\subset\d=:\c$.  By Theorem \ref{dirachom}, this gives the $(H,\A,E)$-homogeneous space
    \[ (H/K,\P,L) = (H/K, H\times_K \p , H\times\k(\c/\k) ) \]
with $\p = \s^{-1}(\c/\k)/\s^{-1}(\c/\k)^\perp$.
\end{proof}
\end{prop}

\begin{rmrk} If $(H,\Pi_H)$ is a Poisson Lie group with corresponding Dirac Lie group $(H,\A,E)$, and $(K,\Pi_K)$ is a coisotropic, closed Poisson Lie subgroup with corresponding Dirac Lie subgroup $(K,\A',E')$, then we have
    \[ (\d,\g,\h) \cong (\h\bowtie\h^*,\h^*,\h),\quad (\d',\g',\k) \cong (\k\bowtie\k^*,\k^*,\k), \]
as well as $\psi = \phi^*$ by Remark 4.3 in \cite{lmdir}.  Since $\phi:\,\k\to\h$ is injective, $\psi:\,\h^*\to\k^*$ is surjective, and $\ker\psi = \ann(\k)$.  This gives $\c=\k\bowtie\ann(\k)$, as in \cite{luh}.
\end{rmrk}

\begin{example} Let $H$ be a Lie group, and $(H,\A,E)$ the standard Dirac Lie group structure, associated to $(\h\ltimes\h^*,\h^*,\h)_\beta$.  Let $K$ be any closed Lie subgroup of $H$, with $(K,\A',E')$ the standard Dirac Lie group structure associated to it, with triple $(\k\ltimes\k^*,\k^*,\k)_\beta$.  The inclusion map $K\hookrightarrow H$ gives rise to a Dirac morphism
    \[ (K,\A',E') \da (H,\A,E) \]
(by Proposition 2.10 in \cite{lmdir}).  As in the Poisson case, the resulting $\phi:\,\k\hookrightarrow\h$ is the inclusion, making $\psi$ surjective, yielding the $K$-invariant Lagrangian subalgebra
    \[ \c := \k\ltimes \ann(\k).\]

\end{example}

\section{General Dirac Actions}

In the general case of Dirac modules $(M,\P,L)$ for $(H,\A,E)$, where $M$ is simply an $H$-manifold, we no longer necessarily have trivialisations of $\P,L$.  The groupoid action $\A\act \P$ does still induce a group action defined by Equation \eqref{HactonPL}:
    \begin{align*}
        H\times\P &\to \P \\
        (h,y) &\mapsto (h,\u(y))\circ y.
    \end{align*}
We refer to this action as $\bullet$.  Once this group action (i.e.: the action of $E$ on $\P$) is defined, we can defined the full action of $\A$ on $\P$ by
    \[ (h,(\u(x),\alpha)) \circ x_m = h\bullet x_m + \u^*_{h.m}(\alpha);\]

This group action does not, in general, preserve the Courant structure of $\P$; it does, however, preserve the bracket structure of the $H$-invariant sections of $\P$.

\begin{prop}\label{genDirHinv}  Let $(M,\P,L)$ be a Dirac module for the Dirac Lie group $(H,\A,E)$.  For the $H$ action on $\P$ defined by Equation \eqref{HactonPL}, the $H$-invariant sections $\Gamma(\P)^H$ are closed under the Courant bracket.
\begin{proof}  Let Act denote the $CA$-morphism given by the groupoid action of $\A$ on $\P$, and let $\sigma_i\in \Gamma(\P)^H$; then by $H$-invariance, $(\eta_i,\sigma_i)\sim_{\mathrm{Act}}\sigma_i$ for $\eta_i\in \Gamma(E)$ composable with $\sigma_i$ along the graph of the $H$ action on $M$.  Hence, we have
    \[ \Cour{(\eta_1,\sigma_1),(\eta_2,\sigma_2)} \sim_{\mathrm{Act}} \Cour{\sigma_1,\sigma_2}. \]
For the first term, $\Cour{(\eta_1,\sigma_1),(\eta_2,\sigma_2)} = (\eta', \Cour{\sigma_1,\sigma_2})$, where
    \[ \eta' = [\eta_1,\eta_2] + \L_{\a(\sigma_1)}\eta_2 - \L_{\a(\sigma_2)}\eta_1. \]
Hence, $(\eta', \Cour{\sigma_1,\sigma_2})\sim_{\mathrm{Act}} \Cour{\sigma_1,\sigma_2}$, meaning $\Cour{\sigma_1,\sigma_2}\in \Gamma(\P)^H$.
\end{proof}
\end{prop}
In certain cases, the $H$ action from Equation \eqref{HactonPL} allows us to generate new Manin Pairs.

\begin{prop}  Let $(M,\P,L)$ be a Dirac module for the Dirac Lie group $(H,\A,E)$ such that $M$ is a principal $H$-bundle.  Then
    \[ (\P/H,L/H) \to M/H \]
is a Manin pair.
\begin{proof}  The metric, being multiplicative, is $H$-invariant, and thus descends to the quotient $\P/H$.  By Proposition \ref{genDirHinv}, the $H$-invariant sections of $\P$ are closed under the Courant bracket, and hence the bracket descends to $\P/H$.  The anchor map $\a:\,\P\to \mathrm{T}M$ satisfies
    \[ \a_\P(\xi\circ x) = \a_\A(\xi)\cdot\a_\P(x) \]
for composable $\xi\in\A_h,\,x\in \P_m$.  Let $(h,\zeta)=\a_\A(h,\u(x))$ in left-trivialisation of $\mathrm{T}H$; we have
    \begin{align*} \a_\P(h\bullet x) &= (h,\zeta)\cdot\a_\P(x_m)\\
                                &= (e,\Ad_h\zeta)\cdot(h,0)\cdot \a_\P(x_m)\\
                                &= h\cdot \a_\P(x_m) + (\Ad_h\zeta)_M|_{h.m}
    \end{align*}
where in the first term, the dot denotes the tangent lift of the action of $H$ on $M$.  Since the second term is tangent to the $H$-orbit directions, we conclude that the anchor descends to the quotient.
\end{proof}
\end{prop}

\addcontentsline{toc}{chapter}{Bibliography}



\end{document}